\documentclass[11pt,a4paper,reqno]{amsart}
\usepackage{amsbsy,amsfonts,amsmath,amssymb,amscd,amsthm,mathtools,bm}
\usepackage{mathrsfs,float,color}
\usepackage[mathcal]{euscript}
\usepackage{subfigure,enumitem,graphicx,epstopdf}
\usepackage[a4paper,text={6.1in,8in},centering,top=1.5in]{geometry}
\usepackage{pxfonts}
\usepackage[colorlinks=true,linkcolor=blue,citecolor=green]{hyperref}
\usepackage{srcltx}
\usepackage[margin=20pt,font=small,labelfont=bf,textfont=it]{caption}
\usepackage[latin1]{inputenc}
\numberwithin{equation}{section}

\theoremstyle{plain}
\newtheorem{theorem}{Theorem}[section]
\newtheorem{lemma}[theorem]{Lemma}
\newtheorem{proposition}[theorem]{Proposition}
\newtheorem{corollary}[theorem]{Corollary}
\newtheorem{claim}[theorem]{Claim}

\theoremstyle{definition}
\newtheorem{definition}[theorem]{Definition}
\newtheorem{remark}[theorem]{Remark}

\hypersetup{linktocpage}

\small\normalsize
\theoremstyle{plain}
\newtheorem{mainthm}{Theorem}
\newtheorem{dualmainthm}{Theorem}
\newtheorem{mainprop}{Proposition}

\newtheorem{maincor}[mainprop]{Corollary}
\newtheorem{dualmaincor}{Corollary}
\newtheorem{thm}{Theorem}[section]
\newtheorem{cor}[thm]{Corollary}
\newtheorem{lem}[thm]{Lemma}
\newtheorem{prop}[thm]{Proposition}

\newtheorem{defi}{Definition}
\newtheorem{dualdefi}{Definition}

\theoremstyle{definition}

\newcommand{\Diff}{\mathrm{Diff}}

\newcommand{\hatF}{\hat{F}}

\newcommand{\supp}{\mathrm{supp}}

\makeatletter
\def\l@part{\@tocline{0}{-2pt}{1pc}{}{}}
\def\l@section{\@tocline{1}{-2pt}{1pc}{4.6em}{}}
\renewcommand{\tocpart}[3]{%
  \indentlabel{\@ifnotempty{#2}{\makebox[2.3em][l]{%
    \ignorespaces#1 #2.\hfill}}}\bf{#3}}
\renewcommand{\tocsection}[3]{%
  \indentlabel{\@ifnotempty{#2}{\hspace*{2.3em}\makebox[2.3em][l]{%
    \ignorespaces#1 #2.\hfill}}}#3}

\makeatother

\makeatletter
\newcommand{\exterior}[1]{\mathop{\mathpalette\exterior@{#1}}}
\newcommand{\exterior@}[2]{%
  \raisebox{\depth}{%
  \fontsize{\sf@size}{0}%
  \m@th
  $\ifx#1\displaystyle\textstyle\else#1\fi\bigwedge$}%
  ^{\mspace{-2mu}#2}%
  \kern-\scriptspace
}
\makeatother

\title{(co)Quasi-irreducible and {(co)}expanding random maps}
\title{(co)Quasi-irreducible and {(co)}expanding random maps}

\author{Pablo G.~Barrientos}
\address{
Instituto de Matem\'atica e Estat\'istica,
Universidade Federal Fluminense,
Niter\'oi, Brazil
}
\email{pgbarrientos@id.uff.br}

\author{Isaia Nisoli}
\address{
Instituto de Matem\'atica,
Universidade Federal do Rio de Janeiro,
Av.\ Athos da Silveira Ramos 149,
Rio de Janeiro, RJ 21941-909, Brazil
}
\email{nisoli@im.ufrj.br}

\author{Dominique Malicet}
\address{
LAMA,
Universit\'e Gustave Eiffel,
5 boulevard Descartes,
77420 Champs-sur-Marne, France
}
\email{mdominique@crans.org}

\date{}
\date{}

\begin{document}
\maketitle

\begin{abstract}
We study quasi-irreducibility, expansion, and their cotangent duals for random \(C^1\)~(local) diffeomorphisms on compact invariant sets of a Riemannian manifold. Quasi-irreducibility is defined through stationary lifts to the projective tangent bundle: every lift must realize the Furstenberg--Kifer formula for the top Lyapunov exponent. We prove that, for ergodic stationary measures, this is equivalent to the absence of equators, namely non-random invariant subbundles on which the top exponent drops. Under simplicity of the first Lyapunov exponent for every stationary measure, quasi-irreducibility is also equivalent to vertical mostly contraction, contraction on average, a vertical spectral gap, and uniqueness of stationary projective lifts. We then apply these criteria to continuity of the top Lyapunov exponent and to expansion on average. In particular, expansion is characterized by positivity of the top Lyapunov exponent on all non-random invariant subbundles; under quasi-irreducibility, it is equivalent to positivity of the top Lyapunov exponent for every stationary measure. Dual statements hold for coquasi-irreducibility, coequators, coexpansion, and continuity of the bottom Lyapunov exponent. We also describe the interplay between expansion and coexpansion and extend the formalism to Grassmannian bundles, obtaining higher-dimensional versions controlling intermediate sums of Lyapunov exponents.

\end{abstract}

\thispagestyle{empty}

\section{Introduction}

Random dynamical systems generated by diffeomorphisms have become a natural
setting in which to study rigidity, orbit closures, stationary measures and
statistical properties beyond the classical deterministic theory. In the
homogeneous setting, the works of Furstenberg~\cite{Fur:63}, Benoist-Quint~\cite{BQ11},
Bourgain-Furman-Lindenstrauss-Mozes~\cite{BFLM07} and Eskin--Lindenstrauss~\cite{EL} show that stationary measures and orbit closures often satisfy
strong rigidity phenomena under irreducibility or expansion assumptions. A central problem is to understand which
analogues of these mechanisms survive for random walks by general smooth
diffeomorphisms.

For smooth dynamics on surfaces, Brown and Rodriguez-Hertz~\cite{BR17} proved a fundamental
trichotomy for hyperbolic stationary measures: either a stable distribution is
non-random (a non-trivial equator), or the measure is SRB, or the measure is finite. For non-uniform hyperbolic random maps of conservative surfaces diffeomorphisms, the absence of non-trivial equators is equivalent to the notions of expansion and coexpansion on average. Recently, expansion and coexpansion on average have appeared as robust
hypotheses~\cite{Pot22,Ell23} leading to absolute continuity~\cite{BLOPR25}, ergodicity~\cite{DDZ25}, spectral gaps~\cite{DD25}, 
exponential mixing~\cite{DD24}, rigidity~\cite{BEFR25} and equidistribution~\cite{BZ26} among others. These results show that
expansion and coexpansion on average conditions are powerful, but they also raise a basic
structural question: which projective or cotangent obstructions prevent expansion and coexpansion on average from being equivalent to the positivity or negativity of the corresponding extremal Lyapunov exponents?

The purpose of this paper is to isolate this obstruction at the level of the derivative cocycle. We introduce and study quasi-irreducibility for random \(C^1\) local diffeomorphisms in terms of stationary measures on the projective tangent bundle generalizing previous results in~\cite{BM} for linear cocycles. Namely, a stationary measure is quasi-irreducible when every stationary lift to the projective bundle maximizes the Furstenberg-Kifer variational formula for the top Lyapunov exponent. Dually, we define
coquasi-irreducibility using the inverse-dual cotangent cocycle. These notions include the usual irreducibility assumptions for random
matrix products.

Our first main result characterizes quasi-irreducibility by the absence of
equators, that is, invariant subbundles on which the top Lyapunov exponent drops. The dual statement characterizes coquasi-irreducibility by the absence of
coequators. Under a uniform gap between the first two Lyapunov exponents, we
prove that quasi-irreducibility is equivalent to several metric and spectral
properties of the projective dynamics: vertical mostly contracting, contraction
on average along projective fibers, a vertical spectral gap on anisotropic
H\"older spaces, and uniqueness of stationary projective lifts. The cotangent
analogue gives the corresponding criteria for coquasi-irreducibility.

We then apply this framework to continuity and expansion. Quasi-irreducibility
gives continuity of the top Lyapunov exponent, while coquasi-irreducibility gives
continuity of the bottom exponent. We also introduce expanding and coexpanding
on average through lower variational formulas on the projective tangent and
cotangent bundles. This yields criteria showing when expansion forces
quasi-irreducibility, when coexpansion forces coquasi-irreducibility, and how the
two notions interact. In dimension two, equators and coequators have the same
dimension, and the resulting statements give a particularly transparent
equivalence between expansion, coexpansion, quasi-irreducibility and
coquasi-irreducibility. 

Finally, we explain that the same formalism coming from~\cite{BN26} applies to Grassmannian bundles and
exterior powers. This gives higher-dimensional versions of
quasi-irreducibility, expansion and coexpansion controlling intermediate sums of
Lyapunov exponents. Thus the paper provides a unified Lyapunov-theoretic
language for the projective, cotangent and Grassmannian mechanisms underlying
recent rigidity and statistical results for random smooth dynamics.

\subsection{Setting}
We establish the framework for the Lyapunov exponents of random maps on a Riemannian manifold $M$ of dimension $m$.  Let $(T , \mathscr{A}, p)$ be a probability space, and consider the Bernoulli product space $(\Omega , \mathscr{F}, \mathbb{P}) = (T^\mathbb{N}, \mathscr{A}^\mathbb{N}, p^\mathbb{N})$. 
Let $f: \Omega \times M \to M$ be a $C^1$ random map, i.e.  $f$ is a measurable map such that $
f_\omega := f(\omega, \cdot) = f_{\omega_0}$ is a local $C^1$  diffeomorphism for $\mathbb{P}$-a.e.~$\omega = (\omega_i)_{i \geq 0} \in \Omega$. 
We introduce the following non-autonomous iterations:
$$
f^0_\omega = \mathrm{id} \quad \text{and} \quad f^n_\omega = f^{}_{\omega_{n-1}} \circ \dots \circ f^{}_{\omega_0} \quad \text{for} \ \ n > 0 \ \ \text{and $\mathbb{P}$-a.e.~$\omega = (\omega_i)_{i \geq 0} \in \Omega$}.
$$
Let $X$ be a compact $f$-invariant set, i.e., $f_\omega(X) \subset X$ for $\mathbb{P}$-a.e.~$\omega\in\Omega$. A typical $f$-invariant set  is $X=\supp \mu$ where $\mu$ is an $f$-stationary measure c.f.~\cite[Claim~3.4.7]{BM}. Recall that $\mu$ is said to be $f$-stationary measure if $\mu=\int f_\omega \mu \,d\mathbb{P}$. This formalism also includes the study of \emph{random linear cocycles}  when $\mu=\delta_p$ is $f$-stationary, i.e., by setting $X=\{p\}$ with $f_\omega(p)=p$  for $\mathbb{P}$-a.e.~$\omega\in \Omega$ and the random matrices $A(\omega)=Df_\omega(p)$.  To simplify notation, in what follows $f$ denotes the random map $f|_X$, i.e., restricting its domain to $\Omega \times X$.

\subsection{Maximal Lyapunov exponents} \label{ss:maximalLE}
Under the integrability condition
\begin{equation} \label{eq:log-integrability}
\int \big( \log^+  \|Df_\omega\|_\infty + \log^+  \|(Df_\omega)^{-1}\|_\infty \big) \, d\mathbb{P} < \infty,
\end{equation}
where $\|\cdot\|_\infty$ denotes the sup-norm (on $X$), 
we define the \emph{Lyapunov exponents} $\lambda_1(\mu)\geq \dots \geq \lambda_m(\mu)$ of any $f$-stationary measure $\mu$ inductively for $k=1,\dots,m$ by
\[
\sum_{i=1}^k \lambda_i(\mu) :=   \lim_{n\to\infty} \frac{1}{n} \int \log \|\wedge^k Df^n_\omega(x)\| \, d\mathbb{P}\, d\mu = \inf_{n\geq 1} \frac{1}{n} \int \log \|\wedge^k Df^n_\omega(x)\| \, d\mathbb{P}\, d\mu.
\]
Here, $\|\wedge^k Df^n_\omega(x)\|$ denotes the operator norm of the $k$-th exterior power of the differential $Df^n_\omega(x)$ defined as the supremum of $\|Df^n_\omega(x) v_1 \wedge \dots \wedge Df^n_\omega(x) v_k\|$ over all $k$-vectors $v_1 \wedge \dots \wedge v_k$ of unit norm in the exterior power space $\wedge^k T_x M$.
We denote by $\mathcal{I}(f)$  the set of  $f$-stationary  measures (supported on~$X$).  We define the \emph{maximal  Lyapunov exponent of $f$} as
$$\lambda_1(f) := \sup \left\{ \lambda_1(\mu) : \mu \in \mathcal{I}(f) \right\}.$$ 
According to~\cite[Thm.~3.2.1]{BM}, this exponent can be characterized as 
\begin{equation} \label{eq:lambdaf}
\begin{aligned}
    \lambda_1(f) =& \max \{\lambda_1(\mu) : \mu \in \mathcal{I}(f) \ \text{ergodic } \} =  \sup_{x\in X} \lambda_1(x)\\ 
    =&  \lim_{n\to\infty} \max_{x \in X} \frac{1}{n}  \int\log\|{Df^n_\omega(x)}\| \, d \mathbb{P} = \inf_{n\geq 1} \max_{x \in X} \frac{1}{n}  \int \log\|{Df^n_\omega(x)}\| \, d \mathbb{P}
\end{aligned}
\end{equation}
where 
\[
\lambda_1(x):= \limsup_{n\to\infty} \frac{1}{n}\int \log \|Df^n_\omega(x)\|\, d\mathbb{P}. 
\]
Moreover, the limsup in the definition of the annealed Lyapunov exponent $\lambda_1(x)$ exists and  coincides with the pointwise Lyapunov exponents almost surely with respect to any $f$-stationary measure $\mu$, that is, 
$$
\lambda_1(x)=  \lim_{n\to\infty} \frac{1}{n}\int \log \|Df^n_\omega(x)\|\, d\mathbb{P} = \lim_{n\to\infty} \frac{1}{n} \log \|Df^n_{\tilde\omega}(x)\| \quad \text{for $(\mathbb{P}\times \mu)$-a.e.~$(\tilde{\omega},x)$}.
$$
The last characterization of $\lambda_1(f)$ in~\eqref{eq:lambdaf} as the infimum of the sequence of continuous functions 
$b_n(f)= \max_{x \in X} \frac{1}{n}  \int \log\|{Df^n_\omega(x)}\| \, d \mathbb{P}$  
implies that $\lambda_1(f)$ is upper semicontinuous. One of the goals of this paper is to study the continuity of this map.


We denote by $\hat X=P(T_XM)$  
the projective tangent bundle of $M$ at $X$. An element of $\hat X$ is a pair $(x, \hat{v})$, where $x \in X$ and $\hat{v}$ is the projective class of a unit vector $v \in T_xM$, i.e., the one-dimensional vector space spanned by $v$. To study directional growth, we lift $f$ to a random map $\skew{6}{\hat}{f}: \Omega \times \hat  X \to \hat  X$ defined by
\[
\hat{f}_\omega(x,\hat{v}) = \left( f_\omega(x), \, Df_\omega(x)\hat{v} \right).
\]
We define the \emph{annealed growth rate function} 
    \[
    \phi_f(x, \hat{v}) := \int \log \|{Df_\omega(x)v} \| \, d \mathbb{P}, \quad (x,\hat v ) \in \hat X
    \]
    where $v$ is any unit vector in the class $\hat{v}$, and the \emph{annealed Lyapunov exponent} of $\skew{6}{\hat}{f}$ as 
    \[
    \lambda_1(x, \hat{v}) := \limsup_{n\to\infty} \frac{1}{n} \int \log \|{Df^n_\omega(x)v}\| \, d \mathbb{P}.
    \]
The sequence of projective continuous maps $\{\phi_n\}_{n\geq 1}$ given by $\phi_{n}(x,\hat v)=\int \log \|Df^n_\omega(x)v\|\, d\mathbb{P}$ is $\hat P$-additive, where $\hat P$ is the annealed Koopman operator associated with the random maps $\skew{6}{\hat}{f}$, i.e.,
$\hat P\varphi
  :=
  \int \varphi \circ \skew{6}{\hat}{f}_\omega \,d\mathbb P$ for $\varphi\in C(\hat X)$.
Hence, by Kingman's subadditive ergodic theorem (c.f.~\cite[Thm.~D and Cor.~2.2.1]{BM}),  the following holds for any $\skew{6}{\hat}{f}$-stationary measure $\nu$:
\begin{enumerate}[label=(\roman*)]
\item $\displaystyle \int \lambda_1\, d\nu= \int \phi_f \, d\nu
        = \lim_{n\to \infty}
        \frac1n
        \int
        \log\|Df_\omega^n(x)v\|
        \,d\mathbb P d\nu$, 
        
\item $\displaystyle \lambda_1(x,\hat v)=\lim_{n\to\infty} \frac{1}{n} \int \log \|Df^n_\omega(x)v\| \, d\mathbb{P}$  for $\nu$-a.e.~$(x,\hat v)\in \hat X$, \\[-0.2cm]
\item $\lambda_1(x,\hat v)=\int \phi_f \, d\nu$ for $\nu$-a.e.~$(x,\hat v)\in \hat X$, provided $\nu$ is ergodic. \\
\end{enumerate}

Denote by $\mathcal{I}_\mu(\skew{6}{\hat}{f}\,)$ the set of $\skew{6}{\hat}{f}$-stationary measures (supported on $\hat X$) with first marginal a measure $\mu$.  Notice that this set is non-empty only if $\mu$ is an $f$-stationary measure. According to~\cite[Prop.~IV, Cor.~V, and Cor.~C.1]{BN26}, under the integrability condition~\eqref{eq:log-integrability} we have
\begin{enumerate}[label=(\alph*)]
    \item \emph{Furstenberg-Kifer variational formula:} for any  $\mu\in \mathcal{I}(f)$, 
    \[
    \lambda_1(\mu) = \sup\left\{\int\phi_f \, d\nu : \nu \in \mathcal{I}_\mu(\skew{6}{\hat}{f}\,) \, \right\} = \max\left\{\int\phi_f \, d\nu : \nu \in \mathcal{I}_\mu(\skew{6}{\hat}{f}\,) \, \text{ergodic} \right\}.
    \]
     Consequently,
        $$\lambda_1(f) = \sup \left\{\int \phi_f \, d\nu : \nu \in \mathcal{I}(\skew{6}{\hat}{f}\,)\right\} = \max \left\{\int \phi_f \, d\nu : \nu \in \mathcal{I}(\skew{6}{\hat}{f}\,) \, \text{ergodic}\right\}.
     $$
    \item \emph{Uniform limit characterization:} 
    \begin{align*}
    \lambda_1(\mu) &= \lim_{n\to\infty} \frac{1}{n}\int \big(\max_{\hat v\in P(T_xM)} \int \log \|Df^n_\omega(x)v\|\, d\mathbb{P}\big)\,d\mu \\[0.25cm]
    \lambda_1(f) &= \lim_{n\to\infty}  \max_{(x,\hat{v}) \in\hat X} \frac{1}{n} \int \log \|{Df^n_\omega(x)v}\| \, d \mathbb{P} = \sup_{(x,\hat{v}) \in \hat X} \lambda_1(x, \hat{v}) 
    \end{align*}
    Moreover, the  "$\displaystyle\lim_{n\to\infty}$" can be replaced by "$\displaystyle\inf_{n\geq 1}$".
\end{enumerate}

\subsection{Quasi-irreducibility}

The following definition is motivated by the analogous concept of random linear cocycles, which extend the classical notion of irreducibility (see~\cite[Prop.~5.3.2]{BM}).

\begin{defi}
\label{def:proj-QI}
We say that $\mu\in \mathcal{I}(f)$ is \emph{quasi-irreducible} if 
every $\nu \in \mathcal{I}_\mu(\skew{6}{\hat}{f}\,)$ is maximizing, i.e., $\int\phi_f\,d\nu=\lambda_1(\mu)$. 
We also say that $f$ is \emph{quasi-irreducible} if every  $\mu \in \mathcal{I}(f)$ is quasi-irreducible. Finally, $f$ is \emph{maximal  quasi-irreducible} if it is quasi-irreducible and $\lambda_1(\mu)=\lambda_1(f)$ for all $\mu\in \mathcal{I}(f)$. 
\end{defi}


As a consequence of the Furstenberg-Kifer variational principle, we get that $\mu$ is  quasi-irreducible whenever $\mathcal{I}_\mu(\skew{6}{\hat}{f}\,)$ is a singleton. In particular, $f$ is maximal  quasi-irreducible provided $\skew{6}{\hat}{f}$ is \emph{uniquely ergodic}, that is, if there is only one $\skew{6}{\hat}{f}$-stationary measure (supported on $\hat X$), which also implies that $f$ has only one stationary measure (supported on $X$). Another example of a maximal quasi-irreducible random map \(f\) which may have
several \(\skew{6}{\hat}{f}\)-stationary measures occurs when $f$ is uniquely ergodic (on $X$) and 
quasi-irreducible.

The following theorem characterises this notion in terms of the non-existence of equators. First, a vector bundle $\mathcal{L}$ of $T_XM$ is said to be  \emph{$Df$-invariant over $\mu$} if $Df_\omega(x)\mathcal{L}_x =\mathcal{L}_{f_\omega(x)}$ for $(\mathbb{P}\times \mu)$-a.e.~$(\omega,x)\in \Omega \times X$. We also say that $\mathcal{L}$ is \emph{trivial} if $\dim \mathcal L_x=0$, \emph{non-trivial} otherwise and \emph{proper} if $1\leq \dim \mathcal L_x <\dim M$ for $\mu$-a.e.~$x\in X$. 

\begin{mainthm}
\label{thm:mainB}
Let $f$ be the restriction to an invariant compact set $X$ of a $C^1$ random map satisfying the integrability condition~\eqref{eq:log-integrability} and consider $\mu\in \mathcal{I}(f)$ ergodic. Then, $\mu$ is  quasi-irreducible if and only if there is no proper $Df$-invariant vector bundle $\mathcal{L}$ over $\mu$ such that $\lambda_1(\mu|_{\mathcal{L}}) < \lambda_1(\mu)$ where 
$$
\lambda_1(\mu|_{\mathcal{L}}):=\lim_{n\to\infty} \frac{1}{n}\int \log \|Df^n_\omega(x)|_{\mathcal{L}_x}\|\, d\mathbb{P}d\mu.
$$
Moreover, if $f$ is maximal quasi-irreducible, then it satisfies the \emph{Strong Law of Large Numbers:} for every $x \in X$ and any unit vector $v\in T_xM$, it holds
    \[
    \lambda_1(f)= \lim_{n\to \infty}  \frac{1}{n}\int \log \|{Df^n_\omega(x)v}\| \, d \mathbb{P} =\lim_{n\to\infty} \frac{1}{n} \log \|{Df^n_{\tilde{\omega}}(x)v}\| \quad \text{for $\mathbb P$-a.e.~$\tilde{\omega}\in\Omega$.} 
    \] 
\end{mainthm}
    
Following the terminology of linear cocycles (see~\cite[Def.~1.10.1]{BM}), we introduce the notion of the equator of a measure:

\begin{defi} We call the \emph{equator} of an ergodic $f$-stationary measure $\mu$  the largest  $Df$-invariant vector bundle $\mathcal{L}$ of $T_XM$ over $\mu$ where the maximal Lyapunov exponent satisfies $\lambda_1(\mu|_{\mathcal{L}}) < \lambda_1(\mu)$. 
\end{defi} 
From the previous theorem, we get that  $f:\Omega \times X \to X$ is quasi-irreducible if and only if the equator of every ergodic $f$-stationary measure $\mu$ supported on $X$ is the trivial vector bundle over $\mu$.

\subsection{Vertical mostly contracting, contraction on average  and spectral gap} To express contraction along the projective fibres, fix the angular distance
$\theta_x$ on $P(T_xM)$ induced by the Riemannian structure of $T_xM$. We 
introduce the \emph{vertical metric} on $\hat X$ as
\begin{equation}\label{eq:def-d-vert}
  \mathrm d_{\mathrm{vert}}\big((x,\hat v),(y,\hat u)\big)
  :=
  \begin{cases}
    \theta_x(\hat v,\hat u) & \text{if } x=y,\\[0.3em]
    1                       & \text{if } x\neq y.
  \end{cases}
\end{equation}
For $r\in(0,1)$, the ball $B_{\mathrm{vert}}(z,r)$ is contained
in the fibre $\{x\}\times P(T_xM)$ where $z=(x,\hat v)$.
Moreover, the local Lipschitz constant of $\skew{6}{\hat}{f}^n_\omega$ at $z$ with
respect to $\mathrm d_{\mathrm{vert}}$ coincides with the norm of the vertical derivative of the projective random map, i.e.,
\[
  L\skew{6}{\hat}{f}^n_\omega(z):=\limsup_{\substack{(z',z'')\to (z,z) \\ z'\not = z''}}
  \frac{\mathrm d_{\mathrm{vert}}(\skew{6}{\hat}{f}^n_\omega(z'),\skew{6}{\hat}{f}^n_\omega(z''))}
       {\mathrm d_{\mathrm{vert}}(z',z'')} =\bigl\|D^{\mathrm{vert}}\skew{6}{\hat}{f}^n_\omega(z)\bigr\|.
\]
Here $D^{\mathrm{vert}}\skew{6}{\hat}{f}^n_\omega(z)$ is the restriction of the differential of $\skew{6}{\hat}{f}$ to the \emph{vertical tangent space} at $z=(x,\hat v)$, 
\[
  T^{\mathrm{vert}}_{z}\hat X
  :=
  \ker(D\pi(z))
  \;\cong\;
  T_{\hat v}P(T_xM)
\]
where $\pi: P(TM) \to M$ is the canonical projection. 
Accordingly, we define the \emph{vertical maximal Lyapunov exponent} of an $\skew{6}{\hat}{f}$-stationary measure $\nu$ by
\[
  \lambda_{\mathrm{vert}}(\nu)
  \coloneqq
  \lim_{n\to\infty}
    \frac{1}{n}
    \int
      \log L\skew{6}{\hat}{f}_\omega^{n}(z)\,
        d\mathbb Pd\nu.
\]
Following the theory of random maps on metric spaces in~\cite{BM}, 
we say that~$\hat{f}:\Omega \times \hat X \to \hat X$~is
\begin{enumerate}[leftmargin=0.65cm]
    \item   \emph{vertical mostly contracting} if the integrability condition~\eqref{eq:log-integrability} holds and
    \[
    \lambda_{\mathrm{vert}}(\hat{f}) \coloneqq \sup \left\{ \lambda_{\mathrm{vert}}(\nu) : \nu \in \mathcal{I}(\hat{f}) \right\} < 0.
    \]
    \item \emph{vertical contracting on average} if there are $0 < \alpha \leq 1$, $q < 1$, and  $n \ge 1$ such that
    \[
    \int \mathrm{d}_{\mathrm{vert}}(\skew{6}{\hat}{f}^n_\omega(z), \skew{6}{\hat}{f}^n_\omega(z'))^\alpha \, d\mathbb{P} \leq q \, \mathrm{d}_{\mathrm{vert}}(z, z')^\alpha 
    \]
    for all $z, z' \in \hat{X}$  lying in the same fiber, i.e.~$\pi(z)=\pi(z')$.
\end{enumerate}

In Lemma~\ref{lem:projective-vertical-regularity} we prove that the map $\nu  \mapsto \lambda_{\rm vert}(\nu)$ is upper semicontinuous on $\mathcal{I}(\skew{6}{\hat}{f}\,)$ with respect to the weak$^*$. Hence, it follows that \(\lambda_{\rm vert}\) attains its maximum on the compact convex
set \(\mathcal{I}(\skew{6}{\hat}{f}\,)\). That is, there exists
\(\hat\mu_*\in\mathcal{I}(\skew{6}{\hat}{f}\,)\) ergodic such that $      \lambda_{\rm vert}(\skew{6}{\hat}{f}\,)   =
        \lambda_{\rm vert}(\hat\mu_*)$. 
In particular, we have that $\skew{6}{\hat}{f}$ is vertical mostly contracting if and only if $\lambda_{\rm vert}(\hat\mu)<0$
        for every $\hat\mu\in\mathcal{I}(\skew{6}{\hat}{f}\,)$.

As mentioned, the projective random map $f_A: \Omega \times P(\mathbb{R}^m) \to P(\mathbb{R}^m)$ associated with a linear cocycle $A:\Omega \to \mathrm{GL}(m)$ can be identified with the lift $\skew{6}{\hat}{f}$ to $\hat X = P(T_XM)$ where $X= \{p\}$ (i.e, $\mu=\delta_p$ is $f$-stationary). In that case, $\skew{6}{\hat}{f}$ is vertical contracting on average if and only if $f_A$ is (global) contracting on average in the classical sense, c.f.~\cite{BM}.  Also, under the assumption of a gap between the first $\lambda_1(A)=\lambda_1(\delta_p)$ and second $\lambda_2(A)=\lambda_2(\delta_p)$ Lyapunov exponents, quasi-irreducibility is characterised by the unique ergodicity of the projective random map $f_A$,  c.f.~\cite[Prop.~VI]{BM}. Furthermore, this property is equivalent to $f_A$ being mostly contracting, and it is also equivalent to $f_A$ being contracting on average. 
The following theorem establishes that these metric properties are also equivalent in our more general context of non-linear random maps. We also get a description of the spectral properties of the Koopman operator $\hat P$ of $\skew{6}{\hat}{f}$ acting on the following vertical quotient space of the anisotropic  space of $\alpha$-H\"older functions:

Let ${\mathscr C}^{\alpha}(\hat X)$ be the Banach space of bounded Borel functions \(\Phi:\hat X\to\mathbb R\) whose restrictions
to the projective fibres are uniformly \(\alpha\)-H\"older. More precisely,
\[
   [\Phi]_{\alpha,\mathrm{vert}}
   :=
   \sup_{x\in X}
   \sup_{\substack{\hat u,\hat v\in P(T_xM)\\ \hat u\ne\hat v}}
   \frac{|\Phi(x,\hat u)-\Phi(x,\hat v)|}
        {d_{\mathrm{vert}}((x,\hat u),(x,\hat v))^\alpha}
   <\infty  \quad \text{and} \quad \|\Phi\|:=\|\Phi\|_{\infty}+[\Phi]_{\alpha,\mathrm{vert}}
\]
Let $\mathcal B_\pi
   :=
   \pi^*B(X)
   =
   \{\varphi\circ\pi:\varphi\in B(X)\}$ where $B(X)$ denotes the Banach space of real-valued bounded measurable functions with the sup-norm $\|\cdot\|_{\infty}$.  Then \(\mathcal B_\pi\) is precisely the kernel of the seminorm
\([\cdot]_{\alpha,\mathrm{vert}}\). Hence the quotient ${\mathscr C}^{\alpha}(\hat X)/\mathcal B_\pi$ is naturally endowed with the norm  $\|[\Phi]\|_{\alpha,\mathrm{vert}}
   :=
   [\Phi]_{\alpha,\mathrm{vert}}$. Moreover, since the Koopman operator $\hat P$ associated with $\skew{6}{\hat}{f}$ preserves \(\mathcal B_\pi\) because of
$\hat P(\psi\circ\pi)
   =
   (P\psi)\circ\pi$, $\hat P$  induces an operator on the quotient, 
\[
   {\hat P}:
   {\mathscr C}^{\alpha}(\hat X)/\mathcal B_\pi
   \longrightarrow
   {\mathscr C}^{\alpha}(\hat X)/\mathcal B_\pi,
   \qquad
   {\hat P}[\Phi]_{}:=[\hat P\Phi].
\]
We say that \(\hat P\) has a \emph{vertical spectral gap on ${\mathscr C}^{\alpha}(\hat X)$} if the induced operator on the vertical quotient spaces has spectral radius strictly less than 1, or equivalently, there exist \(C>0\) and \(\theta<1\) such that
\[
   \|{\hat P}^{\,n}[\Phi]\|_{\alpha,\mathrm{vert}}
   \le
   C\theta^n\|[\Phi]\|_{\alpha,\mathrm{vert}}
   \qquad
   \text{for every }n\ge0.
\]

\begin{mainthm} \label{mainthm:equivalence-spectral-gap} Let $f$ be the restriction to an invariant compact set $X$ of a $C^1$ random map  satisfying for some $\beta>0$ the integrability conditions 
\begin{equation} \label{eq:moment-condition-thmB}
    \int \log^+\|Df_\omega\|_\infty + \log^+\|(Df_\omega)^{-1}\|_\infty \,d\mathbb P<\infty
\quad \text{and} \quad 
\int \sup_{z\in\hat X} \bigl(L\skew{6}{\hat}{f}_\omega(z)\bigr)^\beta \,d\mathbb P<\infty .
\end{equation}
Assume that $\lambda_1(\mu) >\lambda_2(\mu)$ for every $\mu \in\mathcal{I}(f)$. Then, the following are equivalent:
    \begin{enumerate}
        \item $f$ is  quasi-irreducible;
        \item $\skew{6}{\hat}{f}$ is vertical mostly contracting;
        \item $\skew{6}{\hat}{f}$ is vertical contracting on average; 
        \item $\hat P$ has vertical spectral gap on ${\mathscr C}^{\alpha}(\hat X)$ 
        for any $\alpha>0$ small enough; 
        \item $\mathcal{I}_\mu(\skew{6}{\hat}{f}\,)$ is a singleton for every  $\mu\in \mathcal{I}(f)$.
     \end{enumerate}
\end{mainthm}

We will show that vertical mostly contracting implies the simplicity of the first Lyapunov exponent. Thus, from the above theorem, we also conclude the following result: 

\begin{maincor} \label{maincor:equivalence-mostly} Let $f$ be the restriction to an invariant compact set $X$ of a $C^1$ random map  satisfying the integrability conditions~\eqref{eq:moment-condition-thmB}. Then, $\skew{6}{\hat}{f}$ is vertical mostly contracting  if and only if $f$ is quasi-irreducible and $\lambda_1(\mu)>\lambda_2(\mu)$ for every $\mu \in \mathcal{I} (f)$.     
\end{maincor}

\subsection{Continuity} Let $X$ be a compact subset of $M$. We denote by $\Diff^1_{\rm loc, \mathbb P}(X)$ the set of $C^1$ random maps $f:\Omega\times M \to M$ such that $f_\omega(X)\subset X$ for $\mathbb{P}$-a.e.~$\omega\in\Omega$ and  satisfying~\eqref{eq:log-integrability}. 
Also, for a fixed Borel probability measure $\mu$ on $M$, we consider the subspace
\[
   \Diff^1_{\rm loc, \mathbb P}(X,\mu)
   :=
   \bigl\{\,f\in \Diff^1_{\rm loc,\mathbb P}(X):\ \mu\in\mathcal{I}(f)\,\bigr\}.
\]
We endow $\Diff^1_{\rm loc, \mathbb P}(X)$  with the \emph{average $C^1$ distance}
\begin{equation}\label{eq:avg-C1-main}
  d_{C^1}(f,g)
  :=
  \int \Big(
d_{C^0}(f_\omega,g_\omega)
+
\|Df_\omega-Dg_\omega\|_{\infty} \
+
\|(Df_\omega)^{-1}-(Dg_\omega)^{-1}\|_{\infty} \Big) 
\,d\mathbb{P}
\end{equation}
where as usual we are seeing $f$ and $g$  restricted to $X$ (i.e., the corresponding suprema in the above distances are taken over $X$). Finally, we extend the result of continuity of the maximal Lyapunov exponent for quasi-irreducible linear cocycles~\cite[Prop.~XI, XIV]{BM}  to the non-linear context. 

\begin{mainthm}
\label{main:thmD}
Let $f \in \Diff^1_{\rm loc, \mathbb P}(X,\mu)$ be such that $\mu$ is  quasi-irreducible for $f$. Then the map
\[
   \lambda_{1}(\mu,\cdot): \bigl(\Diff^1_{\rm loc, \mathbb P}(X,\mu), d_{C^1}\bigr) \to \mathbb{R},
   \quad
   g \mapsto \lambda_{1}(\mu,g):= \lim_{n\to \infty} \frac{1}{n} \int \log \|Dg^n_\omega(x)\| \, d\mathbb{P}\, d\mu
\]
is continuous at $f$. Moreover, if $f$ is maximal  quasi-irreducible, the map
$$
  \lambda_1:\bigl(\Diff_{\rm loc, \mathbb P}^1(X),d_{C^1}\bigr)\to\mathbb R, \quad g \mapsto \lambda_1(g):=\sup\big\{\lambda_1(\eta): \eta \in \mathcal{I}(g)\big\}
$$
is continuous at $f$. 
\end{mainthm}

Notice that the above theorem only requires a logarithm-moment condition improving the continuity result for linear cocycle in~\cite[Prop.~XI, XIV]{BM} which assume an exponential moment condition. 
Moreover, although the above result is stated for random maps realized
on a fixed probability space $(\Omega,\mathscr F,\mathbb P)$, it has an
equivalent distributional formulation. Namely, by the usual coupling
realization argument, continuity with respect to the average $C^1$
distance corresponds to continuity of the one-step distribution laws
with respect to the associated Wasserstein distance; cf.\
\cite[Rem.~7.6.3]{BM}.

We also study the continuity of $\mu\mapsto \lambda_1(\mu)$ for a fixed random map. It is well known that this map is upper semicontinuous with respect to the weak$^*$ topology; see~\cite[Claim~2.2.7]{BM}. The following shows that under quasi-irreducibility, we can get the continuity: 
\begin{mainthm}
\label{prop:continuity-measure-lyapunov}
Let $f$ be the restriction to an invariant compact set $X$ of a $C^1$ random map satisfying the integrability conditions~\eqref{eq:log-integrability}.
 If \(\mu_0\in\mathcal{I}(f)\) is quasi-irreducible, then the map $\mu\in\mathcal{I}(f)\mapsto \lambda_1(\mu)$  is continuous at $\mu_0$ with respect to the weak$^*$ topology.  
In particular, if \(f\) is quasi-irreducible, then $       \mu \in \mathcal{I}(f) \mapsto \lambda_1(\mu)$ is continuous. 
\end{mainthm}

\subsection{Expanding on average} Given an $f$-stationary measure $\mu$, define 
$$
 \lambda_-(\mu):= \inf_{\hat \mu \in \mathcal{I}_\mu(\skew{6}{\hat}{f}\,)} \int \phi_f \, d\hat\mu  \quad \text{and} \quad \lambda_-(f):= \inf_{\mu\in \mathcal{I}(f)} \lambda_-(\mu). 
$$
We also recall the elementary inequality
\[
        \lambda_m(\mu)\leq \lambda_-(\mu)\leq \lambda_1(\mu).
\]
Indeed, the second inequality follows immediately from the Furstenberg-Kifer variational formula
since $\lambda_1(\mu) =
        \lambda_+(\mu)$ where 
        $$ \lambda_+(\mu):=
        \sup_{\hat\mu\in\mathcal{I}_\mu(\skew{6}{\hat}{f})}
        \int\phi_f\,d\hat\mu.$$
On the other hand, the first inequality follows by combining (i) in~\S\ref{ss:maximalLE} with $\nu \in \mathcal{I}_\mu(\skew{6}{\hat}{f}\,)$, the fact that 
$\|Df_\omega^n(x)v\| \geq
        \mathfrak m(Df_\omega^n(x))$ 
for every unit vector \(v\in T_xM\) where $        \mathfrak m(B):=\|B^{-1}\|^{-1}$ is the co-norm of a linear map $B$ and 
$$
        \lambda_m(\mu)=\lim_{n\to\infty}
        \frac1n
        \int
        \log\mathfrak m(Df_\omega^n(x))
        \,d\mathbb Pd\mu.  
$$        
Moreover, according to~\cite[Thm.~3.2.1 and Cor.~2.2.1]{BM}
and~\cite[Cor.~D.3 and~\S D.4]{BN26},  
\begin{align*}
    \lambda_-(\mu) &= \min\left\{ \int \phi_f \, d\hat\mu : \hat\mu \in \mathcal{I}_\mu(\skew{6}{\hat}{f}\,) \ \text{ergodic} \right\} \\ &= \lim_{n\to\infty} \int \big(\min_{\hat{v} \in P(T_xM)} \frac{1}{n}\int \log \|Df^n_\omega(x)v\|\, d\mathbb{P}\big)\,d\mu \\[0.25cm]
    \lambda_-(f) &=  \min \left\{ \int \phi_f \, d\hat\mu:  \hat \mu\in \mathcal{I}(\skew{6}{\hat}{f}\,) \ \text{ergodic}\right\} =\lim_{n\to\infty}  \min_{(x,\hat{v}) \in\hat X} \frac{1}{n} \int \log \|{Df^n_\omega(x)v}\| \, d \mathbb{P} \\ &=\inf_{(x,\hat v)\in \hat X} \liminf_{n\to \infty} \frac{1}{n}\int \log \|Df_\omega^n(x)v\|\, d\mathbb{P}.
    \end{align*}
 The "$\lim_{n\to\infty}$" above can be replaced by "$\sup_{n\geq 1}$".
In view of this, the following definition agrees with the usual finite-time formulation of the notion of expansion that appears in the literature under several different names: weakly expanding~\cite{Liu17}, uniformly expanding~\cite{Chu20,Pot22,BEFR25}, uniformly expanding on average in the future in~\cite{BLOPR25} and~expanding on average in~\cite{DD24}. 

\begin{defi}
We say that   $f$ is  \emph{expanding on average} (or \emph{expanding} for short) if $\lambda_-(f)>0$. 
\end{defi}

As a first immediate consequence of the representation of $\lambda_-(f)$ as the supremum of the sequence of continuous functions $a_n(f)=\min_{(x,\hat v)\in \hat X} \frac{1}{n} \int \log \|Df_\omega^n(x)v\|\, d\mathbb{P}$ we get that $f \mapsto \lambda_-(f)$ is lower-semicontinuous. Thus, the set of expanding on average random maps is an open set.   
From this definition, we have the following result:

\begin{maincor} \label{maincor:expanding-quasi} Let $f$ be the restriction to an invariant compact set $X$ of a $C^1$ random map satisfying the integrability condition~\eqref{eq:log-integrability}. The following statement holds: 
\begin{enumerate}[label=(\roman*),itemsep=0.1cm,leftmargin=0.75cm]
    \item 
$f$ is expanding if and only if 
every non-trivial $Df$-invariant bundle $\mathcal L$ over an $f$-stationary measure $\mu$ satisfies $\lambda_1(\mu|_{\mathcal L})>0$;
\item if $f$ is expanding and $\lambda_2(\mu)\leq 0$ for every $\mu \in \mathcal{I}(f)$, then $f$ is quasi-irreducible;
\item if $f$ is quasi-irreducible, then
$$\text{$f$ is expanding if and only if
$\lambda_1(\mu)>0$ for every $\mu \in \mathcal{I}(f)$;}$$
\item if $\lambda_2(\mu)\leq 0<\lambda_1(\mu)$ for all $\mu\in\mathcal{I}(f)$, then 
$$\text{$f$ is expanding if and only if it is quasi-irreducible}.$$ 
\end{enumerate}
\end{maincor}

Observe that by taking $\mathcal{L}=T_XM$ in item (i), we have $\lambda_1(\mu|_{\mathcal L})=\lambda_1(\mu)$. In fact, item~(i) in the above corollary was attributed in~\cite[Theorem~1.1]{Pot22} to Chung~\cite{Chu20} and Liu-Xiaochuan~\cite{Liu17}. However, these references only deal with the area-preserving case in dimension $m=2$.  


\subsection{Coexpanding on average and coquasi-irreducibility}
\label{subsec:coexpanding-on-average}

We now introduce the cotangent analogue of expanding on average. We will assume in the following that the $C^1$ random map $f$ satisfies the integrability condition~\eqref{eq:log-integrability}. Let
$\tilde X:=P(T_{\smash{X}}^*M)$
be the projectivized cotangent bundle over \(X\). Since each
\(f_\omega\) is a \(C^1\) local diffeomorphism, we can consider the  projective inverse cotangent random  map
\[
        \skew{4}{\tilde}{f}_\omega(x,[\xi])
        =
        \bigl(f_\omega(x),
        [( Df_\omega(x)^{-1})^*\xi]\bigr),
        \qquad
        (x,[\xi])\in\tilde X,
\]
where \(\xi\in T_x^*M\setminus\{0\}\) is a unit representative of the
projective covector \([\xi]\), and $(Df_\omega(x)^{-1})^*$ denotes the dual map of $Df_\omega(x)^{-1}$, that is $(Df_\omega(x)^{-1})^*\xi =\xi \circ Df_\omega(x)^{-1}$. 

We denote by \(\mathcal{I}_\mu(\skew{4}{\tilde}{f}\,)\) the set of
\(\skew{4}{\tilde}{f}\)-stationary probability measures projecting to a given
\(f\)-stationary measure \(\mu\). The associated cotangent potential is
\[
        \tilde\phi_f(x,[\xi])
        :=
        \int
        \log
        \|(Df_\omega(x)^{-1})^*\xi\|
        \,d\mathbb P.
\]
The Lyapunov exponents of the random linear cocycle $(f,(Df^{-1})^*)$ are exactly given by
$-\lambda_1(\mu) \leq \dots \leq -\lambda_m(\mu)$. By the variational principle in~\cite{BN26} we have that 
\[
        -\lambda_m(\mu)=\sup_{\tilde\mu\in\mathcal{I}_\mu(\skew{4}{\tilde}{f}\,)}
        \int\tilde\phi_f\,d\tilde\mu = \max \big\{\int \tilde{\phi}_f \, d\tilde\mu:  \tilde\mu \in \mathcal{I}_\mu(\skew{4}{\tilde}{f}) \, \text{ergodic}\big\}.
\]
We shall also use the cotangent version of quasi-irreducibility.
\begin{dualdefi}
\label{def:cotangent-quasi-irreducible}
We say that \(\mu\in\mathcal{I}(f)\) is  \emph{coquasi-irreducible}, if every
\(\tilde\mu\in\mathcal{I}_\mu(\skew{4}{\tilde}{f}\,)\) is maximizing for the
cotangent potential, that is,
$    \int\tilde\phi_f\,d\tilde\mu
        =
        -\lambda_m(\mu)$.
We say that \(f\) is coquasi-irreducible if every
\(\mu\in\mathcal{I}(f)\) is coquasi-irreducible. Finally, $f$ is maximal coquasi-irreducible if it is coquasi-irreducible and $\lambda_m(\mu)=\lambda_m(f)$ for all $\mu \in \mathcal{I}(f)$ where
$$
\lambda_m(f):=\inf \{\lambda_m(\eta):\eta \in \mathcal{I}(f)\} = \min \{\lambda_m(\eta):\eta \in \mathcal{I}(f) \ \text{ergodic}\}.
$$
\end{dualdefi}

Analogously to Theorem~\ref{thm:mainB}, we can characterise coquasi-irreducibility as the absence of a proper \((Df^{-1})^*\)-invariant cotangent
subbundle over \(\mu\) on which the maximal cotangent exponent drops. To this, we first introduce the cotangent analogue of the equator. We say that a measurable cotangent subbundle
\(\mathcal{L}\subset T_X^*M\)  is said to be \emph{$(Df^{-1})^*$-invariant over \(\mu\)} if $$(Df_\omega(x)^{-1})^*\mathcal{L}_{x}=\mathcal{L}_{f_\omega(x)} \quad \text{for \(\mathbb P\times\mu\)-a.e.~\((\omega,x)\).}$$ 
If \(\mathcal{L}\) is such a subbundle, we define
\[
        \tilde{\lambda}_1(\mu|_{\mathcal{L}})
        :=
        \lim_{n\to\infty}
        \frac1n
        \int
        \log
        \|
        (Df_\omega^n(x)^{-1})^*
        |_{\mathcal{L}_x}
        \|
        \,d\mathbb P\,d\mu .
\]

\begin{dualdefi}
    We call the \emph{coequator} of an ergodic $f$-stationary measure $\mu$ to the largest $(Df^{-1})^*$-invariant vector bundle $\tilde{\mathcal{L}}$ of $T^*_{X}M$ over $\mu$ where $\tilde{\lambda}_1(\mu|_{\tilde{\mathcal L}})<-\lambda_m(\mu)$. 
\end{dualdefi}  

Finally, we get the following analogously result to Theorem~\ref{thm:mainB}. 
\begin{dualmainthm}
\label{prop:cotangent-equator-criterion}
Let $f$ be the restriction to an invariant compact set $X$ of a $C^1$ random map satisfying the integrability condition~\eqref{eq:log-integrability} and consider an ergodic measure \(\mu\in\mathcal{I}(f)\). Then \(\mu\) is coquasi-irreducible if and only if its coequator is trivial.

Moreover, if $f$ is maximal coquasi-irreducible, then for every \(x\in X\) and unit covector \(\xi\in T_x^*M\),
$$
 -\lambda_m(f) = \lim_{n\to\infty} \frac1n \int \log \|(Df_\omega^n(x)^{-1})^*\xi\| \,d\mathbb P = \lim_{n\to\infty} \frac1n \log \|(Df_{\tilde\omega}^n(x)^{-1})^*\xi\| \quad \text{for \(\mathbb P\)-a.e.~\(\tilde\omega\in\Omega\)}.
$$
\end{dualmainthm}

Moreover, we can analogously introduce a vertical metric $\mathrm{d}_{\mathrm{vert}}$ in $\tilde{X}$. Then, we similarly can define the vertical local Lipschitz constant $L\skew{4}{\tilde}{f}_\omega(z)=\|D^{\rm vert}\skew{4}{\tilde}{f}_\omega(z)\|$ for $z=(x,[\xi])\in \tilde{X}$ and the vertical maximal Lyapunov exponent $\lambda_{\rm vert}(\tilde \nu)$ of an $\skew{4}{\tilde}{f}$-stationary measure $\tilde{\nu}$.  Similarly, we say \emph{vertical mostly contracting} and \emph{vertical contracting on average} for $\skew{4}{\tilde}{f}$  and \emph{vertical spectral gap} on $\mathscr{C}^\alpha(\tilde X)$ for the annealed Koopman operator $\tilde P$  associated with $\skew{4}{\tilde}{f}$. 
Hence, we get the following analogous result to Theorem~\ref{mainthm:equivalence-spectral-gap} and Corollary~\ref{maincor:equivalence-mostly}.

\begin{dualmainthm}
\label{thm:cotangent-projective-criterion}
Let \(f\) be the restriction to an invariant compact set \(X\) of a
  \(C^1\) random map on an \(m\)-dimensional
manifold satisfying for some $\beta>0$ the integrability conditions
\begin{equation}\label{eq:moment-condition-thmBbis}
        \int
        \log^+\|Df_\omega\|_\infty
        +
        \log^+\|(Df_\omega)^{-1}\|_\infty
        \,d\mathbb P<\infty 
\quad \text{and}  \quad 
\int
       \bigl( \sup_{\tilde z\in\tilde X}
        L\skew{4}{\tilde}{f}_\omega(\tilde z)\bigr)^\beta
        \,d\mathbb P<\infty.
\end{equation}
 Assume that $\lambda_{m-1}(\mu)>\lambda_m(\mu)$ for every $\mu\in\mathcal{I}(f)$. 
Then the following conditions are equivalent:
\begin{enumerate}
    \item \(f\) is coquasi-irreducible;
    \item $\skew{4}{\tilde}{f}$ is vertical mostly contracting;
    \item \(\skew{4}{\tilde}{f}\) is vertical contracting on average;
    \item $\tilde{P}$  has vertical spectral gap on $\mathscr{C}^\alpha(\tilde X)$
    for every sufficiently small \(\alpha>0\);
    \item $\mathcal{I}_\mu(\skew{4}{\tilde}{f}\,)$  is a singleton for every \(\mu\in\mathcal{I}(f)\). 
\end{enumerate}
\end{dualmainthm}

\begin{dualmaincor}\label{maincor:equivalence-mostlybis} Let $f$ be the restriction to an invariant compact set $X$ of a $C^1$ random map of an $m$-dimensional manifold satisfying~\eqref{eq:moment-condition-thmBbis}. Then, $\skew{4}{\tilde}{f}$ is vertical mostly contracting if and only if $f$ is coquasi-irreducible and $\lambda_{m-1}(\mu)>\lambda_m(\mu)$ for every $\mu\in \mathcal{I}(f)$.
\end{dualmaincor}

For \(\mu\in\mathcal{I}(f)\), define
\[
        \tilde{\lambda}_-(\mu)
        :=
        \inf_{\tilde\mu\in\mathcal{I}_\mu(\skew{4}{\tilde}{f}\,)}
        \int \tilde\phi_f\,d\tilde\mu
\quad \text{and} \quad 
        \tilde{\lambda}_-(f)
        :=
        \inf_{\mu\in\mathcal{I}(f)}\tilde{\lambda}_-(\mu)
        =
        \inf_{\tilde\mu\in\mathcal{I}(\skew{4}{\tilde}{f}\,)}
        \int_{\tilde X}\tilde\phi_f\,d\tilde\mu .
\]
Moreover, again from~\cite{BM,BN26}, if $\mu$ is ergodic,
\begin{align*}
    \tilde{\lambda}_-(\mu) &= \min\left\{ \int \tilde{\phi}_f \, d\tilde\mu : \tilde\mu \in \mathcal{I}_\mu(\skew{6}{\hat}{f}\,) \ \text{ergodic} \right\}  \\ &= \lim_{n\to\infty} \int \big(\min_{\hat{\xi} \in P(T_x^*M)} \frac{1}{n}\int \log \|(Df^n_\omega(x)^{-1})^*\xi\|\, d\mathbb{P}\big)\,d\mu \\[0.25cm]
    \tilde{\lambda}_-(f) &=  \min \left\{ \int \tilde\phi_f \, d\tilde\mu:  \tilde \mu\in \mathcal{I}(\skew{4}{\tilde}{f}\,) \ \text{ergodic}\right\} =\lim_{n\to\infty}  \min_{(x,[\xi]) \in\tilde X} \frac{1}{n} \int \log \|(Df^n_\omega(x)^{-1})^*\xi\| \, d \mathbb{P} \\
    &=\inf_{(x,[\xi]) \in\tilde X} \liminf_{n\to\infty} \frac{1}{n} \int \log \|(Df^n_\omega(x)^{-1})^*\xi \| \, d \mathbb{P}
    \end{align*}
 The  "$\lim_{n\to\infty}$" above can be replaced by "$\sup_{n\geq 1}$". In view of this, the following definition agrees
with the usual finite-time formulation of coexpansion on average.

\stepcounter{dualdefi}
\begin{dualdefi}
\label{def:coexpanding-on-average}
We say that \(f\) is \emph{coexpanding on average} (or \emph{coexpanding} for short) if $\tilde{\lambda}_-(f)>0$.
\end{dualdefi}

Similarly to before, we have that $\tilde{\lambda}_-(f)$ is lower semicontinuous and thus the set of coexpanding maps is an open set of $C^1$ random maps.  We also find the following dual corollary from the definition of coexpanding and coquasi-irreducibility. 

\begin{dualmaincor}
\label{cor:coexpanding-equator-positive}
Let $f$ be the restriction to an invariant compact set $X$ of a $C^1$ random map satisfying the integrability condition~\eqref{eq:log-integrability}. Then, 
\begin{enumerate}[label=(\roman*),itemsep=0.1cm,leftmargin=0.75cm]
    \item 
 $f$ is coexpanding if and only if every non-trivial $(Df^{-1})^*$-invariant bundle $\mathcal L$ over an $f$-stationary measure $\mu$ satisfies $\tilde{\lambda}_1(\mu|_{\mathcal L})>0$. In particular   $\lambda_m(\mu)<0$ for every $\mu\in\mathcal{I}(f)$;
\item if $f$ is coexpanding and $\lambda_{m-1}(\mu)\geq 0$ for every $\mu\in\mathcal{I}(f)$, then $f$ is coquasi-irreducible;
\item if $f$ is coquasi-irreducible,  $$\text{$f$ is coexpanding if and only if $\lambda_m(\mu)<0$ for every $\mu \in \mathcal{I}(f)$;}$$
\item if $\lambda_m(\mu)<0\leq \lambda_{m-1}(\mu)$ for all $\mu\in\mathcal{I}(f)$, then 
$$\text{$f$ is coexpanding if and only if it is coquasi-irreducible}.$$ 
\end{enumerate}
\end{dualmaincor}

We also get, as in Theorem~\ref{main:thmD}, the continuity of the bottom exponent under coquasi-irreducibility. 

\begin{dualmainthm}
\label{main:thmDbis}
Let $f \in \Diff^1_{\mathbb P}(M,\mu)$ such that $\mu$ is  coquasi-irreducible for $f$. Then the map
\[
   \lambda_m(\mu,\cdot): \bigl(\Diff^1_{\mathbb P}(M,\mu), d_{C^1}\bigr) \to \mathbb{R},
   \quad
   g \mapsto \lambda_m(\mu,g):= \lim_{n\to \infty} \frac{1}{n} \int \log m(Dg^n_\omega(x)) \, d\mathbb{P}\, d\mu
\]
is continuous at $f$. Moreover, if $f$ is maximal  coquasi-irreducible, the map
$$
  \lambda_m:\bigl(\Diff_{\mathbb P}^1(M),d_{C^1}\bigr)\to\mathbb R, \quad g \mapsto \lambda_m(g):=\inf\big\{\lambda_m(\eta): \eta \in \mathcal{I}(g)\big\}
$$
is continuous at $f$. 
\end{dualmainthm}

We also study the continuity of $\mu\mapsto \lambda_m(\mu)$ for a fixed random map $f$ in $\mathrm{Diff}^1_{\mathbb P}(M)$. As before, this map is lower semicontinuous with respect to the weak$^*$ topology. The following shows that under coquasi-irreducibility, we can get the continuity: 
\begin{dualmainthm}
\label{prop:continuity-measure-lyapunov-bis}
Let $f$ be the restriction to an invariant compact set $X$ of a $C^1$ random map satisfying the integrability conditions~\eqref{eq:log-integrability}.  If \(\mu_0\in\mathcal{I}(f)\) is coquasi-irreducible, then  $\mu\in\mathcal{I}(f)\mapsto \lambda_m(\mu)$  is continuous at $\mu_0$ with respect to the weak$^*$ topology.  
In particular, if \(f\) is coquasi-irreducible, then $       \mu \in \mathcal{I}(f) \mapsto \lambda_m(\mu)$ is continuous. 
\end{dualmainthm}

\subsection{Interplay between expanding and coexpanding}
\label{subsec:interplay-expanding-coexpanding}

We first record a simple geometric relation between equators and
coequators. This relation is automatic in dimension two, but in higher
dimension it only gives information for extremal-dimensional
subbundles.

\begin{mainprop}
\label{prop:equator-coequator-extremal-dimension}
Let \(f\) be the restriction to an invariant compact set \(X\) of a
\(C^1\) random map of an \(m\)-dimensional manifold satisfying
\eqref{eq:log-integrability}. Let \(\mu\in\mathcal{I}(f)\) be ergodic.

\begin{enumerate}[label=(\roman*),itemsep=0.1cm,leftmargin=1.25cm]
    \item If
    $\lambda_1(\mu)=\lambda_m(\mu)$,
    then \(\mu\) is quasi-irreducible and coquasi-irreducible.

    \item If \(\mu\) is quasi-irreducible, then \(\mu\) has no
    coequator  of dimension \(m-1\).

    \item If \(\mu\) is coquasi-irreducible, then \(\mu\) has no
     equator  of dimension \(m-1\).

    \item If \(m=2\), then \(\mu\) is quasi-irreducible if and only if
    \(\mu\) is coquasi-irreducible.
\end{enumerate}
\end{mainprop}

We now combine the previous geometric relation with the sign conditions
appearing in the definitions of expanding and coexpanding.

\begin{maincor}
\label{cor:interplay-expanding-coexpanding}
Let \(f\) be the restriction to an invariant compact set \(X\) of a
\(C^1\) random map of a compact \(m\)-dimensional manifold satisfying
\eqref{eq:log-integrability}. Then the following statements hold.
    \begin{enumerate} 
     \item if $f$ is expanding,  has no non-trivial coequators of dimension less than $m-1$ and $\lambda_2(\mu)< 0$ for every $\mu\in \mathcal{I}(f)$, then 
    $f$ is coquasi-irreducible and coexpanding.  
\item if $f$ is coexpanding, has no non-trivial equators of dimension less than $m-1$, and $\lambda_{m-1}(\mu)>0$ for every $\mu\in \mathcal{I}(f)$, then  $f$ is quasi-irreducible and expanding. 
    \item if $m=2$ and $\lambda_2(\mu)<0<\lambda_1(\mu)$ for every $\mu\in\mathcal{I}(f)$, then 
    $$ 
    \text{expanding $\iff$ coexpanding $\iff$  quasi-irreducible  $\iff$  coquasi-irreducible}
    $$
    \end{enumerate}
\end{maincor}

\subsection{Quasi-irreducibility and expansion in intermediate dimensions}
\label{subsec:higher-dimensional-expansion}

The notions introduced above concern the projective tangent bundle and therefore
control the top Lyapunov exponent. It is natural to ask what happens for the
intermediate Lyapunov exponents. We now explain that the same formalism applies
to Grassmannian bundles and gives higher-dimensional versions of
quasi-irreducibility, expansion, and their cotangent duals.

Let \(1\le k\le m\). For \(\mu\in\mathcal{I}(f)\), write
\[
        \Sigma_k(\mu)
        :=
        \lambda_1(\mu)+\cdots+\lambda_k(\mu),
        \qquad
        \Sigma_0(\mu):=0.
\]
Let $\hat X_k
        :=
        \bigsqcup_{x\in X}\mathrm{G}_k(T_xM)$
be the Grassmannian bundle of \(k\)-planes. The derivative cocycle induces a
random map
\[
        f_\omega^{[k]}(x,E)
        :=
        \bigl(f_\omega(x),Df_\omega(x)E\bigr) \qquad  \text{for $(x,E)\in \hat X_k$.}
\]
We consider the \(k\)-Jacobian potential
\[
        \phi_f^{[k]}(x,E)
        :=
        \int
        \log \mathrm{Jac}_k\bigl(Df_\omega(x)|_E\bigr)
        \,d\mathbb P \quad \text{where} \ \ \mathrm{Jac}_k\bigl(Df_\omega(x)|_E\bigr):=\|\wedge^k (Df_\omega(x)|_E)\|.
\]
For \(\mu\in\mathcal{I}(f)\), define
\[
        \lambda_-^{[k]}(\mu)
        :=
        \inf_{\hat\mu\in\mathcal{I}_\mu(f^{[k]})}
        \int \phi_f^{[k]}\,d\hat\mu
\quad \text{and} \quad
        \lambda_+^{[k]}(\mu)
        :=
        \sup_{\hat\mu\in\mathcal{I}_\mu(f^{[k]})}
        \int \phi_f^{[k]}\,d\hat\mu 
\]
where \(\mathcal{I}_\mu(f^{[k]})\) denotes the set of \(f^{[k]}\)-stationary
probability measures projecting to \(\mu\). Similarly we denote by \(\mathcal{I}^{\rm erg}_\mu(f^{[k]})\) the subset of \(\mathcal{I}_\mu(f^{[k]})\) of ergodic \(f^{[k]}\)-stationary measures. 
According to~\cite{BN26}, the variational formula gives
\[
        \Sigma_k(\mu)=\lambda_+^{[k]}(\mu)=\max_{\hat\mu\in\mathcal{I}^{\rm erg}_\mu(f^{[k]})}
        \int \phi_f^{[k]}\,d\hat\mu.
\]
Moreover, 
\begin{equation} \label{eq:limit}
    \begin{aligned}
        \lambda_-^{[k]}(\mu)  &= \min_{\hat\mu\in\mathcal{I}^{\rm erg}_\mu(f^{[k]})}
        \int \phi_f^{[k]}\,d\hat\mu =
        \lim_{n\to\infty}
        \int
        \min_{E\in\mathrm{G}_k(T_xM)}
        \frac1n
        \int
        \log\mathrm{Jac}_k\bigl(Df_\omega^n(x)|_E\bigr)
        \,d\mathbb P
        \,d\mu.
\end{aligned}
\end{equation}
We also define the uniform
quantities
\[
        \lambda_-^{[k]}(f)
        :=
        \inf_{\mu\in\mathcal{I}(f)}\lambda_-^{[k]}(\mu) \quad \text{and} \quad \Sigma_k(f)=\sup_{\mu\in\mathcal{I}(f)} \Sigma_k(\mu) = \max_{\mu\in\mathcal{I}^{\rm erg}(f)} \Sigma_k(\mu) .
\]
Equivalently
\[
        \lambda_-^{[k]}(f)
        =
        \lim_{n\to \infty}  
        \min_{(x,E)\in\hat X_k} \frac1n
        \int
        \log\mathrm{Jac}_k\bigl(Df_\omega^n(x)|_E\bigr)
        \,d\mathbb P =  \inf_{(x,E)\in \hat X_k} \liminf_{n\to\infty} \frac{1}{n} \int \log\mathrm{Jac}_k\bigl(Df_\omega^n(x)|_E\bigr) \,d\mathbb{P}.
\]
The limit in the above and in~\eqref{eq:limit} can be replaced by \(\sup_{n\ge1}\).

\begin{defi} A measure $\mu \in \mathcal{I}(f)$ is \emph{$k$-quasi-irreducible} if $\lambda^{[k]}_-(\mu)=\Sigma_k(\mu)$. We also say  that $f$ is 
\begin{itemize}
    \item \emph{$k$-quasi-irreducible} if every $\mu\in \mathcal{I}(f)$ is $k$-quasi-irreducible;
    \item \emph{maximal $k$-quasi-irreducible} if $\lambda^{[k]}_-(f)=\Sigma_k(f)$; 
    \item \emph{$k$-expanding} if $\lambda^{[k]}_-(\mu)>0$ for every $\mu\in\mathcal{I}(f)$.
\end{itemize}
\end{defi}
Thus \(k\)-expansion is precisely uniform expansion on \(k\)-planes or in dimension $k$ in the
finite-time sense used in \cite{Ell23,BEFR25}.\footnote{The expansion gap condition of~\cite{BEFR25} can also be written in the same variational language $\bar{\lambda}^{[k],\delta}_{-}(f)>0$ coming from~\cite{BN26} using the induced random map on the flag space $\mathcal{F}_k$ of subspaces $E\subset F$ of dimension $k$ and $k+1$ respectively and the flag potential 
the annealed logarithmic expansion of the one-dimensional quotient \(F/E\), with the sign prescribed by \(\delta\in \{\pm 1\}\).} Expansion in all dimensions means
\(k\)-expansion for every \(1\le k\le m-1\). For \(k=1\), this is exactly the quasi-irreducibility, maximal quasi-irredicibillity and expansion on average previously defined.

In particular, $k$-expansion immediately implies $\Sigma_k(f)>0$.  The converse is not true in general. The missing condition is precisely $k$-quasi-irreducibility: if \(f\) is \(k\)-quasi-irreducible,~then
\[
        f\text{ is \(k\)-expanding}
        \quad\Longleftrightarrow\quad
        \Sigma_k(\mu)>0
        \text{ for every \ } \mu\in\mathcal{I}(f).
\]
As in the projective case, expansion gives quasi-irreducibility under a spectral
sign condition. Let $\Gamma_k(\mu)
        := 
        \lambda_1(\mu)+\cdots+\lambda_{k-1}(\mu)+\lambda_{k+1}(\mu)$.
Then
\[
        k\text{-expanding}
        \;+\;
        \Gamma_k(\mu)\le0
        \text{ for every}\mu\in\mathcal{I}(f)
        \quad\Longrightarrow\quad
        k\text{-quasi-irreducibility}.
\]
Therefore, if $\Gamma_k(\mu)\le0<\Sigma_k(\mu)$  for every $\mu\in\mathcal{I}(f)$, then
\[
        k\text{-expanding}
        \quad\Longleftrightarrow\quad
        k\text{-quasi-irreducibility}.
\]

We now introduce the cotangent analogues. Define on $\tilde X_k:=\mathrm{G}_k(T^*_XM)$, 
\[
        \skew{4}{\tilde}{f}_\omega^{[k]}(x,F)
        :=
        \bigl(f_\omega(x),(Df_\omega(x)^{-1})^*F\bigr)
        \quad \text{and} \quad \tilde\phi_f^{[k]}(x,F)
        :=
        \int
        \log
        \mathrm{Jac}_k\bigl((Df_\omega(x)^{-1})^*|_F\bigr)
        \,d\mathbb P
\]        
For \(\mu\in\mathcal{I}(f)\), we define $\tilde{\lambda}_\pm^{[k]}(\mu)$ as above replacing  $\phi_f^{[k]}$ by $\tilde{\phi}_{f}^{[k]}$. Analogously, we introduce  $\tilde{\lambda}_\pm^{[k]}(f)$ and get similar asymptotic and variational representation substituting $Df^n_\omega(x)$ by $(Df^n_\omega(x)^{-1})^*$, and $\Sigma_k(\mu)$ and $\Sigma_k(f)$ by 
$$ \tilde{\Sigma}_k(\mu):= -\lambda_m(\mu)-\dots-\lambda_{m-k+1}(\mu) \quad \text{and} \quad \tilde\Sigma_k(f)=\sup_{\mu\in\mathcal{I}(f)} \tilde{\Sigma}_k(\mu).$$
\begin{defi} A measure $\mu \in \mathcal{I}(f)$ is \emph{$k$-coquasi-irreducible} if $\tilde{\lambda}^{[k]}_-(\mu)=\tilde{\Sigma}_k(\mu)$ and $f$ is 
\begin{itemize}
    \item \emph{$k$-coquasi-irreducible} if every $\mu\in \mathcal{I}(f)$ is $k$-coquasi-irreducible;
    \item \emph{maximal $k$-coquasi-irreducible} if $\tilde{\lambda}^{[k]}_-(f)=\tilde{\Sigma}_k(f)$; 
    \item \emph{$k$-coexpanding} if $\tilde{\lambda}^{[k]}_-(\mu)>0$ for every $\mu\in\mathcal{I}(f)$.
\end{itemize}
\end{defi}
Again, for \(k=1\), these are the notions of
coexpansion and coquasi-irreducibility previously introduced.
The relation between the tangent and cotangent quantities is governed by the
Jacobian determinant. If \(E\in\mathrm{G}_{m-k}(T_xM)\) and
\(E^\perp\in\mathrm{G}_k(T_x^*M)\) is its annihilator, then
\[
        \mathrm{Jac}_k\bigl((Df_\omega(x)^{-1})^*|_{E^\perp}\bigr)
        =
        \frac{\mathrm{Jac}_{m-k}\bigl(Df_\omega(x)|_E\bigr)}
             {|\det Df_\omega(x)|}.
\]
Therefore, for every \(\mu\in\mathcal{I}(f)\),
\[
        \widetilde\lambda_\pm^{[k]}(\mu)
        =
        \lambda_\pm^{[m-k]}(\mu)-\Sigma_m(\mu),
\]
where
\[
        \Sigma_m(\mu)
        =
        \lambda_1(\mu)+\cdots+\lambda_m(\mu)
        =
        \int
        \log|\det Df_\omega(x)|
        \,d\mathbb P\,d\mu.
\]
Consequently,
\[
        \widetilde\lambda_-^{[k]}(f)
        =
        \inf_{\mu\in\mathcal{I}(f)}
        \bigl(
        \lambda_-^{[m-k]}(\mu)-\Sigma_m(\mu)
        \bigr) = \min_{\mu\in\mathcal{I}^{\rm erg}(f)}
        \bigl(
        \lambda_-^{[m-k]}(\mu)-\Sigma_m(\mu)
        \bigr).
 \]
Indeed, by the global lower variational formula for the cotangent
Grassmannian cocycle, \(\widetilde\lambda^{[k]}_-(f)\) is attained by an
ergodic \(\widetilde f^{[k]}\)-stationary probability measure.
The projection of an ergodic stationary lift is an ergodic
\(f\)-stationary measure, which proves the second equality.

\begin{maincor}
$f$ is \(k\)-coexpanding if and only if
$\lambda_-^{[m-k]}(\mu)>\Sigma_m(\mu)$ for every $\mu\in\mathcal{I}(f)$.
\end{maincor}
If the determinant contribution is non-positive, we get that
$$
\text{$\Sigma_m(f) \leq 0$  \;+\;  $(m-k)$-expanding \quad $\implies$ \quad   $k$-coexpanding.}
$$
In the volume-preserving case, where $|\det Df_\omega(x)|=1$ 
for \(\mathbb P\)-a.e.~\(\omega\) and every~\(x\), we have 
$\Sigma_m(f)=0$ and hence
$\tilde\lambda_\pm^{[k]}(\mu)
        =
        \lambda_\pm^{[m-k]}(\mu)$. Therefore, in this volume preserving case, we get
\[
        k\text{-coexpanding}
        \quad\Longleftrightarrow\quad
        (m-k)\text{-expanding}
\]
and
\[
        k\text{-coquasi-irreducibility}
        \quad\Longleftrightarrow\quad
        (m-k)\text{-quasi-irreducibility}.
\]
These higher-dimensional notions clarify the relation between equators and
coequators: 
\begin{enumerate}
    \item If \(\mu\) is \(k\)-quasi-irreducible, then \(\mu\) has no coequator
of dimension \(m-k\). Dually,
\item if \(\mu\) is \(k\)-coquasi-irreducible, then
\(\mu\) has no equator of dimension \(m-k\). Hence,
\item  if \(\mu\) is
\(k\)-quasi-irreducible for every \(1\le k\le m-1\), then \(\mu\) has no
non-trivial coequators and is coquasi-irreducible.  Similarly,
\item if \(\mu\) is
\(k\)-coquasi-irreducible for every \(1\le k\le m-1\), then \(\mu\)  is quasi-irreducible.
\end{enumerate}
For \(m=2\), the only non-trivial
proper dimension is \(1\), so quasi-irreducibility and coquasi-irreducibility are
equivalent as we have concluded before. 
 Finally, \(k\)-quasi-irreducibility gives the expected continuity criterion for
intermediate exponents. 
\begin{maincor}
If \(\mu\) is \(k\)-quasi-irreducible for $f_0\in \mathrm{Diff}_{\rm loc,\mathbb{P}}^1(X)$,  then the partial sum $\Sigma_k(\mu)=\Sigma_k(\mu,f)$  
is continuous at \(f_0\) in the fixed-\(\mu\) space of random \(C^1\) local
diffeomorphisms $f$ with $\|Df_\omega\|_\infty$ and $\|(Df_\omega)^{-1}\|_\infty$ bounded by some fixed constant $C>0$ for $\mathbb{P}$-a.e~$\omega\in \Omega$. 
\end{maincor}
Since $\lambda_i(\mu)=\Sigma_i(\mu)-\Sigma_{i-1}(\mu)$ the exponents $      \lambda_1(\mu),\dots,\lambda_k(\mu)$
are continuous at \(f_0\) provided $f_0$ is $i$-quasi-irreducible for every
\(1\le i\le k\). Dually, \(i\)-coquasi-irreducibility for
\(1\le i\le k\) gives continuity of the bottom \(k\) Lyapunov exponents. Also, maximal $k$-quasi-irreducibility and maximal $k$-coquasi-irreducibility give continuity of $\Sigma_k(f)$ and $\tilde{\Sigma}_k(f)$ respectively.

\subsection{Organization of the paper}
\label{subsec:organization}

The remainder of the paper is organized as follows. In
Section~\ref{sec:proof-equators} we prove Theorem~\ref{thm:mainB} and its
cotangent analogue, Theorem~\ref{prop:cotangent-equator-criterion}. We first work
in the general setting of random linear bundle isomorphisms and then apply the
result to the derivative cocycle and to the inverse-dual cotangent cocycle.

In Section~\ref{sec:proof-contraction-spectral-gap} we prove
Theorem~\ref{mainthm:equivalence-spectral-gap},
Theorem~\ref{thm:cotangent-projective-criterion}, and their corollaries. The proof
is again developed in a general linear-bundle framework: we first establish a
two-metric criterion for contraction on average, then prove the required projective
estimates and the vertical spectral gap, and finally apply these results to the
projective tangent and cotangent dynamics.

In Section~\ref{sec:proof-continuity} we prove the continuity results for the
extremal Lyapunov exponents, namely Theorem~\ref{main:thmD} and its cotangent
counterpart Theorem~\ref{main:thmDbis}. Finally, in
Section~\ref{sec:proof-corollaries} we prove the remaining propositions and
corollaries concerning expansion, coexpansion, their interplay, and the
higher-dimensional Grassmannian versions announced in
Subsection~\ref{subsec:higher-dimensional-expansion}.


\section{Proof of Theorem~\ref{thm:mainB} and~\ref{prop:cotangent-equator-criterion}} \label{sec:proof-equators}

 We prove a slightly more general version of such theorems for a random linear bundle isomorphism.

 \subsection{Random linear bundle isomorphisms}
\label{s:setting-linear}
Let $X$ be a compact metric space. Let $\pi : \mathcal E \to X$ be a continuous rank-$m$ vector bundle with a continuous Euclidean norm $\|\cdot\|$. Hence, since \(X\) is compact and \(\pi:\mathcal E\to X\) is a continuous finite-dimensional vector bundle, the projective total space $\hat{\mathcal{E}}=P({\mathcal{E}})$ is compact and metrizable.

Consider a continuous random map $f:\Omega \times X \to X$. Let $A:\Omega \times \mathcal E \to \mathcal E$ be a random linear bundle isomorphism covering $f$. That is, a measurable map $A_\omega :\mathcal{E}\to \mathcal{E}$ where $A_{\omega}(x) : \mathcal E_x \to \mathcal E_{f_{\omega}(x)}$ is an invertible linear map for $x\in X$ and 
$\pi\circ A_\omega = f_\omega \circ \pi$ for $\mathbb{P}$-a.e.~$\omega \in \Omega$. We have 
$$
  A^0_\omega(x)=\mathrm{id} \quad \text{and} \quad A^{n+1}_{\omega}(x):= A_{\sigma^n \omega}(f^n_\omega(x))\circ A^{n}_\omega(x) \quad \text{for } n\geq 0. 
$$
In what follows we assume the integrability condition 
$$
  \int \log^+ \|A_\omega\|_\infty \, d\mathbb P<\infty.
$$
Under this condition, we define the Lyapunov exponents as in the introduction, just replacing $Df^n_\omega(x)$ by $A^n_\omega(x)$. We also denote by $\hat A:\Omega \times \hat{\mathcal{E}} \to \hat{\mathcal{E}}$ the projective random map associated with $A$ and by $\phi_A:\hat{\mathcal E}\to\mathbb [-\infty,\infty)$ the potential given by $$ \phi_A (x,\hat v):=\int \log \|A_\omega(x)v\|\, d\mathbb{P}$$ 
where $\hat v \in P(\mathcal{E}_x)$ and $v$ is a unit representative.  

\subsection{Quasi-irreducibility, equators and strong law of large numbers} The first part of Theorem~\ref{thm:mainB} could be followed easily from the non-random multiplicative ergodic theorem (see~\cite[Thm.~1.1 Chap.~III]{Kif:86}). However, we will provide a different, direct proof that follows the ideas of~\cite[Prop.~5.3.2]{BM}. To do this, we first show the following lemma, where we get an $A$-invariant bundle from an ergodic projective stationary measure. Recall that a bundle $\mathcal L$ is said to be $A$-invariant if $A_\omega(x)\mathcal{L}_x = \mathcal L_{f_\omega(x)}$ for $\mathbb P$-a.e.~$\omega \in \Omega$.  Also, for a subset \(S\subset P(V)\) of the projective space of a finite-dimensional vector space \(V\), we write $\operatorname{span}(S)$ for the smallest linear subspace \(L\subset V\) such that \(S\subset P(L)\).

\begin{lem} 
\label{lem:bundle-generated-by-nu}
Let $\nu\in\mathcal{I}(\hat A)$ be ergodic and set $\mu=\pi\nu$ and $d\nu=\nu_x \, d\mu(x)$. Define
\[
   {S}_x:=\supp(\nu_x)\subset P(\mathcal E_x),
   \qquad 
   \mathcal L_x:=\mathrm{span}(S_x)\subset \mathcal E_x .
\]
Then $\mathcal L=\sqcup_{x\in X} \mathcal{L}_x$ is an $A$-invariant bundle over $\mu$ with $\nu(\hat{\mathcal{L}})=1$  where
$\hat{\mathcal L}:=P(\mathcal L)$ and
\begin{equation}\label{eq:nu-realizes-restricted-exponent}
   \int \phi_A\, d\nu
   =
   \lambda_1(\mu|_{\mathcal L}).
\end{equation}
Moreover, if
\(\mathcal F\) is an \(A\)-invariant subbundle over \(\mu\)
such that \(\nu(\hat{\mathcal F})=1\), then
$\int \phi_A\,d\nu
        \leq
        \lambda_1(\mu|_{\mathcal F})$. In particular, 
        $$\lambda_-(\mu):=\inf_{\hat \mu \in \mathcal{I}_\mu(\hat A)} \int \phi_A \, d\hat\mu \leq \lambda_1(\mu|_{\mathcal{F}}) \quad \text{for every $A$-invariant bundle $\mathcal F$.}
        $$
\end{lem}

\begin{proof}
By construction, $\nu(\hat{\mathcal L})=1$.
Since $\nu$ is $\hat A$-stationary, 
\[
1=\nu(\hat{\mathcal L})
=\int (\hat A_\omega)_*\nu(\hat{\mathcal L})\, d\mathbb P
=\int \nu( (\hat A_\omega)^{-1}\hat{\mathcal L})\, d\mathbb P,
\]
hence $\nu((\hat A_\omega)^{-1}\hat{\mathcal L})=1$ for $\mathbb P$-a.e.\ $\omega$. Using the disintegration $d\nu(x,\hat v) = \nu_x\, d\mu(x)$, this implies that for $\mu$-a.e.~$x$, the measure $\nu_x$ gives full mass to the set $\{\hat v \in P(\mathcal E_x) : A_\omega(x)\hat v \in P(\mathcal L_{f_\omega(x)}) \}$. As this set is closed and $S_x = \supp(\nu_x)$, it must contain $S_x$. Hence $A_\omega(x)S_x \subset P(\mathcal L_{f_\omega(x)})$ for \(\mathbb P\times\mu\)-a.e.~\((\omega,x)\) and consequently,  since $\mathcal L_x = \mathrm{span}(S_x)$ and the map $A_\omega(x)$ is linear, we obtain
\[
   A_\omega(x)\mathcal L_x
   = A_\omega(x)\bigl(\mathrm{span}(S_x)\bigr)
   = \mathrm{span}( A_\omega(x) S_x)
   \subset \mathcal L_{f_\omega(x)}.
\]
Let $d(x) := \dim(\mathcal L_x)$. Since $A_\omega(x)$ is invertible and $A_\omega(x)\mathcal L_x \subset \mathcal L_{f_\omega(x)}$, we have $d(f_\omega(x)) \ge d(x)$ for $\mathbb{P} \times \mu$-a.e.\ $(\omega,x)$. Because $\mu$ is $f$-stationary, the measure $\bar{\mu}=\mathbb{P} \times \mu$ is invariant, which implies that the non-negative function $\Delta(\omega, x) = d(f_\omega(x)) - d(x)$ has zero integral. Thus, $d(f_\omega(x)) = d(x)$ $\bar{\mu}$-a.e., meaning $d$ is $F$-invariant. Since $\nu$ is ergodic, $\bar{\mu}$ is ergodic, and therefore $d(x)$ must be constant $\mu$-a.e. Consequently, the inclusion $A_\omega(x)\mathcal L_x \subset \mathcal L_{f_\omega(x)}$ becomes an equality of dimensions, proving that $\mathcal L$ is an $A$-invariant subbundle.

To obtain \eqref{eq:nu-realizes-restricted-exponent}, let $\phi(\omega,(x,\hat v))= \log \|A_\omega(x)v\|$.  Notice that
$\phi^+(\omega,(x,\hat v)) \leq
    \log^+\|A_\omega\|_\infty$,
and hence $\phi^+\in L^1(\mathbb P\times\nu)$. 
Therefore the extended Birkhoff theorem for quasi-integrable
observables, cf.\ \cite[Rem.~2.3.2]{BM}, applies to the skew-product $\hat F$ induced by $\hat f$ 
and the observable $\phi$  concluding
\begin{equation}\label{eq:directional-quenched-limit}
   \lim_{n\to\infty}
   \frac1n\log\|A^n_\omega(x)v\|
   = \lim_{n\to \infty}  \frac{1}{n} \sum_{i=0}^{n-1} \phi(\hat F^i(\omega,z))=
   \int \phi_A\, d\nu
   \quad
   \text{for $(\mathbb{P}\times\nu)$-a.e.~$(\omega,(x,\hat v))$.}
\end{equation}
Since $\nu(\hat{\mathcal L})=1$, the above limit holds for
 $(\mathbb P\times \nu)$-a.e.~$(\omega,(x,\hat v))$ with $\hat v\in P(\mathcal L_x)$.
Let $k:=\dim\mathcal L_x$, which is constant $\mu$-a.e.
{By Fubini's theorem, there exists a measurable set $G\subset\hat{\mathcal L}$ with $\nu(G)=1$ such that for every $(x,\hat v)\in G$, the limit in \eqref{eq:directional-quenched-limit} holds for $\mathbb P$-a.e.\ $\omega$. Since $\nu_x(G_x)=1$ where $G_x:=G\cap P(\mathcal L_x)$, and $S_x=\mathrm{supp}(\nu_x)$, we have $S_x\subset\overline{G_x}$ and consequently $\mathrm{span}(G_x)=\mathrm{span}(S_x)=\mathcal L_x$ for $\mu$-a.e.\ $x$. Choose a measurable family of unit vectors $v_1(x),\dots,v_k(x)\in\mathcal L_x$ with $\hat v_i(x)\in G_x$ such that $\{v_1(x),\dots,v_k(x)\}$ is a basis of $\mathcal L_x$ for $\mu$-a.e.\ $x$ (possible by a measurable selection argument since $\mathrm{span}(G_x)=\mathcal L_x$). By construction, for each $i=1,\dots,k$,}
\begin{equation}\label{eq:frame-limit}
   \lim_{n\to\infty}
   \frac1n\log\|A^n_\omega(x){v_i(x)}\|
   =
   \int \phi_A\, d\nu
   \quad
   \text{for $\mathbb P\times\mu$-a.e.~$(\omega,x)$.}
\end{equation}
{Since $\{v_1(x),\dots,v_k(x)\}$ is a basis of $\mathcal L_x$, there exists a measurable $C(x)>0$ such that}
\[
{C(x)^{-1}\|A^n_\omega(x)|_{\mathcal L_x}\|
   \le
   \max_{1\le i\le k}\|A^n_\omega(x)v_i(x)\|
   \le
   \|A^n_\omega(x)|_{\mathcal L_x}\| \quad \text{for every $n\ge1$}.}
\]
Taking $\frac{1}{n} \log$ and letting $n\to\infty$ {(noting that $\frac1n\log C(x)\to 0$)}, we obtain from
\eqref{eq:frame-limit} that
\begin{equation}\label{eq:operator-quenched-limit}
   \lim_{n\to\infty}
   \frac1n\log\|A^n_\omega(x)|_{\mathcal L_x}\|
   =
   \int \phi_A\, d\nu
   \quad
   \text{for $\mathbb P\times\mu$-a.e.\ $(\omega,x)$.}
\end{equation}
Finally, since
\(
   \log\|A^n_\omega(x)|_{\mathcal L_x}\|
\)
satisfies the usual subadditive inequality (with respect to the skew-product $F$ induced by $f$) and 
\(
   \log^+\|A_\omega\|_\infty
\)
is $\mathbb P$-integrable, Kingman's subadditive ergodic theorem gives 
\[
   \lambda_1(\mu|_{\mathcal L})
   :=
   \lim_{n\to\infty}\frac1n
   \int \log \|A^n_\omega(x)|_{\mathcal L_x}\|
   \, d\mathbb P d\mu
   =
      \lim_{n\to\infty}\frac1n
      \log \|A^n_\omega(x)|_{\mathcal L_x}\|
\]
for $(\mathbb{P}\times \mu)$-a.e.~$(\omega,x)$. 
Combining this with \eqref{eq:operator-quenched-limit} yields~\eqref{eq:nu-realizes-restricted-exponent}.

Finally, if $\mathcal F \subset \mathcal E$ is a \(A\)-invariant subbundle over \(\mu\) such that \(\nu(\hat{\mathcal F})=1\), one has $       \supp(\nu_x)\subset P(\mathcal F_x)$ for \(\mu\)-a.e.~$x$.
Hence $\mathcal L_x=\operatorname{span}(\supp(\nu_x)) \subset
        \mathcal F_x$  for \(\mu\)-a.e.~$x$. 
Therefore, 
 $ \|A_\omega^n(x)|_{\mathcal L_x}\|
        \leq
        \|A_\omega^n(x)|_{\mathcal F_x}\|$ for every \(n\geq1\). 
Taking logarithms, integrating, dividing by \(n\), and passing to the limit gives $\lambda_1(\mu|_{\mathcal L})
        \leq
        \lambda_1(\mu|_{\mathcal F})$. 
By~\eqref{eq:nu-realizes-restricted-exponent}, we have $     \int\phi_A\,d\nu
        \leq
        \lambda_1(\mu|_{\mathcal F})$ concluding the proof.
\end{proof}

\begin{prop}
\label{prop:non-linear-quasi-irreducibility}
Let  $A:\Omega \times \mathcal E \to \mathcal E$ be a random linear bundle isomorphism covering a continuous random map $f:\Omega \times X \to X$ of a compact metric space $X$ satisfying that $\log^+\|A_\omega\|_\infty$ is $\mathbb{P}$-integrable. Given $\mu\in\mathcal{I}(f)$ ergodic, the following conditions are equivalent:
\begin{enumerate}[label=(\arabic*), itemsep=0.05cm]
    \item {quasi-irreducibility:} 
     $\int \phi_A \, d\nu = \lambda_1(\mu)$ for every  $\nu\in \mathcal{I}_\mu(\hat A)$. 
    \item {absence of equators:} there is no proper $A$-invariant subbundle $\mathcal L$ over $\mu$ with $\lambda_1(\mu|_{\mathcal{L}})<\lambda_1(\mu)$.  
\end{enumerate}
\end{prop}

\begin{proof}
\noindent\textbf{$(1)\Rightarrow(2)$.}
Let $\mathcal L$ be a proper $A$-invariant bundle.
Let $K$ be the set of probability measures $\eta$ of $\hat{\mathcal{E}}$ such that
$\pi\eta=\mu$ and $\eta(\hat{\mathcal L})=1$.
Every $\eta\in K$ is canonically identified with an element of
$\operatorname{Prob}_\mu(\hat{\mathcal L})
   :=\{\eta\in\operatorname{Prob}(\hat{\mathcal L}):\ \pi\eta=\mu\}$,
which we endow with the {Young-measure topology}: $\eta_n\to\eta$ if and only if
$\int\psi\,d\eta_n\to\int\psi\,d\eta$ for every bounded measurable
$\psi:\hat{\mathcal L}\to\mathbb R$ which is continuous on each fiber $P(\mathcal L_x)$.
The set $K$ is non-empty: if $s:X\to\hat{\mathcal L}$ is a measurable section, 
then $\eta_0:=\int\delta_{s(x)}\,d\mu(x)\in K$. It is clearly convex, and by
\cite[Prop.~B.2]{BN26} it is compact and metrizable in the Young-measure topology. 
Since $\hat{\mathcal L}$ is $\hat A$-invariant, we have $\hat P^*(K)\subset K$ where $\hat P$
is the Koopman operator associated with $\hat A$. Since 
the map $\hat P^*:K\to K$ is affine, and it is continuous in the
 Young-measure topology,
by Schauder--Tychonoff, $\hat P^*$ admits a fixed point
$\nu\in K$. Hence $\nu\in\mathcal{I}_\mu(\hat A)$ and
$\nu(\hat{\mathcal L})=1$. From the ergodic decomposition and since $\mu$ is ergodic, we can
assume that $\nu$ is an ergodic $\hat A$-stationary measure. Hence, by~(1), we have
{$\lambda_1(\mu)=\int\phi_A\,d\nu\leq\lambda_1(\mu|_{\mathcal L})
\leq\lambda_1(\mu)$, where the first inequality is Lemma~\ref{lem:bundle-generated-by-nu}
applied with $\mathcal F=\mathcal L$ and the second one follows from
$\mathcal L\subset\mathcal E$; consequently,}
$\lambda_1(\mu|_{\mathcal{L}})=\int \phi_A \, d\nu =\lambda_1(\mu)$.  This shows (2).


\noindent\textbf{$(2)\Rightarrow(1)$.}
Let $\nu\in\mathcal{I}(\hat A)$ be ergodic with $\mu=\pi\nu$.
Let {$\mathcal L_\nu$} be the bundle generated by $\nu$ as in
Lemma~\ref{lem:bundle-generated-by-nu}. By \eqref{eq:nu-realizes-restricted-exponent}, $\int\phi_A\, d\nu = \lambda_1(\mu|_{{\mathcal L_\nu}})$. If {$\mathcal L_\nu$} is proper, then (2) implies $\lambda_1(\mu|_{\mathcal L_\nu})$ cannot be strictly smaller than $\lambda_1(\mu)$. {Since $\mathcal L_\nu \subset \mathcal E$, we always have $\lambda_1(\mu|_{\mathcal L_\nu}) \leq \lambda_1(\mu)$,} hence $\lambda_1(\mu|_{{\mathcal L_\nu}})=\lambda_1(\mu)$.
If {$\mathcal L_\nu=\mathcal{E}$}, the same equality is trivial. In both cases, $ \int\phi_A\, d\nu = \lambda_1(\mu)$. 
Thus, every ergodic $\hat A$-stationary measure projecting on $\mu$ is maximizing.
The general case follows from ergodic decomposition.
\end{proof}

By the upper semicontinuity of $\mu \in \mathcal{I}(f) \mapsto \lambda_1(\mu)$ and the Bauer maximum principle (see~\cite[Claim~2.2.7]{BM} for more details), we have that 
$$\lambda_1(A):= \sup  \{\lambda_1(\mu): \mu \in \mathcal{I}(f)\}=\max  \{\lambda_1(\mu): \mu \in \mathcal{I}(f) \ \text{ergodic}\}.$$  
Now we prove the Strong Law of Large Numbers:

\begin{prop} \label{prop:SLLN} Under the hypothesis of Proposition~\ref{prop:non-linear-quasi-irreducibility} and assuming aditionally that  $\log^+ \|A_\omega^{-1}\|_{\infty}$ is also $\mathbb P$-integrable, if $\lambda_1(A)=\int \phi_A \, d\nu$ for every $\nu \in \mathcal{I}(\hat A)$, i.e., if $A$ is maximal quasi-irreducible, then for every $x\in X$ and  unit vector $v \in \mathcal{E}_x$,
\begin{equation} \label{eq:SLLN1}
      \lambda_1(A)= \lim_{n\to \infty}  \frac{1}{n}\int \log \|{A^n_\omega(x)v}\| \, d \mathbb{P}
\end{equation}
and 
\begin{equation} \label{eq:SLLN2}
\lambda_1(A) =\lim_{n\to\infty} \frac{1}{n} \log \|{A^n_{\omega}(x)v}\| \quad \text{for $\mathbb{P}$-a.e.~$\omega \in \Omega$.}
\end{equation}
\end{prop}

\begin{proof} Set
$\phi(\omega,(x,\hat v))
   :=\log\|A_\omega(x)v\|$. This function only depends on the
zero-coordinate of $\omega$ and $\phi(\omega,\cdot)$ belongs to \(C(\hat{\mathcal E})\). For every \(\omega\), since  $\|\phi(\omega,\cdot)\|_\infty
   \leq
   \log^+\|A_\omega\|_\infty
   +
   \log^+\|(A_\omega)^{-1}\|_\infty$, the integrability conditions ensure that 
$
   \int\|\phi(\omega,\cdot)\|_\infty\,d\mathbb P<\infty$.
Moreover, its annealed version is $\phi_A(x,\hat v)
   =
   \int\phi(\omega,(x,\hat v))\,d\mathbb P$ and 
$$
 \frac{1}{n} \log \|{A^n_\omega(x)v}\|= \frac{1}{n} \sum_{i=0}^{n-1} \phi\big(\hatF^i(\omega, (x,\hat v))\big)
$$
where $\hat F$ is the skew-product associated with $\hat A$.  Hence, since $A$ is  maximal quasi-irreducibility, $\max \int \phi_A \, d\nu = \min \int \phi_A \, d\nu =\lambda_1(A)$, and therefore by~\cite[Cor.~C.2.5]{BM}, it holds~\eqref{eq:SLLN2}.

Similarly, denoting by $\hat P$ the annealed Koopman operator associated with $\hat A$, we have 
\[
\frac{1}{n}  \int \log \|{A^n_\omega(x)v}\| \, d\mathbb{P}= \frac{1}{n} \sum_{k=0}^{n-1} \hat P^k\phi_A(x,\hat v). 
\]
Hence, again as $A$ is  maximal quasi-irreducibility, by~\cite[Cor.~C.2.4]{BM}, it holds~\eqref{eq:SLLN1}.  
\end{proof}

\begin{proof}[Proof of Theorem~\ref{thm:mainB} and \ref{prop:cotangent-equator-criterion}] This follows by combine Propositions~\ref{prop:non-linear-quasi-irreducibility} and~\ref{prop:SLLN} for the random linear bundle isomorphisms $A=Df$ on $\mathcal{E}=T_XM$ and $A=(Df^{-1})^{*}$ on $\mathcal{E}=T^*_XM$ respesctively and noting that $\lambda_1(Df)=\lambda_1(f)$ and $\lambda_1((Df^{-1})^*)=-\lambda_m(f)$. 
\end{proof}

\section{Proof of 
Theorem~\ref{mainthm:equivalence-spectral-gap} and~\ref{thm:cotangent-projective-criterion} and Corollary~\ref{maincor:equivalence-mostly} and~\ref{maincor:equivalence-mostlybis}.
}
\label{sec:proof-contraction-spectral-gap}

First, we develop an abstract criterion for contraction on average in a metric space
which is not necessarily compact. The setting is the following.

\subsection{A two-metric criterion for contraction on average}
\label{subsec:two-metric-contraction}

Let \((X,d)\) be a compact metric space, and let \(\rho\) be another metric on \(X\). The metric \(d\) is used for compactness and semicontinuity, whereas \(\rho\) is used to measure contraction of a Bernouilli random map $f :\Omega\times X\to X$.  
For \(x\in X\) and \(n\ge1\), define the pointwise local
\(\rho\)-Lipschitz constant by
\[
   L^\rho f_\omega^n(x)
   :=
   \lim_{r\downarrow0} L_r^\rho f_\omega^n(x),
\]
where
\[
   L_r^\rho f_\omega^n(x)
   :=
   \sup
   \left\{
      \frac{\rho(f_\omega^n(y),f_\omega^n(z))}{\rho(y,z)}
      :
      y,z\in B_\rho(x,r),\ y\ne z
   \right\}.
\]
Since \(r\mapsto L_r^\rho f_\omega^n(x)\) is non-increasing, the limit
exists in \([0,\infty]\). We use the convention \(\log 0=-\infty\). 

\begin{definition}
\label{def:global-distortion-principle}  
We say that \(f\) satisfies the \emph{global distortion principle} in the metric
\(\rho\) on a subset $R \subset X\times X$, if there exists \(C\ge1/2\) such that, for
every \(n\ge1\), \(0<\alpha\le1\) and 
\((x,y)\in R\),
\[
   \rho(f_\omega^n(x),f_\omega^n(y))^\alpha
   \le
   C
   \left[
      \bigl(L^\rho f_\omega^n(x)\bigr)^\alpha
      +
      \bigl(L^\rho f_\omega^n(y)\bigr)^\alpha
   \right]
   \rho(x,y)^\alpha \quad \text{for \(\mathbb P\)-a.e.~$\omega$.}
   \]
\end{definition}

Now, we establish the main result of this subsection on the two-metric global contraction criterion.

\begin{theorem}
\label{thm:two-metric-global-contraction}
Consider $(X,d,\rho,f)$ as above and assume the following conditions:

\begin{enumerate}[itemsep=0.2cm, label=(\textup{G\arabic*)}]
\item \label{G1}  $f$ is mostly contracting with respect to $\rho$, that is,  $\lambda(f;\rho)<0$;
\item \label{G2} Finite exponential moment: There exists \(\beta>0\) such that
$ \int
   \left(
      \sup_{x\in X}L^\rho f_\omega(x)
   \right)^\beta
   d\mathbb P<\infty$;
\item  For every \(n\ge1\) and for \(\mathbb P\)-a.e.~\(\omega\), the map $ x\mapsto L^\rho f_\omega^n(x)$ \label{G3}
is upper semicontinuous on \((X,d)\);
\item \label{G4}  \(f\) satisfies the global
distortion principle in the metric \(\rho\) on a set \( R \subset X\times X\).
\end{enumerate}
Then there exist \(N_0\ge1\) and \(0<\alpha_0\le\beta\) such that for every $0<\alpha \leq \alpha_0$ there is \(q<1\) satisfying
\[
   \int
      \rho(f_\omega^{N_0}(x),f_\omega^{N_0}(y))^\alpha
      d\mathbb P
   \le
   q\,\rho(x,y)^\alpha \quad \text{for every \((x,y)\in R\)}.
\]
\end{theorem}

We split the proof into several lemmas.

\begin{lemma}
\label{lem:sufficient-exponential-moment}
Assume that there exists \(\beta>0\) such that
$\int
   \left(
      \sup_{x\in X}L^\rho f_\omega(x)
   \right)^\beta\,
   d\mathbb P
   <\infty$.
Then, 
\[
   \int   \left(
      \sup_{x\in X}L^\rho f_\omega^n(x)
   \right)^\beta \,
   d\mathbb P
   <\infty \quad \text{for every \(n\ge1\)}.
\]
In particular,
\[
   \sup_{x\in X}
   \int
      \bigl(L^\rho f_\omega^n(x)\bigr)^\beta
      \,d\mathbb P
   <\infty \quad \text{for every \(n\ge1\)}.
\]
\end{lemma}

\begin{proof}
The chain rule for local \(\rho\)-Lipschitz constants gives
\[
   L^\rho f_\omega^n(x)
   \le
   \prod_{j=0}^{n-1}
   L^\rho f_{\omega_j}
      \bigl(f_{\omega}^{j}(x)\bigr).
\]
Consequently,
\[
   \sup_{x\in X}L^\rho f_\omega^n(x)
   \le
   \prod_{j=0}^{n-1}
   \sup_{x\in X}L^\rho f_{\omega_j}(x).
\]
Raising to the power \(\beta\), integrating, and using the independence of
the Bernoulli coordinates, we get
\[
   \int
   \left(
      \sup_{x\in X}L^\rho f_\omega^n(x)
   \right)^\beta\,
   d\mathbb P
   \le
   \left[
      \int
      \left(
         \sup_{x\in X}L^\rho f_\omega(x)
      \right)^\beta \,
      d\mathbb P
   \right]^n
   <\infty.
\]
This completes the proof.
\end{proof}

\begin{lemma}
\label{lem:usc-after-integration}
Fix \(n\ge1\). Assume that, for \(\mathbb P\)-a.e.~\(\omega\), the map
$x\mapsto L^\rho f_\omega^n(x)$ is upper semicontinuous on \((X,d)\). Assume also that there exists
\(\beta>0\) such that
$
   \int
  (
      \sup_{x\in X} L^\rho f_\omega^n(x)
  )^\beta \,
   d\mathbb P<\infty$.  
Then the following maps are upper semicontinuous on \((X,d)\):

\begin{enumerate}
\item for every \(0<\alpha\le\beta\),
\[
   x\longmapsto
   \int
      \bigl(L^\rho f_\omega^n(x)\bigr)^\alpha
      \, d\mathbb P;
\]

\item the extended real-valued map
\begin{equation} \label{eq:phin}
       \phi_n(x)
   :=
   \int
      \log L^\rho f_\omega^n(x)\,d\mathbb P,
\end{equation}
where we use the convention \(\log 0=-\infty\).
\end{enumerate}
\end{lemma}

\begin{proof}
Let \(x_k\to x\) in \((X,d)\). We prove the two assertions separately.

First fix \(0<\alpha\le\beta\). Since \ the map
$x\mapsto L^\rho f_\omega^n(x)$ 
is upper semicontinuous for \(\mathbb P\)-a.e.~\(\omega\), and \(t\mapsto t^\alpha\) is increasing and
continuous on \([0,\infty)\), we have
\[
   \limsup_{k\to\infty}
   \bigl(L^\rho f_\omega^n(x_k)\bigr)^\alpha
   \le
   \bigl(L^\rho f_\omega^n(x)\bigr)^\alpha \quad \text{for \(\mathbb P\)-a.e.~\(\omega\).}
\]
Moreover,
\[
   \bigl(L^\rho f_\omega^n(x_k)\bigr)^\alpha
   \le
   1+
   \left(
      \sup_{y\in X}L^\rho f_\omega^n(y)
   \right)^\beta 
\]
and the right-hand side is integrable by assumption. Therefore, by the
reverse Fatou lemma,
\[
\begin{aligned}
   \limsup_{k\to\infty}
   \int
      \bigl(L^\rho f_\omega^n(x_k)\bigr)^\alpha
      \, d\mathbb P
   &\le
   \int
      \limsup_{k\to\infty}
      \bigl(L^\rho f_\omega^n(x_k)\bigr)^\alpha
      \, d\mathbb P                                    \le
   \int
      \bigl(L^\rho f_\omega^n(x)\bigr)^\alpha
      \, d\mathbb P.
\end{aligned}
\]
This proves upper semicontinuity of the averaged \(\alpha\)-moment.

We now prove the logarithmic statement. Since \(t\in [0,\infty)\mapsto \log t\in [-\infty,\infty)\) is
increasing and continuous as a map, 
the upper semicontinuity of \(x\mapsto L^\rho f_\omega^n(x)\) gives
\[
   \limsup_{k\to\infty}
   \log L^\rho f_\omega^n(x_k)
   \le
   \log L^\rho f_\omega^n(x) \quad \text{
for \(\mathbb P\)-a.e.~\(\omega\).}
\]
To apply the reverse Fatou lemma,  observe that for every \(t\ge0\),
$\log^+ t\le \frac{1}{\beta}t^\beta$, and hence
\[
   \log L^\rho f_\omega^n(x_k)
   \le
   \log^+ L^\rho f_\omega^n(x_k)
   \le
   \frac1\beta
   \left(
      \sup_{y\in X}L^\rho f_\omega^n(y)
   \right)^\beta.
\]
Since the right-hand side is integrable by assumption, the reverse
Fatou lemma gives
\[
\begin{aligned}
   \limsup_{k\to\infty}
   \phi_n(x_k)
   &=
   \limsup_{k\to\infty}
   \int
      \log L^\rho f_\omega^n(x_k)
      \,d\mathbb P                                    \\
   &\le
   \int
      \limsup_{k\to\infty}
      \log L^\rho f_\omega^n(x_k)
      \,d\mathbb P                                    \le
   \int
      \log L^\rho f_\omega^n(x) \,
      d\mathbb P
   =
   \phi_n(x).
\end{aligned}
\]
Thus \(\phi_n\) is upper semicontinuous as an extended real-valued
function. 
\end{proof}

We use the following finite-time logarithmic growth quantity:
\[
   \Lambda^\rho_{\max}
   := \lim_{n\to \infty}
   \frac1n
   \max_{x\in X}
\phi_n(x)
=
   \inf_{n\ge1}
   \frac1n
   \max_{x\in X}
   \int
      \log L^\rho f_\omega^n(x)\,d\mathbb P.
\]
By assumptions~\ref{G2} and \ref{G3} and the previous lemmas,  $(\phi_n)_{n\geq 1}$ are upper semicontinuous functions on $(X,d)$ and thus the maximum above is justified. The existence of the limit and its expression as an infimum follows from Fekete's lemma. 

Following~\cite{BM}, we also introduce the metric maximal Lyapunov exponents
$$ 
\lambda(\mu;\rho):=\lim_{n\to \infty} \frac{1}{n} \int \log L^\rho f_\omega^n(x) \, d\mathbb P d\mu
$$
and
$$
\lambda(f;\rho):=\sup \{ \lambda(\mu;\rho) : \mu\in \mathcal{I}(f)\} = \max \{ \lambda(\mu;\rho) : \mu\in \mathcal{I}(f) \ \text{ergodic}\}. 
$$

\begin{lemma}
\label{lem:mostly-to-finite-time-negativity}
Under the assumptions~\ref{G2} and \ref{G3}, we have that 
$\Lambda_{\max}^{\rho} = \lambda(f;\rho)$.  
\end{lemma}


\begin{proof}
We first prove that the sequence of functions \((\phi_n)_{n\ge1}\) given in~\eqref{eq:phin} is \(P\)-subadditive where $P$ is the Koopman operator associated with $f$. That is,
$$
P^n\varphi = \int \varphi(f^n_\omega(x)) \, d\mathbb{P}, \quad \varphi\in C(X)  \ \text{and } n\geq 1.  
$$
The chain
rule for local \(\rho\)-Lipschitz constants gives
\[
   L^\rho f_\omega^{n+m}(x)
   \le
   L^\rho f_{\sigma^n\omega}^{m}
      \bigl(f_\omega^n(x)\bigr)
   \cdot
   L^\rho f_\omega^n(x).
\]
Taking logarithms,
\[
   \log L^\rho f_\omega^{n+m}(x)
   \le
   \log L^\rho f_\omega^n(x)
   +
   \log L^\rho f_{\sigma^n\omega}^{m}
      \bigl(f_\omega^n(x)\bigr).
\]
Integrating in \(\omega\), and using that the future
\(\sigma^n\omega\) is independent of the first \(n\) coordinates in the
Bernoulli setting, we obtain
\[
\begin{aligned}
   \phi_{n+m}(x)
   \le
   \phi_n(x)
   +
   \int_\Omega
      \phi_m(f_\omega^n(x))\,d\mathbb P(\omega)   =
   \phi_n(x)+P^n\phi_m(x).
\end{aligned}
\]
Thus \((\phi_n)\) is \(P\)-subadditive.

Since \((X,d)\) is compact, \(P:C(X)\to C(X)\) is Markov, \cite[Thm~C]{BM} applies to the \(P\)-subadditive sequence of extended-real valued functions
\((\phi_n)_{n\ge 1}\) which are upper semicontinuous  on $(X,d)$  by Lemmas~\ref{lem:sufficient-exponential-moment} and~\ref{lem:usc-after-integration} and assumptions~\ref{G2} and~\ref{G3}. Hence
\begin{equation*}
\lambda(f;\rho)=\sup_{\mu\in\mathcal{I}(f)}\lambda(\mu;\rho)
   =\lim_{n\to \infty}
   \frac1n\max_{x\in X}\phi_n(x)=
   \inf_{n\ge1}
   \frac1n\max_{x\in X}\phi_n(x) = \Lambda_{\max}^{\rho}. \qedhere
\end{equation*}
\end{proof}

The main local criterion below assumes directly that  \(\Lambda^\rho_{\max}<0\). 

\begin{proposition}
\label{prop:uniform-moment-contraction}
Let \(N\ge1\) and  \(\beta>0\). Assume that the following conditions hold:
\begin{enumerate}[itemsep=0.2cm, leftmargin=1.5cm, label=\textup{(H\arabic*)}] 
\item \label{H1} Negative $\rho$-maximal Lyapunov exponent: $
   \sup_{x\in X}
   \int
      \log L^\rho f_\omega^N(x)\,d\mathbb P <0$.

\item  \label{H2} Finite exponential moment:
$
   \sup_{x\in X}
   \int
      \bigl(L^\rho f_\omega^N(x)\bigr)^\beta
      \,d\mathbb P
   <\infty$. 

\item \label{H3}  For every \(0<\alpha\le\beta\), the map
$
   x \mapsto
   \int
      \bigl(L^\rho f_\omega^N(x)\bigr)^\alpha
      \,d\mathbb P
$
is upper semicontinuous on $(X,d)$. \\[-0.4cm]
\end{enumerate}
Then there exist \(0<\alpha_0\le\beta\) such that for every $0<\alpha \leq \alpha_0$, there is \(\kappa<1\) satisfying
\[
   \sup_{x\in X}
   \int
      \bigl(L^\rho f_\omega^N(x)\bigr)^\alpha
      \,d\mathbb P
   \le \kappa .
\]
\end{proposition}

\begin{proof}
Fix \(x\in X\) and put $Y_x(\omega):=L^\rho f_\omega^N(x)$.
For \(0<\alpha\le\beta/2\), define
$
   \Psi_{x,\alpha}(\omega)
   :=
   \frac{Y_x(\omega)^\alpha-1}{\alpha}$. 

For every \(t>0\), we have that  $(t^\alpha-1)/{\alpha}\to \log t$  as $\alpha \to 0^+$.  We need an upper integrable majorant. For \(t\geq 1\), since 
\[
   \frac{t^\alpha-1}{\alpha}
   =
   \int_0^1 t^{s\alpha}\log t\,ds
   \le
   t^\alpha\log t \quad \text{and}  \quad 
   \log t\le \frac{2}{\beta}t^{\beta/2}
\]
we get that for \(0<\alpha\le\beta/2\),
\begin{equation}  \label{eq:beta-t}
       \frac{t^\alpha-1}{\alpha}
   \le
   \frac{2}{\beta}t^\beta. 
\end{equation}
Moreover, since  for \(0<t\le 1\), we have trivially that $(t^\alpha-1)/{\alpha}\le0  \leq \frac{2}{\beta}t^\beta$, and thus~\eqref{eq:beta-t} holds for every $t >0$.  Applying this to \(t=Y_x(\omega)\), assumption \ref{H2} gives an
integrable upper majorant for \(\Psi_{x,\alpha}\). Therefore, by the
reverse Fatou lemma, and \ref{H1}, 
\[
\begin{aligned}
   \limsup_{\alpha \to 0^+}
   \int
      \Psi_{x,\alpha}\,d\mathbb P
   \le
   \int
      \limsup_{\alpha \to 0^+}
      \Psi_{x,\alpha}\,d\mathbb P    =
   \int
      \log Y_x\,d\mathbb P <0.
\end{aligned}
\]
Thus, since $\int \Psi_{x,\alpha} \, d\mathbb{P} =\frac{1}{\alpha} (\int Y^\alpha_x \, d\mathbb{P} -1)$, for each \(x\in X\), there exists
\(0<\alpha(x)\le\beta/2\) such that $\int
      Y_x^{\alpha(x)}\,d\mathbb P  <1$. 
Choose \(q_x<1\) such that
\[
   \int 
         \bigl(L^\rho f_\omega^N(x)\bigr)^{\alpha(x)}
         \,d\mathbb P=\int
      Y_x^{\alpha(x)}\,d\mathbb P
   < q_x <1.
\]
By \ref{H3}, the set
\[
   U_x
   :=
   \left\{
      y\in X:
      \int 
         \bigl(L^\rho f_\omega^N(y)\bigr)^{\alpha(x)}
         \,d\mathbb P
      < q_x
   \right\}
\]
is an open neighbourhood of \(x\) in the topology of \(d\). Since
\((X,d)\) is compact, choose \(x_1,\dots,x_\ell\in X\) such that $X=U_{x_1} \cup \dots \cup U_{x_\ell}$ and set
\[
   \alpha_0:=\min_{1\le i\le\ell}\alpha(x_i)\in (0,\beta]. 
\]
Now, for every $0<\alpha\leq \alpha_0$, we define $ p_i:=\frac{\alpha(x_i)}{\alpha}\ge1$  for $i=1,\dots,\ell$. 
If \(y\in U_{x_i}\), then H\"older's inequality gives
\[
\begin{aligned}
   \int
      \bigl(L^\rho f_\omega^N(y)\bigr)^\alpha
      \,d\mathbb P
   \le
   \left(
      \int
         \bigl(L^\rho f_\omega^N(y)\bigr)^{\alpha p_i}
         \,d\mathbb P
   \right)^{1/p_i}                                      
   =
   \left(
      \int
         \bigl(L^\rho f_\omega^N(y)\bigr)^{\alpha(x_i)}
         \,d\mathbb P
   \right)^{1/p_i}                                      
   <
   q_{x_i}^{1/p_i}.
\end{aligned}
\]
Therefore
\[
   \sup_{y\in X}
   \int_\Omega
      \bigl(L^\rho f_\omega^N(y)\bigr)^\alpha
      \,d\mathbb P
   \le
   \max_{1\le i\le\ell}q_{x_i}^{1/p_i}
   =:\kappa<1.
\]
This proves the proposition.
\end{proof}

The last preliminary result that we need is the following amplification result:

\begin{lemma}
\label{lem:two-metric-amplification}
Let  \(N\ge1\), \(0<\alpha\le1\), and \(\kappa<1\) be such that $\sup_{x\in X}
   \int
      \bigl(L^\rho f_\omega^N(x)\bigr)^\alpha
      \,d\mathbb P      
   \le\kappa$.
Then, 
\[
   \sup_{x\in X}
   \int
      \bigl(L^\rho f_\omega^{mN}(x)\bigr)^\alpha
      \, d\mathbb P
   \le \kappa^m \quad \text{for every \(m\ge1\).}
\]
\end{lemma}

\begin{proof}
The chain rule gives
$L^\rho f_\omega^{(m+1)N}(x)
   \le
   L^\rho f_{\sigma^{mN}\omega}^{N}
      \bigl(f_\omega^{mN}(x)\bigr) \cdot 
   L^\rho f_\omega^{mN}(x)$.
Raising to the power \(\alpha\), integrating, and using independence of
the future block, we obtain
\[
\begin{aligned}
   \int
      \bigl(L^\rho f_\omega^{(m+1)N}(x)\bigr)^\alpha
      d\mathbb P      &\le
   \int
      \bigl(L^\rho f_\omega^{mN}(x)\bigr)^\alpha
      \left[
         \int
            \bigl(
               L^\rho f_{\omega'}^N(f_\omega^{mN}(x))
            \bigr)^\alpha 
            d\mathbb P(\omega')
      \right]
      d\mathbb P                                   \le
   \kappa
   \int
      \bigl(L^\rho f_\omega^{mN}(x)\bigr)^\alpha
      d\mathbb P.
\end{aligned}
\]
The result follows by induction.
\end{proof}

Finally, we are ready to prove the theorem:

\begin{proof}[Proof of Theorem~\ref{thm:two-metric-global-contraction}]
By \ref{G1} and Lemma~\ref{lem:mostly-to-finite-time-negativity},
there exist \(N\ge1\) and \(\eta>0\) such that
\begin{equation} \label{eq:eta}
   \sup_{x\in X}
   \int
      \log L^\rho f_\omega^N(x)\,d\mathbb P
   \le -\eta .
\end{equation}
By \ref{G2} and Lemma~\ref{lem:sufficient-exponential-moment},
\[
   \int   \left(
      \sup_{x\in X}L^\rho f_\omega^N(x)
   \right)^\beta \,
   d\mathbb P
   <\infty \quad \text{and} \quad 
   \sup_{x\in X}
   \int
      \bigl(L^\rho f_\omega^N(x)\bigr)^\beta
      \,d\mathbb P
   <\infty.
\]
Hence, by \ref{G3} and Lemma~\ref{lem:usc-after-integration}, for every
\(0<\alpha\le\beta\), the map
$x\mapsto
   \int
      \bigl(L^\rho f_\omega^N(x)\bigr)^\alpha
      \, d\mathbb P$
is upper semicontinuous on $(X,d)$. Thus the assumptions of
Proposition~\ref{prop:uniform-moment-contraction} are satisfied. Consequently,
there exist \(0<\alpha\le\beta\) and \(\kappa<1\) such that
\[
   \sup_{x\in X}
   \int
      \bigl(L^\rho f_\omega^N(x)\bigr)^\alpha
      \,d\mathbb P
   \le \kappa.
\]
Let \(C\ge1/2\) be the constant in the global distortion principle. By
Lemma~\ref{lem:two-metric-amplification}, after replacing \(N\) by
\(N_0=mN\) for some \(m\ge1\), we may assume that
\[
   2C
   \sup_{x\in X}
   \int
      \bigl(L^\rho f_\omega^{N_0}(x)\bigr)^\alpha
      \,d\mathbb P
   <1.
\]
Set
\[
   \kappa_0
   :=
   \sup_{x\in X}
   \int
      \bigl(L^\rho f_\omega^{N_0}(x)\bigr)^\alpha
      d\mathbb P,
   \qquad
   q:=2C\kappa_0<1.
\]
Let \((x,y)\in R\) be a point where~\ref{G4} holds. By the global distortion principle,
\[
   \rho(f_\omega^{N_0}(x),f_\omega^{N_0}(y))^\alpha
   \le
   C
   \left[
      \bigl(L^\rho f_\omega^{N_0}(x)\bigr)^\alpha
      +
      \bigl(L^\rho f_\omega^{N_0}(y)\bigr)^\alpha
   \right]
   \rho(x,y)^\alpha .
\]
Integrating gives
\[
   \int
      \rho(f_\omega^{N_0}(x),f_\omega^{N_0}(y))^\alpha
      \, d\mathbb P
   \le
   C(\kappa_0+\kappa_0)\rho(x,y)^\alpha    =
   q\,\rho(x,y)^\alpha \quad \text{for every $(x,y)\in R$}.
\]
This proves the theorem.
\end{proof}

\subsubsection{Two-metric local contraction criterion for general cocycles}
\label{subsubsec:two-metric-local-general-cocycles}

The previous global criterion used the Bernoulli structure in two places:
to pass from mostly contracting to finite-time negativity through the
annealed Markov operator, and to amplify a finite-time moment contraction.
The local contraction statement below is more elementary. It only uses a
single finite time \(N\), and therefore applies to general cocycles.

Let \((\Omega,\mathscr F,\mathbb P,T)\) be a probability space with a
measure-preserving transformation. Consider a measurable skew-product 
$$F(\omega,x)=(T(\omega),f(\omega,x))$$
and denote by $f^n_\omega(x)$ the second coordinate of $F^n(\omega,x)$ for $n \geq 1$. That is, the map
$f:\Omega \times X\to X$ is a measurable cocycle over~\(T:\Omega\to \Omega\), 
\[
   f_\omega^0=\mathrm{id},
   \qquad
   f_\omega^{n+m}
   =
   f_{T^m\omega}^{n}\circ f_\omega^m
   \quad\text{for all }m,n\ge0.
\]
For \(x\in X\), \(r>0\), and \(n\ge1\), define
\[
   L_r^\rho f_\omega^n(x)
   :=
   \sup
   \left\{
      \frac{\rho(f_\omega^n(y),f_\omega^n(z))}{\rho(y,z)}
      :
      y,z\in B_\rho(x,r),\ y\ne z
   \right\},
\]
and
\[
   L^\rho f_\omega^n(x)
   :=
   \lim_{r\downarrow0}L_r^\rho f_\omega^n(x) = \limsup_{\substack{(y,z) \to (x, x) \\ y\not = z} } \frac{\rho(f_\omega^n(y),f_\omega^n(z))}{\rho(y,z)}.
\]

\begin{proposition}
\label{prop:two-metric-local-contraction} Let \((X,d)\) be compact metric space and consider another metric \(\rho\)  on~\(X\).  
Let $f:\Omega \times X \to X$ be a measurable cocycle over a mesure-preserving transformation $T$ of probability space  $(\Omega,\mathscr{F},\mathbb{P})$.  
Assume that there exist \(N\ge1\) and \(\beta>0\) such that:

\begin{enumerate}[itemsep=0.2cm, leftmargin=1.25cm, label=\textup{(H\arabic*')}] \stepcounter{enumi}
\item[\textup{(H1)}]  Negative $\rho$-maximal Lyapunov exponent: $\sup_{x\in X}
   \int
      \log L^\rho f_\omega^N(x)\,d\mathbb P <0$;

\item \label{H2'}   
For every \(x\in X\), there is \(\bar r_x>0\) such that
 $
   \int
      \bigl(L_{\bar r_x}^\rho f_\omega^N(x)\bigr)^\beta
      d\mathbb P
   <\infty$;

\item \label{H3'} For every \(r>0\) and \(0<\alpha\le\beta\),
$x\mapsto
   \int
      \bigl(L_r^\rho f_\omega^N(x)\bigr)^\alpha
      d\mathbb P$
is upper semicontinuous on \((X,d)\). \\[-0.4cm]
\end{enumerate}
Then there exist \(r_0>0\) and $\alpha_0\in (0,1]$, such that for every \(0<\alpha\le\alpha_0\), there exist \(q<1\) sastisfying
\[
   \int
      \rho(f_\omega^N(y),f_\omega^N(z))^\alpha
      d\mathbb P
   \le
   q\,\rho(y,z)^\alpha
\]
whenever \(\rho(y,z)<r_0\).
\end{proposition}

\begin{proof}
By~\ref{H1}, fix \(\eta>0\) as in~\eqref{eq:eta}. For \(x\in X\), put 
$Y_x(\omega):=L^\rho f_\omega^N(x)$.
Exactly as in the proof of
Proposition~\ref{prop:uniform-moment-contraction}, using
\ref{H1} and the pointwise integrability supplied by
\ref{H2'}, for each \(x\in X\) there exists
\(0<\alpha(x)\le\beta/2\) and $0<q_x<1$ such that
\[
   \int
      \bigl(L^\rho f_\omega^N(x)\bigr)^{\alpha(x)}
      d\mathbb P
   <q_x.
\]
By \ref{H2'}, for the fixed point \(x\in X\) there exists
\(\bar r_x>0\) such that
$\int
      \bigl(L_{\bar r_x}^\rho f_\omega^N(x)\bigr)^\beta
      d\mathbb P<\infty$.  
For \(0<r\le \bar r_x\), the monotonicity 
gives $L^\rho f_\omega^N(x)
   \le
   L_r^\rho f_\omega^N(x)
   \le
   L_{\bar r_x}^\rho f_\omega^N(x)$. 
Since \(0<\alpha(x)\le\beta\), we have
\[
   \bigl(L_r^\rho f_\omega^N(x)\bigr)^{\alpha(x)}
   \le
   1+
   \bigl(L_{\bar r_x}^\rho f_\omega^N(x)\bigr)^\beta ,
\]
and the right-hand side is integrable. Hence, by dominated convergence
(or equivalently by continuity of the integral from above),
\[
   \lim_{r\downarrow0}
   \int
      \bigl(L_r^\rho f_\omega^N(x)\bigr)^{\alpha(x)}
      d\mathbb P
   =
   \int
      \bigl(L^\rho f_\omega^N(x)\bigr)^{\alpha(x)}
      d\mathbb P .
\]
Henve, we may choose \(0<r(x)\le \bar r_x\) such that
\[
   \int
      \bigl(L_{r(x)}^\rho f_\omega^N(x)\bigr)^{\alpha(x)}
      d\mathbb P<q_x .
\]
By \ref{H3'}, the set
\[
   U_x
   :=
   \left\{
      y\in X:
      \int_\Omega
         \bigl(L_{r(x)}^\rho f_\omega^N(y)\bigr)^{\alpha(x)}
         d\mathbb P(\omega)
      <q_x
   \right\}
\]
is an open neighbourhood of \(x\) in the topology of \(d\). By compactness, choose \(x_1,\dots,x_\ell\in X\) such that
$X=U_{x_1}\cup \dots \cup U_{x_\ell}$.
Set
\[
   \alpha_0:=\min_{1\le i\le\ell}\alpha(x_i)>0,
   \qquad
   r_0:=\min_{1\le i\le\ell}r(x_i)>0,
\]
and for every $0<\alpha\leq \alpha_0$, 
\[
   p_i:=\frac{\alpha(x_i)}{\alpha}\ge1,
   \qquad
   q:=\max_{1\le i\le\ell}q_{x_i}^{1/p_i}<1.
\]
Let \(y,z\in X\) satisfy \(\rho(y,z)<r_0\). Choose \(i\) such that
\(y\in U_{x_i}\). Since \(r_0\le r(x_i)\), we have
\(z\in B_\rho(y,r(x_i))\). Therefore,
\[
   \rho(f_\omega^N(y),f_\omega^N(z))
   \le
   L_{r(x_i)}^\rho f_\omega^N(y)\cdot\rho(y,z).
\]
Raising to the power \(\alpha\), integrating, and using H\"older's
inequality,
\[
\begin{aligned}
   \int
      \rho(f_\omega^N(y),f_\omega^N(z))^\alpha
      \,d\mathbb P         &\le
   \rho(y,z)^\alpha
   \int
      \bigl(L_{r(x_i)}^\rho f_\omega^N(y)\bigr)^\alpha
      \,d\mathbb P         \\
   &\le
   \rho(y,z)^\alpha
   \left(
      \int
         \bigl(L_{r(x_i)}^\rho f_\omega^N(y)\bigr)^{\alpha(x_i)}
        \, d\mathbb P
   \right)^{1/p_i}                                         <
   q_{x_i}^{1/p_i}\rho(y,z)^\alpha
   \le
   q\,\rho(y,z)^\alpha .
\end{aligned}
\]
This proves the proposition.
\end{proof}

\subsection{Projective estimates}
\label{subsec:projective-estimates} 
Assume the setting and notations in~\S\ref{s:setting-linear}. Denote by $\hat\pi:\hat{\mathcal E}\to X$ the projection of the projective bundle. Given $z=(x,\hat v)\in \hat{\mathcal E}$, we also consider 
vertical tangent space
$T_z^{\mathrm{vert}}\hat{\mathcal E}
   :=
   \ker D\hat\pi(z)
   \cong
   T_{\hat v}P(\mathcal E_x)$. 
We denote by \(d\) any fixed metric compatible with the compact topology of \(\hat{\mathcal E}\). On the other hand, we also consider the
vertical metric
\begin{equation}
\label{eq:def-d-fib}
   d_{\rm vert}((x,\hat v),(y,\hat u))
   :=
   \begin{cases}
      \theta_x(\hat v,\hat u) & \text{if }x=y,\\[0.3em]
      1                 & \text{if }x\ne y,
   \end{cases}
\end{equation}
where \(\theta_x\) is the angular distance on \(P(\mathcal{E}_x)\) induced by the Euclidean norm on \(\mathcal E_x\). Thus \(0\le d_{\rm vert}\le 1\), and
for every \(0<r<1\), the ball
$B_{{\rm vert}}(z,r)$ is contained in the fibre \(P(\mathcal{E}_x)\) whenever \(z=(x,\hat v)\). We recall that $v$ denotes a unitary vector representing $\hat v$. 

Let us additionally assume that the  random bundle isomorphism $A:\Omega \times \mathcal{E}\to\mathcal{E}$ covering the continuous random map $f:\Omega\times X\to X$ also satisfies that the map $A_\omega:\mathcal E\to\mathcal E$ is continuous for $\mathbb{P}$-a.e.~$\omega\in\Omega$. That is, $A:\Omega \times \mathcal{E}\to \mathcal{E}$ is a continuous random map.  
Let  $\hat A: \Omega \times \hat{\mathcal E} \to \hat{\mathcal E}$ be the associated projective random map. Notice that  for \(n\ge1\), the restriction of
\(\hat A_\omega^n= \hat A_{\omega_{n-1}} \circ \dots \circ \hat A_{\omega_0}\) to the fibre \(\hat{\mathcal E}_x=P(\mathcal{E}_x)\) is the projective map induced
by the linear isomorphism
$A_\omega^n(x):\mathcal{E}_x\to \mathcal{E}_{f_\omega^n(x)}$. 
We define the fibrewise local Lipschitz constant by
\[
   L^{\rm vert}\hat A_\omega^n(z)
   :=
   \lim_{r\downarrow0}
   \sup_{\substack{z',z''\in B_{{\rm vert}}(z,r)\\ z'\ne z''}}
   \frac{
      d_{\rm vert}
      (\hat A_\omega^n(z'),\hat A_\omega^n(z''))
   }{
      d_{\rm vert}(z',z'')
   } .
\]
Since \(B_{{\rm vert}}(z,r)\) is contained in the single projective fibre at $z=(x,\hat v)$ for \(0<r<1\), this is exactly the local Lipschitz constant of the projective
fibre map $$\widehat{A_\omega^n(x)}:P(\mathcal{E}_x)\to P(\mathcal{E}_{f_\omega^n(x)})$$ at \(\hat v\). Hence
\[
   L^{\rm vert}\hat A_\omega^n(z) = \|D \widehat{A_\omega^n(x)}(\hat v) \|
   \qquad \text{for $z=(x,\hat v) \in \hat{\mathcal E}$.}
\]   

In what follows we shall use the following standard estimates for projective maps. 

\begin{lemma}
\label{lem:projective-estimates} 
For any \(n\ge1\), it holds that
\[
  \frac{\|\wedge^2 A_\omega^n(x)\|}{\|A_\omega^n(x)\| \cdot \|A_\omega^n(x)v\|}\leq L^{\rm vert}\hat A_\omega^n(z)
   \le
   \frac{\|\wedge^2 A_\omega^n(x)\|}
        {\|A_\omega^n(x)v\|^2} \quad \text{for every \(z=(x,\hat v)\in\hat{\mathcal E}\)}.
\]
Moreover, for every \(z,z'\in\hat{\mathcal E}\) with \(\pi(z)=\pi(z')\),
\[
   d_{\rm vert}
      \bigl(\hat A_\omega^n(z),\hat A_\omega^n(z')\bigr)^\alpha
   \le
   \frac{1}{2}
   \left[
      \bigl(L^{\rm vert}\hat A_\omega^n(z)\bigr)^\alpha
      +
      \bigl(L^{\rm vert}\hat A_\omega^n(z')\bigr)^\alpha
   \right]
   d_{\rm vert}(z,z')^\alpha .
\]
\end{lemma}

\begin{proof}
Let $B:V\to V$ be an invertible linear map of an $n$-dimensional real vector space~$V$. By \cite[Prop.~2.28]{DK16}, the differential of the projective map \(\hat B\) at \(\hat v\) is given, under the identification \(T_{\hat v} P(V)\cong v^\perp\), by
\[
    D\hat B(\hat v)w
    =
    \frac{1}{\lVert Bv \rVert}
    \Pi_{(Bv)^\perp}Bw,
    \qquad w\in v^\perp 
\]
where $\Pi_{(Bv)^\perp}$ is the orthogonal projection onto $(Bv)^\perp$, the subspace orthogonal to $Bv$. Choose orthonormal bases with the first vector $v$ in the domain and $Bv/\lVert Bv \rVert$ in the codomain. Then
\[
    B = \begin{pmatrix} a &  b \\ 0 & C \end{pmatrix} \quad \text{with $a = \lVert Bv \rVert > 0$, $b^*\in\mathbb{R}^{n-1}$  and $C$  the matrix of $\Pi_{(Bv)^{\perp}} B |_{v^{\perp}}$.}
\]
Now
\[
    B^*B = \begin{pmatrix} a^2 & ab \\ (ab)^* & b^*b + C^*C \end{pmatrix}
\]
and the Schur complement of the $1 \times 1$ block $a^2$ is
$(b^*b + C^*C) - (ab)^*(a^2)^{-1}(ab) = C^*C$.
Smith's Schur-complement interlacing theorem~\cite[Theorem 5]{Smith1992Schur} gives that the eigenvalues of $B^*B$ and $C^*C$ satisfy that $\lambda_j(B^*B) \ge \lambda_j(C^*C) \ge \lambda_{j+1}(B^*B)$, $j=1,\dots,n-1$. Taking square roots, the singular values are interlacing as  $$\sigma_j(B) \ge \sigma_j(\Pi_{(Bv)^\perp} B |_{v^\perp}) \ge \sigma_{j+1}(B) \quad j=1,\dots,n-1.
$$
In particular, for $j=1$, this gives 
$\sigma_2(B) \le \lVert \Pi_{(Bv)^\perp} B |_{v^\perp} \rVert$. Since $\lVert D\hat B({\hat v}) \rVert 
    = \lVert Bv \rVert^{-1} \lVert \Pi_{(Bv)^\perp} B |_{v^\perp} \rVert$, 
we obtain 
$ \lVert Bv \rVert^{-1}\sigma_2(B) \le \lVert D(\hat B)_{\hat v} \rVert$. 
Taking into account that $\sigma_2(B)= \|B\|^{-1}\|\wedge^2B\|$ and applying this to $B=A_\omega^n(x)$, we get the first inequality in the statement. The remaining inequalities are standard angular estimates for the
projective action of a linear isomorphism, see
\cite[Lem.~5.6.1 and Lem.~5.6.2]{BM}. 
\end{proof}

\begin{corollary} \label{cor:vertical-laypunov-exponent}
Let $\mu \in \mathcal{I}(f)$. For every $\hat A$-stationary measure $\hat \mu$ projecting on $\mu$,  
$$\lambda_{\rm vert}(\hat\mu):=\lambda(\hat \mu; d_{\rm vert}) 
\geq \lambda_2(\mu)-\lambda_1(\mu).$$
Moreover, if $\hat\mu$ is maximizing, then $\lambda_{\rm vert}(\hat\mu)= \lambda_2(\mu)-\lambda_1(\mu)$.
\end{corollary}
\begin{proof}
    By Lemma~\ref{lem:projective-estimates},  for
every \(n\ge1\),
\begin{align*}
      \log\|\wedge^2 A_\omega^n(x)\|
   &- \log \|A_\omega^n(x)\| - 
   \log\|A_\omega^n(x)v\| \leq \\ &\leq \log L^{\rm vert}\hat A_\omega^n(x,\hat v)
   \le
   \log\|\wedge^2 A_\omega^n(x)\|
   -
   2\log\|A_\omega^n(x)v\|.
\end{align*}
Integrating with respect to \(d\mathbb P(\omega)d\hat\mu(x,\hat v)\),
dividing by \(n\), and taking the limit, we obtain
\[
\begin{aligned}
  \lambda_2(\mu)  &- \lim_{n\to\infty}
   \frac1n \int
      \log\|A_\omega^n(x)v\|
      \,d\mathbb P \,d\hat\mu \leq \\ &\leq \lambda_{\rm vert}(\hat\mu)
   \le
   \lambda_1(\mu)+\lambda_2(\mu) -
   2
   \lim_{n\to\infty}
   \frac1n
   \int
      \log\|A_\omega^n(x)v\|
      \,d\mathbb P \,d\hat\mu.
\end{aligned}
\]
Since the process $(\phi_n)_{n\geq 1}$ where $\phi_n(x,\hat v)=\int \log\|A_\omega^n(x)v\|
      \,d\mathbb P$ is additive with respect to the the annealed Koopman operator $\hat P$ associated with~$\hat A$, we have that 
\[
   \frac1n
   \int
      \log\|A_\omega^n(x)v\|
      \,d\mathbb P\,d\hat\mu
   =
   \int \phi_A\,d\hat\mu   \quad \text{where} \quad \phi_A=\phi_1.
\]
Since $\int \phi_A\,d\hat\mu \leq  \lambda_1(\mu)$, then we get that $\lambda_2(\mu)-\lambda_1(\mu) \leq \lambda_{\rm vert}(\hat\mu)$. Moreover, if $\hat \mu$ is maximizing, that is $\int \phi_A\,d\mu =\lambda_1(\mu)$, then we get the equality concluding the proof of the corollary. 
\end{proof}

\begin{lemma}
\label{lem:projective-vertical-regularity}
For every \(n\ge1\) and for \(\mathbb P\)-a.e.~\(\omega\), the
    map $z\mapsto L^{\rm vert}\hat A_\omega^n(z)$   is continuous on \((\hat{\mathcal E},d)\). Moreover, under condition 
\[
        \int
        \big(\log^+\|A_\omega\|_\infty
        +
        \log^+\|(A_\omega)^{-1}\|_\infty\big)
        \,d\mathbb P<\infty, 
\]
the function $\hat\phi_n(z)
        :=
        \int
        \log L^{\rm vert}\hat A_\omega^n(z)
        \,d\mathbb P$
    is continuous on \((\hat{\mathcal{E}},d)\) and the map $\hat\mu   \mapsto \lambda_{\rm vert}(\hat\mu)$ 
is upper semicontinuous on \(\mathcal{I}(\hat A)\) for the weak$^*$ topology. 
\end{lemma}

\begin{proof}
Fix \(n\ge1\) and \(\omega\) such that \(A^n_\omega\) is a continuous linear bundle map.  For \(z=(x,\hat v)\), the restriction of
\(\hat A_\omega^n\) to the fibre \(\hat{\mathcal{E}}_x=P(\mathcal{E}_x)\) is the projective action of $A_\omega^n(x)$. 
Therefore
\[
        L^{\rm vert}\hat A_\omega^n(x,\hat v)
        =
        \|D\widehat{A_\omega^n(x)}(\hat v)\|.
\]
The derivative of the projective action depends continuously on the linear isomorphism and on the projective direction. Since \(x\mapsto A_\omega^n(x)\) is continuous, it follows that $ z\mapsto L^{\rm vert}\hat A_\omega^n(z)$ is also continuous on \((\hat{\mathcal E},d)\).

It remains to prove the continuity of \(\hat\phi_n\). By Lemma~\ref{lem:projective-estimates}, for \(z=(x,\hat v)\),
\[
   \big(\|A_\omega^n(x)\|\,
   \|(A_\omega^n(x))^{-1}\|\big)^{-1} \leq L^{\rm vert}\hat A_\omega^n(z)
   \le
   \big(\|A_\omega^n(x)\| \,
   \|(A_\omega^n(x))^{-1}\|\big)^2 .
\]
Therefore
\[
   \left|
   \log L^{\rm vert}\hat A_\omega^n(z)
   \right|
   \le
   2\Bigl(
   \log^+\|A_\omega^n(x)\|
   +
   \log^+\|(A_\omega^n(x))^{-1}\|
   \Bigr) \leq 
   2\Bigl(
   \log^+\|A_\omega^n\|_\infty
   +
   \log^+\|(A_\omega^n)^{-1}\|_\infty
   \Bigr)
\]
which is integrable under the logarithmic moment assumption. 
Hence, since the map
\(z\mapsto \log L^{\rm vert}\hat A_\omega^n(z)\) is continuous for
\(\mathbb P\)-a.e.~\(\omega\), dominated convergence implies that $\hat \phi_n$ is also continuous on  $(\hat{\mathcal E},d)$. 

On the other hand,  for every
\(n\ge1\), the map $F_n(\hat\mu)
        :=
        \frac1n
        \int \hat\phi_n\,d\hat\mu$
is continuous on \(\mathcal{I}(\hat A)\) for the weak$^*$  topology. Thus, 
$\lambda_{\rm vert}(\hat\mu)
        =
        \lim_{n\to\infty}F_n(\hat\mu)
        =
        \inf_{n\ge1}F_n(\hat\mu)$
is also upper semicontinuous, being the infimum of a countable family of upper semicontinuous functions. 
\end{proof}

\begin{lemma}
\label{lem:verify-two-metric-projective}
Assume that for some \(\beta>0\),
\[
   \int
   \bigl(\sup_{z\in\hat{\mathcal E}}
   L^{\rm vert}\hat A_\omega(z)\bigr)^\beta
   \,d\mathbb P<\infty .
\]
Then $(\hat{\mathcal E}, d, d_{\rm vert},\hat A)$  satisfies conditions
\textup{\ref{G2}}, \textup{\ref{G3}}, and \textup{\ref{G4}} of
Theorem~\ref{thm:two-metric-global-contraction} on \[ R=\hat{\mathcal E}\times_{\hat\pi} \hat{\mathcal E} :=\{(z,z')\in \hat{\mathcal E} \times \hat{\mathcal E}: \hat{\pi}(z)=\hat{\pi}(z')\}.\]
\end{lemma}

\begin{proof} Let $\rho=d_{\rm vert}$. 
Condition \textup{\ref{G2}} is exactly the assumed vertical exponential
moment, because the \(\rho\)-local Lipschitz constant of \(\hat A_\omega\)
is $L^\rho\hat A_\omega(z)
   =
   L^{\rm vert}\hat A_\omega(z)$. Also, condition \textup{\ref{G3}} follows from  Lemma~\ref{lem:projective-vertical-regularity}.  Finally, for every \(n\ge1\), 
\(0<\alpha\le1\), and pair $(z,z')\in R$,
one has that $\hat A$ satisfies the global distortion principle in the metric $\rho$ on $R$ from Lemma~\ref{lem:projective-estimates}. Thus, condition~\ref{G4} holds. 
\end{proof}

We shall also use the following elementary observation, which will be applied fibrewise to the projective distance \(d_{\rm vert}\): a contraction estimate in one H\"older exponent automatically yields contraction for every smaller exponent.

\begin{lemma}
\label{lem:contraction-smaller-exponents}
Let \(g: \Omega \times Z\to Z\) be a
random map on a compact metric space $(Z,\delta)$. Assume that, for some \(0<\alpha_0\le1\) and \(q_0<1\),
\[
   \int \delta(g_\omega (z),g_\omega (z'))^{\alpha_0}\,d\mathbb P
   \le
   q_0\,\delta(z,z')^{\alpha_0}
   \qquad\text{for all }z,z'.
\]
Then, for every \(0<\alpha\le\alpha_0\),
\[
   \int \delta(g_\omega(z),g_\omega(z'))^{\alpha}\,d\mathbb P
   \le
   q_0^{\alpha/\alpha_0}\,\delta(z,z')^\alpha. 
\]
\end{lemma}

\begin{proof}
Put \(Y(\omega)=\delta(g_\omega(z),g_\omega (z)')^{\alpha_0}\). Since $0<\alpha/\alpha_0\leq 1$,
\(t\in [0,\infty)\mapsto t^{\alpha/\alpha_0}\) is concave, and thus, Jensen's inequality gives
\[
   \int \delta(g_\omega(z),g_\omega(z'))^{\alpha}\,d\mathbb P=\int Y^{\alpha/\alpha_0}\,d\mathbb P
   \le
   \left(\int Y\,d\mathbb P\right)^{\alpha/\alpha_0}
   \le
   q_0^{\alpha/\alpha_0}\rho(z,z')^\alpha . \qedhere
\]
\end{proof}

\subsection{Vertical spectral gap for continuous random linear bundle isomorphisms}  With the previous notations, the
quotient ${\mathscr C}^{\alpha}(\hat{\mathcal E})/\mathcal B_{\hat{\pi}}$ where $\mathcal B_{\hat\pi}=\hat{\pi}^*(B(X))$, 
removes the observables which depend only on the base. Therefore its dual
does not act on arbitrary signed measures on \(\hat{\mathcal E}\), but naturally on
signed measures whose projection to \(X\) is zero. 
Let $\hat\zeta$ be such a finite signed measure on $\hat{\mathcal E}$ with $\pi_*\hat\zeta=0$. Define 
\[
   \|\hat\zeta\|_{\alpha,{\rm vert} }^*
   :=
   \sup
   \left\{
      \left|
         \int\Phi\,d\zeta
      \right|:
      \Phi\in{\mathscr C}^{\alpha}(\hat{\mathcal E}),\
      [\Phi]_{\alpha,{\rm vert} }\le1
   \right\}.
\] 
This is the dual norm associated with the vertical quotient
${\mathscr C}^{\alpha}(\hat{\mathcal{E}})/\mathcal B_{\hat\pi}$. We also define
\[
\operatorname{Lip}_{\mathrm{vert}}(\hat A_\omega)
:=
\sup_{\substack{z\neq z'\\ \hat\pi z=\hat\pi z'}}
\frac{
d_{\mathrm{vert}}(\hat A_\omega z,\hat A_\omega z')
}{
d_{\mathrm{vert}}(z,z')
}
=
\sup_{z\in\hat{\mathcal E}}
L^{\mathrm{vert}}\hat A_\omega(z)
\]
and denote by $\hat P$ the annealed Koopman operator associated with $\hat A$.

\begin{thm} \label{mainthm:equivalence-spectral-gap-linear} Let  $A:\Omega \times \mathcal E \to \mathcal E$ be a continuous random linear bundle isomorphism covering a continuous random map $f:\Omega \times X \to X$ of a compact metric space $X$ satisfying for some $\beta>0$ the integrability conditions 
\begin{equation} \label{eq:moment-condition-thmB-linear}
    \int \log^+\|A_\omega\|_\infty + \log^+\|(A_\omega)^{-1}\|_\infty \,d\mathbb P<\infty
\quad \text{and} \quad 
\int \bigl(\mathrm{Lip}_{\rm vert}(\hat A_\omega)\bigr)^\beta \,d\mathbb P<\infty.
\end{equation}
Assume that $\lambda_1(\mu) >\lambda_2(\mu)$ for every $\mu \in\mathcal{I}(f)$. Then, the following are equivalent:
    \begin{enumerate}
        \item $A$ is  quasi-irreducible: $\lambda_-(\mu)=\lambda_1(\mu)$ for every $\mu \in \mathcal{I}(f)$.
        \item $\hat A$ is vertical mostly contracting: $\lambda_{\rm vert}(\nu) <0$ for every $\nu \in \mathcal{I}(\hat A)$.
        \item $\hat A$ is vertical contracting on average: there are $0<\alpha\leq 1$, $q<1$ and $n\geq 1$ such that
        $$\int d_{\rm vert}(\hat A^n_\omega(z),\hat A^n_\omega(z'))^\alpha \, d\mathbb{P} \leq q \, d_{\rm vert}(z,z')^\alpha \quad \text{with \ $\hat\pi(z)=\hat\pi(z')$.}$$
        \item $\hat P$ has vertical spectral gap on ${\mathscr C}^{\alpha}(\hat{\mathcal E})$ 
        for any $\alpha>0$ small enough:  there exist \(C>0\) and \(\theta<1\) such that for every $\Phi \in {\mathscr C}^{\alpha}(\hat{\mathcal E})$,
\[
   \|{\hat P}^{\,n}[\Phi]\|_{\alpha,\mathrm{vert}}
   \le
   C\theta^n\|[\Phi]\|_{\alpha,\mathrm{vert}}
   \qquad
   \text{for } n\ge0.
\]
        \item $\hat P$ has vertical spectral gap on the dual of  ${\mathscr C}^{\alpha}(\hat{\mathcal E})$ 
        for any $\alpha>0$ small enough: there
exist \(C>0\) and \(\theta<1\) such that, for every
 finite signed measure \(\hat\zeta\) on $\hat X$ with $\pi_*\hat\zeta =0$,
\[
   \|(\hat P^*)^n\zeta\|_{\alpha,{\rm vert} }^*
   \le
   C\theta^n\|\zeta\|_{\alpha,{\rm vert} }^* \quad \text{for  \(n\ge0\)}.
\]
        \item $\mathcal{I}_\mu(\hat A)$ is a singleton for every  $\mu\in \mathcal{I}(f)$.
     \end{enumerate}
Moreover,   $\hat{A}$ is vertical mostly contracting if and only if $A$ is quasi-irreducible and $\lambda_{1}(\mu)>\lambda_2(\mu)$ for every $\mu\in \mathcal{I}(f)$.
     
\end{thm}

\begin{proof} We prove the implications: 
\[
   \textup{(i)}
   \Longrightarrow
   \textup{(ii)}
   \Longrightarrow
   \textup{(iii)}
   \Longrightarrow
   \textup{(iv)}
   \Longrightarrow
   \textup{(v)}
   \Longrightarrow
   \textup{(vi)}
   \Longrightarrow
   \textup{(i)}.
\]
\noindent
\textup{(i)} \(\Longrightarrow\) \textup{(ii)}.
Let \(\hat\mu\) be an  $A$-stationary measure and put $\mu:=\pi_*\hat\mu$. 
By \textup(i), i.e., quasi-irreducibility,
$\int\phi_A\,d\hat\mu =\lambda_1(\mu)$. That is, $\hat\mu$ is maximizing. Hence Corollary~\ref{cor:vertical-laypunov-exponent} we have $\lambda_{\rm vert}(\hat \mu)=\lambda_2(\mu)-\lambda_1(\mu)<0$ because of $\lambda_1(\mu)>\lambda_2(\mu)$.  Thus \textup{(ii)} holds.

\medskip
\noindent
\medskip
\noindent
\textup{(ii)} \(\Longrightarrow\) \textup{(iii)}.
We apply Theorem~\ref{thm:two-metric-global-contraction} to $(\hat{\mathcal E}, d, d_{\rm vert},\hat A)$ and $R=\{(z,z'):\pi(z)=\pi(z')\}$.

By Lemma~\ref{lem:projective-vertical-regularity}, $\nu \mapsto \lambda_{\rm vert}(\nu)$ is upper semicontious on the compact set $\mathcal{I}(\hat A)$ with respect to the weak$^*$ topology. Hence the suppremum is attained in a measure $\nu_*\in \mathcal{I}(\hat A)$. This implies, 
$$\lambda_{\rm vert}(\hat A):= \sup_{\nu \in \mathcal{I}(\hat A)}\lambda_{\rm vert}(\nu) = \lambda_{\rm vert}(\nu_*)<0.
$$
This is exactly Condition~\ref{G1}. Conditions \ref{G2}, \ref{G3}, and \ref{G4} are verified in
Lemma~\ref{lem:verify-two-metric-projective}. Hence there exist
\(N\ge1\), \(0<\alpha_0\le\beta\), and, for every
\(0<\alpha\le\alpha_0\), a constant \(q<1\) such that
\begin{equation} \label{eq:itemiii}
       \int
   d_{\rm vert}
      (\hat A_\omega^N(z),\hat A_\omega^N(z'))^\alpha
   \,d\mathbb P
   \le
   q\,d_{\rm vert}(z,z')^\alpha
\end{equation}
whenever \(\pi(z)=\pi(z')\). This is \textup{(iii)}.

\medskip
\noindent
\textup{(iii)} \(\Longrightarrow\) \textup{(iv)}. Assume that the fiberwise contraction on average holds for some \(N\ge1\), \(0<\alpha_0\), and
\(q_0<1\). By Lemma~\ref{lem:contraction-smaller-exponents}, we have that~\eqref{eq:itemiii} holds for any $0<\alpha \leq \alpha_0$ with $q=q_0^{\alpha/\alpha_0}<1$. 
Let \(\Phi\in{\mathscr C}^{\alpha}(\hat{\mathcal E})\), and take
\(z,z'\in\hat{\mathcal E}\) with \(\pi(z)=\pi(z')\). Then
\[
\begin{aligned}
   \big|\hat P^N\Phi(z)-\hat P^N\Phi(z')\big|
   &\le
   \int
      \big|\Phi(\hat A_\omega^N(z))
        -\Phi(\hat A_\omega^N(z'))\big|
      \,d\mathbb P \\ &\le
   [\Phi]_{\alpha,{\rm vert} }
   \int
      d_{\rm vert}
         (\hat A_\omega^N(z),\hat A_\omega^N(z'))^\alpha
      \,d\mathbb P                           
   \le
   q \cdot [\Phi]_{\alpha,{\rm vert} }\, 
   d_{\rm vert}(z,z')^\alpha.
\end{aligned}
\]
Therefore $[\hat P^N\Phi]_{\alpha,{\rm vert} }
   \le
   q[\Phi]_{\alpha,{\rm vert} }$. 
Passing to the quotient, this gives
$
   \|{\hat P}_\alpha^N[\Phi]\|_{\alpha,{\rm vert} }
   \le
   q\|[\Phi]\|_{\alpha,{\rm vert} }$.
Hence, the spectral radius of $\hat P: {\mathscr C}^{\alpha}(\hat{\mathcal E})/\mathcal{B}_{\hat\pi} \to {\mathscr C}^{\alpha}(\hat{\mathcal E})/\mathcal{B}_{\hat\pi}$ is less than or equal to $q^{1/N}<1$.
Thus \textup{(iv)} holds.

\medskip
\noindent
\textup{(iv)} \(\Longrightarrow\) \textup{(v)}.
Fix \(\alpha>0\) for which $\hat P$ has vertical spectral radius on $\mathscr{C}^\alpha(\hat{\mathcal E})\setminus \mathcal{B}_{\hat\pi}$.  Then there exist \(C>0\) and \(\theta<1\)
such that
$\|\hat P^n[\Phi]\|_{\alpha,{\rm vert} }
   \le
   C\theta^n\|[\Phi]\|_{\alpha,{\rm vert} }
   =
   C\theta^n[\Phi]_{\alpha,{\rm vert} }$ 
for every \(\Phi\in{\mathscr C}^{\alpha}(\hat{\mathcal E})\) and every \(n\ge0\). Let \(\hat\zeta\) be a finite signed measure on $\hat{\mathcal E}$ with $\pi_*\hat  \zeta =0$. Hence, for every
\(\Phi\in{\mathscr C}^{\alpha}(\hat{\mathcal E})\) with
\([\Phi]_{\alpha,{\rm vert} }\le1\),
\[
\begin{aligned}
   \left|
      \int \Phi\,d(\hat P^*)^n\hat\zeta
   \right|
   =
   \left|
      \int\hat P^n\Phi\,d\hat\zeta
   \right|                                             
   \le
   \|\hat P^n[\Phi]\|_{\alpha,{\rm vert} }\,
   \|\hat\zeta\|_{\alpha,{\rm vert} }^*                  
   \le
   C\theta^n\|\hat\zeta\|_{\alpha,{\rm vert} }^* .
\end{aligned}
\]
Since $\pi_*(\hat P^*)^n\hat\zeta=(P^*)^n(\pi_*\hat\zeta)=0$, 
we also have that \((\hat P^*)^n\hat\zeta\) is a finite signed measure projecting on zero and thus, 
taking the supremum over all such \(\Phi\) gives
\[
   \|(\hat P^*)^n\zeta\|_{\alpha,{\rm vert} }^*
   \le
   C\theta^n\|\zeta\|_{\alpha,{\rm vert} }^* .
\]
This concludes \textup{(v)}.

\medskip
\noindent
\textup{(v)} \(\Longrightarrow\) \textup{(vi)}. Let \(\mu\in\mathcal{I}(f)\), and let $\hat\nu,\hat\eta\in\mathcal{I}_\mu(\hat A)$. Set $\hat\zeta:=\hat\nu-\hat\eta$. 
Then \(\pi_*\zeta=0\). Moreover,
since both measures are \(\hat P^*\)-invariant, $(\hat P^*)^n\hat\zeta=\hat\zeta$  for every $n\ge0$.
By \textup{(v)}, 
$$\|\zeta\|_{\alpha,{\rm vert} }^*
   =
   \|(\hat P^*)^n\zeta\|_{\alpha,{\rm vert} }^*
   \le
   C\theta^n\|\zeta\|_{\alpha,{\rm vert} }^* \quad \text{for every $n\ge0$.}
$$
 Letting \(n\to\infty\), we get $ \|\zeta\|_{\alpha,{\rm vert} }^*=0$ and hence \(\zeta=0\). Hence $\hat\nu=\hat\eta$ and therefore \(\mathcal{I}_\mu(\hat A)\) is a singleton.

\medskip
\noindent
\textup{(vi)} \(\Longrightarrow\) \textup{(i)}.
Let \(\mu\in\mathcal{I}(f)\). By the variational formula for
the top Lyapunov exponent,
\[
   \lambda(\mu)
   =
   \sup
   \left\{
      \int \phi_A\,d\hat\nu:
      \hat\nu\in\mathcal{I}_\mu(\hat A)
   \right\},
\]
and the supremum is attained on an ergodic lift. By \textup{(vi)}, the set
\(\mathcal{I}_\mu(\hat A)\) consists of a single measure. Therefore this
unique lift must attain the supremum. Hence every \(\hat A\)-stationary
lift over \(\mu\) is maximizing. Since \(\mu\) was arbitrary, \(A\) is quasi-irreducible.

\medskip
Finally, to conclude the proof of the theorem, notice that by Corollary~\ref{cor:vertical-laypunov-exponent} for every  $\hat A$-stationary measure $\hat \mu$ projecting on $\mu$, one has that $\lambda_{\rm vert}(\hat\mu) \geq \lambda_{2}(\mu)-\lambda_1(\mu)$. Hence, if $\hat A$ is vertical mostly contracting,  we have 
$\lambda_{\rm vert}(\hat \mu)<0$ and thus $ \lambda_2(\mu)<\lambda_1(\mu)$. In particular, from the equivalence between \textup{(i)} and \textup{(ii)} above we conclude that $\hat A$ is vertical mostly contracting if and only if $A$ is quasi-irreducible and $\lambda_1(\mu)>\lambda_2(\mu)$ for every $\mu\in\mathcal{I}(f)$.   
\end{proof}

\begin{remark}[Meaning of the vertical spectral gap]
\label{rem:meaning-vertical-spectral-gap}
Condition \textup{(iv)} says that there exist \(C>0\) and \(\theta<1\) such that  $\|{\hat P}^{\,n}[\Phi]\|_{\alpha,{\rm vert} }
   \le
   C\theta^n\|[\Phi]\|_{\alpha,{\rm vert} }$ for every $n\ge0$. 
If $\hat P\Phi=\lambda\Phi$ with $|\lambda|=1$,
then passing to the quotient gives
${\hat P}[\Phi]=\lambda[\Phi]$. Hence, above inequality forces that $[\Phi]=0$. Therefore \(\Phi\in\mathcal B_{\hat\pi}\), that is, $\Phi=\varphi\circ \hat\pi$
for some bounded Borel function \(\varphi\) on \(X\), and then
$P\varphi=\lambda\varphi$ where $P$ is the annealed Koopman operator of $f$. 
Thus the peripheral eigenfunctions of \(\hat P\) come entirely from the
base operator \(P\). The quotient estimate proves that there are no new
peripheral eigenfunctions in the projective fibres.
Also, condition \textup{(v)} implies that, over each stationary base
measure \(\mu\), there is at most one projective stationary lift. In
other words, the projection map $\mathcal{I}(\hat A)\longrightarrow\mathcal{I}(f)$, $ \hat\nu\longmapsto\pi_*\hat\nu$,
is injective. 

When \(X\) is a singleton, this vertical statement reduces to the usual
spectral gap for projective linear cocycles. Indeed, if \(X=\{\ast\}\),
then
\[
   \hat{\mathcal{E}}=P(\mathbb R^m),
   \qquad
   \mathcal B_{\hat\pi}=\{\text{constant functions}\},
\]
and therefore
\[
   {\mathscr C}^{\alpha}(\hat{\mathcal{E}})/\mathcal B_{\hat\pi}
   =
   C^\alpha(P(\mathbb R^m))/\mathbb R.
\]
The vertical spectral radius says exactly that the projective
Markov operator contracts H\"older observables modulo constants. Together
with the existence of a stationary projective measure \(\hat\nu\), this
is equivalent to the usual spectral gap decomposition
\[
   C^\alpha(P(\mathbb R^m))
   =
   \mathbb R 1_{P(\mathbb{R}^m)}
   \oplus
   \big\{
      \Phi\in C^\alpha(P(\mathbb R^m)):
      \int \Phi\,d\hat\nu=0
   \big\},
\]
with exponential contraction on the zero-average subspace.
\end{remark}

\subsection{Proof of Theorem~\ref{mainthm:equivalence-spectral-gap} and~\ref{thm:cotangent-projective-criterion} and Corollary~\ref{maincor:equivalence-mostly} and~\ref{maincor:equivalence-mostlybis}} 
Let $f$ be the restriction to an invariant compact set $X$ of a $C^1$ random map. Hence, $f:\Omega \times X \to X$ is a continuous random map.  Moreover, we also have continuous random linear bundle isomorphisms $A=Df$ on $\mathcal{E}=T_XM$ and $A=(Df^{-1})^{*}$ on $\mathcal{E}=T^*_XM$. The integrability conditions~\eqref{eq:moment-condition-thmB-linear} correspond exactly with~\eqref{eq:moment-condition-thmB} for the tangent cocycle and with~\eqref{eq:moment-condition-thmBbis} for the inverse cotangent cocycle. Therefore, Theorem~\ref{mainthm:equivalence-spectral-gap} and~\ref{thm:cotangent-projective-criterion} and  Corollary~\ref{maincor:equivalence-mostly} and~\ref{maincor:equivalence-mostlybis} follow by  applying Theorem~\ref{mainthm:equivalence-spectral-gap-linear} to these cocycles.

\section{Continuity of Lyapunov Exponents: Proof of Theorem~\ref{main:thmD} and Theorem~\ref{main:thmDbis}}
\label{sec:proof-continuity}

We again assume the setting and notations in~\S\ref{s:setting-linear}. Thus, $(X,d)$ is a compact metric space and $\pi:\mathcal E\to X$ is a continuous rank-$m$ vector bundle with a continuous Euclidean norm $\|\cdot\|$. We fix once and for all an auxiliary continuous fibrewise isometric embedding
$\iota:\mathcal E\hookrightarrow X\times\mathbb R^N$. This auxiliary embedding is used only to compare linear maps with different
target fibres. More precisely, let $F:\mathcal E\to\mathcal E$ and $G:\mathcal E\to\mathcal E$ be continuous linear bundle isomorphisms covering continuous maps $f:X\to X$ and $g:X\to X$ respectively. 
We then define 
$$
\|F-G\|_\infty := \sup_{x\in X} \| \iota_{f(x)}\circ F_x\circ \iota_x^{-1}\circ P_x - \iota_{g(x)}\circ G_x\circ \iota_x^{-1}\circ P_x\| 
$$
where \(\iota_x^{-1}\) is understood only on the subspace \(\iota_x(\mathcal E_x)\), and the orthogonal projection onto the embedded fibre \(P_x : \mathbb R^N \to \mathcal E_x \) makes the composition defined on all of \(\mathbb R^N\).

Let $\mathrm{L}_{\mathbb{P}}(\pi)$ denote the space of random linear bundle isomorphisms $A:\Omega \times \mathcal{E} \to \mathcal E$  covering some continuous random map which we denote as $A^\pi:\Omega \times X \to X$ and satisfying 
$$
  \int \big(\log^+ \|A_\omega\|_\infty + \log^+ \|(A_\omega)^{-1}\|_\infty\big)\, d\mathbb{P}<\infty.
$$
We endow this space with the metric
$$
 {\rm d}(A,B)=\int \big(d_{C^0}((A^\pi)_\omega,(B^\pi)_\omega) + \|A_\omega-B_\omega\|_{\infty} + \|(A_\omega)^{-1}-(B_\omega)^{-1}\|_{\infty} \big)\, d\mathbb P. 
$$
The topology defined in this way does not depend on the chosen auxiliary
embedding. Indeed, two different embeddings give two continuous metrics on $\mathrm{L}_{\mathbb{P}}(\pi)$, and these metrics induce the same
compact-open topology. On every compact subset of the space of bundle
isomorphisms, they are uniformly equivalent. Consequently, convergence with respect to
one auxiliary embedding is equivalent to convergence with respect to any other.

For a fixed Borel probability measure $\mu$ on $X$, we condier the subspace 
$$
\mathrm{L}_{\mathbb P}(\pi,\mu):=\left \{A\in \mathrm{L}_{\mathbb P}(\pi): \mu \in \mathcal{I}(A^\pi) \right\}.
$$
We study the continuity of the maximal Lyapunov exponent $$\lambda_1(\mu,A):=\lim_{n\to \infty} \frac{1}{n}\int \log \|A^n_\omega(x)\|\,d\mathbb{P}d\mu \quad  \text{and} \quad  \lambda_1(A):=\sup_{\nu \in \mathcal{I}(A^\pi)} \lambda_1(\nu,A)
$$ with respect to $A$ in  $\mathrm{L}_{\mathbb P}(\pi,\mu)$ and $\mathrm{L}_{\mathbb P}(\pi)$ respectively.  To treat both the global case and the fixed base measure case, we first establish a crucial lemma regarding the stability under perturbation of the annealed potential functions 
$$
\phi_n^A(x, \hat v) := \int \log \|A^n_\omega(x)v\| \, d\mathbb{P},  \quad (x,\hat v)\in \hat{\mathcal E}  \ \ \text{and} \ \ n\geq 1.
$$
Notice  that 
$\phi_1^A= \phi_A$ in the notation of~\S\ref{s:setting-linear}.

\begin{lem}
\label{lem:uniform-conv-phi}
Let $A, (A_k)_{k\geq1}$ be in $\mathrm{L}_{\mathbb P}(\pi)$ such that $\mathrm{d}(A_k,A)\to 0$ as $k\to \infty$.  Then, for every $n\geq 1$, 
 \[
 \|\phi_n^{A_{k}} - \phi_n^{A}\|_\infty
\longrightarrow  0 \quad \text{as $k\to \infty$}.
 \]
\end{lem}

\begin{proof}
We prove first the lemma for $n=1$:

\begin{claim} \label{claim:phi}
    It holds that $\|\phi_{A_k}-\phi_A\|_\infty\to0$ as $k\to \infty$.
\end{claim}
\begin{proof} Set
$E_k:=\|A_{k,\omega}-A_\omega\|_\infty$, 
        $R_k:=\|(A_{k,\omega})^{-1}-(A_\omega)^{-1}\|_\infty$ and $U:=\|(A_\omega)^{-1}\|_\infty$. 
By assumption, $E_k\to0$  and $R_k\to0$ in $L^1(\mathbb P)$.   For \(z=(x,\hat v)\), since
$\|A_\omega(x)v\|
        \geq
        \|A_\omega(x)^{-1}\|^{-1}
        \geq
        U^{-1}$, we have
\[
\begin{aligned}
\log\frac{\|A_{k,\omega}(x)v\|}{\|A_\omega(x)v\|}
&\leq
\log\left(
1+U E_k
\right).
\end{aligned}
\]
Reversing the roles of \(A_k\) and \(A\), and using
$\|(A_{k,\omega})^{-1}\|_\infty
        \leq
        U+R_k$, $(1+(a+b)c)\leq (1+ac)(1+bc)$  and $\log (1+ab) \leq a + b$ for $a,b,c\geq 0$,
we get
\[
\begin{aligned}
\left|
\log\|A_{k,\omega}(x)v\|
-
\log\|A_\omega(x)v\|
\right| &\leq
\log\left(
1+U E_k
\right)
+
\log\left(
1+(U+R_k)E_k
\right) \\ 
&\leq 2 \log\left(
1+U E_k
\right) + \log \left(1+ R_k E_k\right) \leq 2 \log\left(
1+U E_k
\right) + R_k+ E_k.
\end{aligned}
\]
Since $R_k\to 0$ and $E_k\to 0$ in $L^1(\mathbb{P})$, to prove the claim is enough to show that $\log(1+UE_k) \to 0$ in $L^1(\mathbb{P})$. 
To see this, fix \(\varepsilon>0\) and using that $\log (1+ab)\leq \log(1+a)+\log(1+b)$ and $\log(1+b)\leq b$ for every $a,b\geq 0$, 
\begin{equation} \label{eq:onde-ver}
    \begin{aligned}
            \int\log(1+UE_k)\,d\mathbb P &=
         \int_{\{E_k\leq\varepsilon\}}\log(1+UE_k)\,d\mathbb P
        +
        \int_{\{E_k>\varepsilon\}}\log(1+UE_k)\,d\mathbb P \\
        &\leq\int\log(1+\varepsilon  U)\,d\mathbb P
        +
        \int_{\{E_k>\varepsilon\}} \log(1+U) \, d\mathbb{P} +  \int E_k\,d\mathbb P.
\end{aligned}
\end{equation}
Since $\log^+ U$ is $\mathbb{P}$-integrable and $\log (1+a U) \leq \log^+ U + \log(1+a)$ for any $a\geq 0$, by dominated convergence,
\[
        \int\log(1+\varepsilon U)\,d\mathbb P
        \longrightarrow0
        \quad\text{as }\varepsilon\downarrow0.
\]
Also  since $E_k\to 0$ in $L^1(\mathbb{P})$, hence also in probability and thus, the integrability of $\log(1+U)$ and the absolute continuity of the Lebesgue integral gives 
$$
\int_{\{E_k>\varepsilon\}} \log(1+U) \, d\mathbb{P}  \longrightarrow 0
        \quad\text{as } k \to \infty.$$
Hence, taking \(\limsup_{k\to\infty}\) in~\eqref{eq:onde-ver},
we obtain
\begin{equation} \label{eq:final}
       \limsup_{k\to\infty}
   \int\log(1+UE_k)\,d\mathbb P
   \leq
   \int\log(1+\varepsilon U)\,d\mathbb P.
\end{equation}
Finally, $0\leq\log(1+\varepsilon U)\leq\log(1+U)$ for $0<\varepsilon\leq1$, and \(\log(1+U)\in L^1(\mathbb P)\). Therefore, by dominated
convergence,
$\int\log(1+\varepsilon U)\,d\mathbb P
   \longrightarrow0$ as $\varepsilon\downarrow0$. We conclude that
\[
   \int\log(1+UE_k)\,d\mathbb P\longrightarrow0,
\]
which, in view of~\eqref{eq:final} proves the claim.
\end{proof}

Let $\hat P\psi(z)
        :=
        \int \psi(\hat A_\omega z)\,d\mathbb P$,  $z\in\hat{\mathcal E}$,  be the annealed Markov operator associated to the projective cocycle
\(\hat A\). Similarly, write \(\hat P_k\) the annealed Markov operator associated with $\hat A_k$. We have that 
\[
        \phi_n^A 
        =
        \sum_{i=0}^{n-1}\hat P^i\phi_A,
        \qquad
        \phi_n^{A_k}
        =
        \sum_{i=0}^{n-1}\hat P_k^i\phi_{A_k}.
\]
Therefore
\begin{equation} \label{eq:conclusion}
\begin{aligned}
\|\phi_n^{A_k}-\phi_n^{A}\|_\infty
&\le
\sum_{i=0}^{n-1}
\|\hat P_k^i\phi_{A_k}-\hat P^i\phi_A\|_\infty        
\le
\sum_{i=0}^{n-1}
\|\hat P_k^i(\phi_{A_k}-\phi_A)\|_\infty
+
\sum_{i=0}^{n-1}
\|(\hat P_k^i-\hat P^i)\phi_A\|_\infty .
\end{aligned}
\end{equation}
Since \(\hat P_k\) is a Markov operator, it is a contraction and  hence $\|\hat P_k^i(\phi_{A_k}-\phi_A)\|_\infty
        \le
        \|\phi_{A_k}-\phi_A\|_\infty\to 0$ by Claim~\ref{claim:phi}. Thus the first sum in the last inequality in~\eqref{eq:conclusion} tends to zero as $k\to\infty$.   
On the other hand, the telescoping identity for powers gives
\[
        \hat P_k^i-\hat P_A^i
        =
        \sum_{\ell=0}^{i-1}
        \hat P_k^{i-1-\ell}
        (\hat P_k-\hat P)
        \hat P^\ell .
\]
and hence, using again that \(\hat P_k\) is a contraction,
\begin{equation} \label{eq:telescopi}
\begin{aligned}
\|(\hat P_k^i-\hat P^i)\phi_A\|_\infty
&\le
\sum_{\ell=0}^{i-1}
\|(\hat P_k-\hat P)\hat P^\ell\phi_A\|_\infty .
\end{aligned}
\end{equation}
\begin{claim} \label{claim:P}
It holds that $\|(\hat P_k-\hat P)\psi\|_\infty\to0$ as $k\to\infty$ for every $\psi\in C(\hat{\mathcal E})$. 
\end{claim}
\begin{proof} Since \(\psi\) is continuous on the compact space
\((\hat{\mathcal E},d_{{\mathcal E}})\), it is uniformly continuous. Let
$\omega_\psi(t)
        :=
        \sup\{|\psi(z)-\psi(z')|:
        d_{\mathcal E}(z,z')\le t\}$ be its modulus of continuity. Then
\[
\begin{aligned}
\|\hat P_k\psi-\hat P\psi\|_\infty
&\le
\int
\sup_{z\in\hat{\mathcal E}}
|\psi(\hat A_{k,\omega}z)-\psi(\hat A_\omega z)|
\,d\mathbb P(\omega)        \le
\int \omega_\psi(\Delta_k(\omega))\,d\mathbb P
\end{aligned}
\]
where 
\[
        \Delta_k(\omega)
        :=
        \sup_{z\in\hat{\mathcal E}}
        d_{\mathcal E}
        \bigl(\hat A_{k,\omega}z,\hat A_\omega z\bigr).
\]
Set again $E_k:=\|A_{k,\omega}-A_\omega\|_\infty$, $U:=\|(A_\omega)^{-1}\|_\infty$ and $\rho_k:=d_{C^0}(A^\pi_{k,\omega},A^\pi_\omega)$. Given \(z=(x,\hat v)\in\hat{\mathcal E}\), write
$y_k:=A^\pi_{k,\omega}(x)$, $y:=A^\pi_\omega(x)$,
and $u_k:=A_{k,\omega}(x)v$, $u:=A_\omega(x)v$. Hence, there is a constant  \(C_0>0\),
\begin{align*}
d_{\hat{\mathcal E}}
\bigl(\hat A_{k,\omega}z,\hat A_\omega z\bigr)
&\le
C_0\bigl(d(y_k,y)+\theta(\hat{u}_k,\hat{u})\bigr)  
\le C_0\bigl(\rho_k+\big\| 
\frac{u_k}{\|u_k\|}
-
\frac{u}{\|u\|}
\big\|\bigr) 
\\ 
&\le
2C_0(\rho_k+\frac{\|u_k-u\|}{\|u\|}) \leq 2C_0 (\rho_k + UE_k).
\end{align*}
Thus, $\Delta_k \le C(\rho_k + U E_k)$.
    Since $E_k \to 0$ and $\rho_k \to 0$ in $L^1(\mathbb P)$, they converge to 0 in probability. Because $A_\omega$ is an isomorphism for $\mathbb P$-a.e. $\omega$, we have $U(\omega) < \infty$ a.s. Consequently, $U E_k \to 0$ in probability, which implies $\Delta_k \to 0$ in probability. Since   $\omega_\psi(t) \to 0$ as $t \to 0$, it follows that $\omega_\psi(\Delta_k) \to 0$ in probability. Hence, since  $\omega_\psi(\Delta_k) \le 2\|\psi\|_\infty$, bounded convergence in probability implies that $\int \omega_\psi(\Delta_k) \, d\mathbb P \to 0$. This completes the proof of Claim~\ref{claim:P}.
\end{proof}

Since $\phi_A$ is continuous on the compact set $\hat{\mathcal E}$ and $\hat P$ is a Markov operator associated with continuous maps, the function $\hat P^\ell \phi_A$ is continuous for every $\ell \ge 0$. By Claim~\ref{claim:P}, for each $0 \le \ell \le i-1$, we have $\|(\hat P_k - \hat P)\hat P^\ell \phi_A\|_\infty \to 0$ as $k \to \infty$. Summing over the finitely many terms in \eqref{eq:telescopi}, it follows that $\|(\hat P_k^i - \hat P^i)\phi_A\|_\infty \to 0$ as $k \to \infty$. Finally, from \eqref{eq:conclusion}, both sums converge to $0$ as $k \to \infty$. Therefore, $\|\phi_n^{A_k} - \phi_n^A\|_\infty \to 0$, completing the proof of the lemma. 
\end{proof}

Let $\mu$ be $A^\pi$-invariant stationary measure of a linear bundle isomorphisms $A\in \mathrm{L}_{\mathbb{P}}(\pi)$. Define
$$
\lambda_-(\mu,A)=\inf_{\hat \mu \in \mathcal{I}_\mu(\hat A)} \int \phi_A \, d\hat\mu \quad \text{and} \quad \lambda_-(A)=\inf_{\nu \in \mathcal{I}(A^\pi)} \lambda_-(\nu,A).
$$
We have that 
\[
\lambda_m(\mu,A)
 :=
 \lim_{n\to\infty}\frac1n
 \int\log\mathfrak m(A_\omega^n(x))\,d\mathbb P\,d\mu
 \leq \lambda_-(\mu,A)\leq\lambda_1(\mu,A).
\]

\begin{definition} A measure $\mu$ is \emph{quasi-irreducible} if $\lambda_-(\mu,A)=\lambda_1(\mu,A)$.  
    We also say that $A$ is 
    \begin{itemize}
        \item \emph{quasi-irreducible} if $\lambda_-(\mu,A)=\lambda_1(\mu,A)$ for every $\mu\in \mathcal{I}(A^\pi)$;
        \item \emph{maximal quasi-irreducible} if $\lambda_-(A)=\lambda_1(A)$;
        \item \emph{expanding} if $\lambda_-(A)>0$.
    \end{itemize}   
    The dual notions are defined by applying the corresponding definitions to the inverse-dual random map $A^{-*}:\Omega\times\mathcal E^*\to\mathcal E^*$ covering $A^\pi:\Omega \times X \to X$. Thus, $\mu$ is \emph{coquasi-irreducible} for $A$ if it is quasi-irreducible for $A^{-*}$ and $A$ is \emph{coquasi-irreducible}, \emph{maximal coquasi-irreducible} and \emph{coexpanding} if $A^{-*}$ is quasi-irreducible, maximal quasi-irreducible and expanding respectively.
\end{definition}

\begin{prop} \label{prop:mainthmD} The functions 
\begin{itemize}
    \item 
$A\in \mathrm{L}_{\mathbb{P}}(\pi,\mu) \mapsto  \lambda_1(\mu,A)$ and $A\in \mathrm{L}_{\mathbb{P}}(\pi) \mapsto  \lambda_1(A)$ are upper semicontinuous. 
\item $A\in \mathrm{L}_{\mathbb{P}}(\pi,\mu) \mapsto  \lambda_-(\mu,A)$ and $A\in \mathrm{L}_{\mathbb{P}}(\pi) \mapsto  \lambda_-(A)$ are lower semicontinuous. 
\end{itemize}
In particular, given $A_0\in \mathrm{L}_{\mathbb{P}}(\pi)$,
\begin{itemize}
    \item if $\mu$ is quasi-irreducible for $A_0$, then $A \in \mathrm{L}_{\mathbb{P}}(\pi,\mu)\mapsto \lambda_1(\mu,A)$ is continuous at $A_0$.  
    \item if $A_0$ is maximal quasi-irreducible, then $A \in \mathrm{L}_{\mathbb{P}}(\pi)\mapsto \lambda_1(A)$ is continuous at $A_0$.
\end{itemize}
\end{prop} 
\begin{proof}
Now define
\begin{align*}
a_n(\mu,A)
 :=
 \frac1n\int_X
 \min_{\hat v\in P(\mathcal E_x)}
 \phi_n^A(x,\hat v)\,d\mu(x) \quad \text{and} \quad 
b_n(\mu,A)
 :=
 \frac1n\int_X
 \max_{\hat v\in P(\mathcal E_x)}
 \phi_n^A(x,\hat v)\,d\mu(x).
\end{align*}
By Lemma~\ref{lem:uniform-conv-phi}, for every fixed \(n\ge1\),
the maps \(A\mapsto a_n(\mu,A)\) and \(A\mapsto b_n(\mu,A)\)
are continuous. According to~\cite[Cor.~V and~\S D.4]{BN26},
\[
\lambda_1(\mu,A)=\inf_{n\ge1}b_n(\mu,A),
\qquad
\lambda_-(\mu,A)=\sup_{n\ge1}a_n(\mu,A).
\]
Therefore \(A\mapsto\lambda_1(\mu,A)\) is upper semicontinuous, because
it is the infimum of continuous functions, while
\(A\mapsto\lambda_-(\mu,A)\) is lower semicontinuous, because it is the
supremum of continuous functions. In particular, if $\mu$ is quasi-irreducible for $A_0\in \mathrm{L}_{\mathbb{P}}(\pi)$, since $\lambda_1(\mu,A_0)=\lambda_-(\mu,A_0)$, it follows that $A\mapsto \lambda_1(\mu,A)$ is continuous at $A_0$.

Similarly, define
$$
\bar{a}_n(A)=\min_{z\in \hat{\mathcal{E}}} \phi^A_n(z) \quad \text{and} \quad 
\bar{b}_n(A)=\max_{z\in \hat{\mathcal{E}}} \phi^A_n(z).
$$
Again, by Lemma~\ref{lem:uniform-conv-phi}, $\bar{a}_n$ and $\bar{b}_n$ are continuous functions for every $n\ge 1$. Thus, since according to~\cite[Theorem C]{BM}
\[
\lambda_-(A)=\sup_{n\ge1}\bar a_n(A),
\qquad
\lambda_1(A)=\inf_{n\ge1}\bar b_n(A).
\]
we get that $A\mapsto\lambda_-(A)$ and $A\mapsto\lambda_1(A)$ are lower and upper semicontinuous functions respectively. In particular, when $A_0\in \mathrm{L}_{\mathbb{P}}(\pi)$ is a maximal quasi-irreducible, we have $\lambda_1(A_0)=\lambda_-(A_0)$ and thus, $A\mapsto\lambda_1(A)$ is continuous at $A_0$. 
\end{proof}

\begin{prop} \label{prop:thm-prop} Let $A\in \mathrm{L}_{\mathbb{P}}(\pi)$. 
The functions 
$$\mu \in \mathcal{I}(A^\pi) \mapsto \lambda_-(\mu,A) \quad  \text{and} \quad \mu \in \mathcal{I}(A^\pi) \mapsto \lambda_1(\mu,A)$$ are lower and upper semicontinuous respectively. In particular, if $A$ is quasi-irreducible, then $\mu \in \mathcal{I}(A^\pi) \mapsto\lambda_1(\mu,A)$ is continuous.    
\end{prop}
\begin{proof} The semicontinuity follows immediately by~\cite[Claim~2.2.7]{BM} and its dual result. The continuity also is immediate because quasi-irreducibility implies $\lambda_-(\mu,A)=\lambda_1(\mu,A)$. 
\end{proof}

\begin{proof}[Proof of Theorem~\ref{main:thmD} and~\ref{main:thmDbis}] These results are consequence of Proposition~\ref{prop:mainthmD} applied to $A=Df$ and $A=(Df^{-1})^*$ observing that $\lambda_1(\mu,Df)=\lambda_1(\mu)$ and $\lambda_1(\mu,(Df^{-1})^*)=-\lambda_m(\mu)$.    
\end{proof}

\begin{proof}[Proof of Theorem~\ref{prop:continuity-measure-lyapunov} and~\ref{prop:continuity-measure-lyapunov-bis}]
 Similarly as before, the continuity of $\mu\mapsto \lambda_1(\mu)$ and $\mu\mapsto \lambda_m(\mu)$ at quasi-irreducible and coquasi-irreducible random maps follows from Proposition~\ref{prop:thm-prop}.   
\end{proof}

\section{Proof of propositions and corollaries}
\label{sec:proof-corollaries}

In this section, we prove the propositions and corollaries stated in the
introduction. We first prove them for a general random linear bundle isomorphism,
and then recover the corresponding statements for the tangent and cotangent
cocycles.

We first record the linear-bundle version of Corollary~\ref{maincor:expanding-quasi}. 

\begin{prop}
\label{prop:abstract-expanding-quasi}
Let  $A:\Omega \times \mathcal E \to \mathcal E$ be a random linear bundle isomorphism covering a continuous random map $f:\Omega \times X \to X$ of a compact metric space $X$ satisfying that $\log^+\|A_\omega\|_\infty$ is $\mathbb{P}$-integrable.   Then the following statements hold.
\begin{enumerate}
\item \(A\) is expanding if and only if every non-trivial \(A\)-invariant
subbundle \(\mathcal L\subset\mathcal E\) over an
\(f\)-stationary measure \(\mu\) satisfies
$\lambda_1(\mu,A|_{\mathcal L})>0$.
\item If \(A\) is expanding and $\lambda_2(\mu,A)\le0$
for every $\mu\in\mathcal{I}(f)$, then \(A\) is quasi-irreducible.
\item If \(A\) is quasi-irreducible, then
\[
        A\text{ is expanding}
        \quad\Longleftrightarrow\quad
        \lambda_1(\mu,A)>0
        \text{ for every }\mu\in\mathcal{I}(f).
\]
\item If $\lambda_2(\mu,A)\le0<\lambda_1(\mu,A)$ for every $\mu\in\mathcal{I}(f)$, then
\[
        A\text{ is expanding}
        \quad\Longleftrightarrow\quad
        A\text{ is quasi-irreducible}.
\]
\end{enumerate}
\end{prop}

\begin{proof}
Assume that \(A\) is expanding and let \(\mathcal L\) be a non-trivial
\(A\)-invariant bundle over an \(f\)-stationary measure \(\mu\). Hence $0<\lambda_-(A)\leq \lambda_1(\mu,A|_{\mathcal L})$;
see Lemma~\ref{lem:bundle-generated-by-nu} for more details.
Reciprocally, since
$\lambda_-(A)=\int \phi_A\,d\nu$ for some \(\nu\in\mathcal{I}(\hat A)\) ergodic, and according to
Lemma~\ref{lem:bundle-generated-by-nu},
$\int \phi_A\,d\nu
        =
        \lambda_1(\mu,A|_{\mathcal L})$,
where \(\mathcal L\) is the \(A\)-invariant bundle spanned by \(\nu\), we get
$ \lambda_-(A)=\lambda_1(\mu,A|_{\mathcal L})$.
Hence, the assumption that
$\lambda_1(\mu,A|_{\mathcal L})>0$ for every non-trivial \(A\)-invariant bundle \(\mathcal L\) implies that
\(\lambda_-(A)>0\). Thus \(A\) is expanding. This proves~(i).

To prove (ii), observe that if a measure \(\mu\) has a non-trivial equator \(\mathcal L\), then
$\lambda_1(\mu,A|_{\mathcal L})
        \leq
        \lambda_2(\mu,A)$.
Hence if \(A\) is expanding, then
$ 0<\lambda_1(\mu,A|_{\mathcal L})$ 
by (i), and thus \(\lambda_2(\mu,A)>0\). Since (ii) assumes that
$\lambda_2(\mu,A)\le0$
for every \(\mu\in\mathcal{I}(f)\), \(A\) needs to be
quasi-irreducible.

We prove (iii). Assume that \(A\) is quasi-irreducible and expanding. Since
quasi-irreducibility gives
$ \lambda_-(\mu,A)=\lambda_1(\mu,A)$ for every $\mu\in\mathcal{I}(f)$, we get
$\lambda_1(\mu,A)         =
        \lambda_-(\mu,A)
        \ge
        \lambda_-(A)>0$
for every \(\mu\in\mathcal{I}(f)\). Conversely, assume that \(A\) is quasi-irreducible and $\lambda_1(\mu,A)>0$
for every $\mu\in\mathcal{I}(f)$.
Then $\lambda_-(\mu,A)=\lambda_1(\mu,A)>0$
for every $\mu\in\mathcal{I}(f)$.
The map 
$ \mu\mapsto \lambda_-(\mu,A)$
is lower semicontinuous on the compact set \(\mathcal{I}(f)\). Hence, it
attains its minimum. Since it is positive at every point, this minimum is
positive. Therefore
$\lambda_-(A)
        =
        \inf_{\mu\in\mathcal{I}(f)}
        \lambda_-(\mu,A)
        >0$. 
Thus \(A\) is expanding. This proves~(iii).

Finally, (iv) follows immediately from (ii) and (iii).
\end{proof}

Applying Proposition~\ref{prop:abstract-expanding-quasi} to the linear random morphisms $A^{-*}$, whose Lyapunov exponents are $-\lambda_m(\mu,A)\ge\cdots\ge -\lambda_1(\mu,A)$,  we obtain the following generalization of  Corollary~\ref{cor:coexpanding-equator-positive}.

\begin{cor} \label{cor:abstract-coexpanding-coquasi}Let  $A:\Omega \times \mathcal E \to \mathcal E$ be a random linear bundle isomorphism covering a continuous random map $f:\Omega \times X \to X$ of a compact metric space $X$ satisfying that $\log^+\|(A_\omega)^{\pm 1}\|_\infty$ is $\mathbb{P}$-integrable.  Then the following statements hold. 
\begin{enumerate} \item \(A\) is coexpanding if and only if every non-trivial \(A^{-*}\)-invariant subbundle \(\mathcal L\subset\mathcal E^*\) over an \(f\)-stationary measure \(\mu\) satisfies $\lambda_1(\mu,A^{-*}|_{\mathcal L})>0$. In particular, $\lambda_m(\mu,A)<0$.  
\item If \(A\) is coexpanding and $\lambda_{m-1}(\mu,A)\ge0$ for every $\mu\in\mathcal{I}(f)$,  then \(A\) is coquasi-irreducible. 
\item If \(A\) is coquasi-irreducible, then \[ A\text{ is coexpanding} \quad\Longleftrightarrow\quad \lambda_m(\mu,A)<0 \text{ for every }\mu\in\mathcal{I}(f). \] 
\item If $\lambda_m(\mu,A)<0\le\lambda_{m-1}(\mu,A)$ for every $\mu\in\mathcal{I}(f)$,  then 
\[ A\text{ is coexpanding} \quad\Longleftrightarrow\quad A\text{ is coquasi-irreducible}. 
\] 
\end{enumerate} 
\end{cor}

Before proving the corresponding generalization of  Proposition~\ref{prop:equator-coequator-extremal-dimension} we need the following lemma.

\begin{lemma}
\label{lem:annihilator-coequator-abstract}
Let  $A:\Omega \times \mathcal E \to \mathcal E$ be a random linear bundle isomorphism covering a continuous random map $f:\Omega \times X \to X$ of a compact metric space $X$ satisfying that $\log^+\|(A_\omega)^{\pm 1}\|_\infty$ is $\mathbb{P}$-integrable.  
Consider an ergodic measure $\mu\in\mathcal{I}(f)$. Suppose that \(\mu\) has a
coequator \(\mathcal F\subset\mathcal E^*\). Consider the annihilator bundle
$\mathcal K=\mathcal F^\perp$
given by
\[
        \mathcal K_x
        :=
        \mathcal F_x^\perp
        =
        \{v\in\mathcal E_x:\xi(v)=0
        \text{ for every }\xi\in\mathcal F_x\}.
\]
Then \(\mathcal K\) is an \(A\)-invariant bundle and
$\lambda_m(\mu,A|_{\mathcal K})
        =
        \lambda_m(\mu,A)$.
\end{lemma}

\begin{proof}
If \(v\in\mathcal K_x\) and
\(\eta\in\mathcal F_{f_\omega(x)}\), then, by invariance of
\(\mathcal F\), there exists \(\xi\in\mathcal F_x\) such that
\[
        \eta=(A_\omega(x)^{-1})^*\xi
        =
        \xi\circ A_\omega(x)^{-1}.
\]
Therefore $\eta(A_\omega(x)v)=\xi(v)=0$,
and hence
$ A_\omega(x)\mathcal K_x
        \subset
        \mathcal K_{f_\omega(x)}$.
Since \(A_\omega(x)\) is invertible for
\(\mathbb P\times\mu\)-a.e.~\((\omega,x)\), this inclusion is an equality.
This shows that \(\mathcal K\) is \(A\)-invariant.

Let $\mathcal Q:=\mathcal E/\mathcal K$
be the quotient bundle. Since
$\mathcal F=\mathcal K^\perp$, we have the natural identification
$
        \mathcal F_x
        \equiv
        \mathcal Q_x^*
        =
        (\mathcal E_x/\mathcal K_x)^*$.
Under this identification, the random inverse-dual cocycle $A^{-*}$ covering $f$  restricted to
\(\mathcal F\) is the inverse-dual of the quotient cocycle on \(\mathcal Q\).
Consequently,
\[
        \lambda_1(\mu,A^{-*}|_{\mathcal F})
        =
        -\lambda_{\min}(\mu,A_{\mathcal Q}),
\]
where \(\lambda_{\min}(\mu,A_{\mathcal Q})\) denotes the bottom Lyapunov
exponent of the quotient cocycle. Since \(\mathcal F\) is a coequator,
$   \lambda_1(\mu,A^{-*}|_{\mathcal F})
        <
        -\lambda_m(\mu,A)$,
and thus
$        \lambda_{\min}(\mu,A_{\mathcal Q})
        >
        \lambda_m(\mu,A)$.

Hence, taking into account that
$
        \lambda_m(\mu,A)
        =
        \min\{
        \lambda_m(\mu,A|_{\mathcal K}),
        \lambda_{\min}(\mu,A_{\mathcal Q})
        \}$,
cf.~\cite[Lem.~3.6]{FK83}, we get that
$
        \lambda_m(\mu,A)=\lambda_m(\mu,A|_{\mathcal K})$,
as we wanted to prove.
\end{proof}

We now prove the linear-bundle version of Proposition~\ref{prop:equator-coequator-extremal-dimension}.

\begin{prop}
\label{prop:abstract-equator-coequator-extremal-dimension}
Let  $A:\Omega \times \mathcal E \to \mathcal E$ be a random linear bundle isomorphism of rank \(m\) covering a continuous random map $f:\Omega \times X \to X$ of a compact metric space $X$ satisfying that $\log^+\|(A_\omega)^{\pm 1}\|_\infty$ is $\mathbb{P}$-integrable.    Let \(\mu\in\mathcal{I}(f)\) be ergodic. Then,
\begin{enumerate}
\item If $\lambda_1(\mu,A)=\lambda_m(\mu,A)$, 
then \(\mu\) is quasi-irreducible and coquasi-irreducible.
\item If \(\mu\) is quasi-irreducible for \(A\), then \(\mu\) has no coequator
of dimension \(m-1\).
\item If \(\mu\) is coquasi-irreducible for \(A\), then \(\mu\) has no equator
of dimension \(m-1\).
\item If \(m=2\), then \(\mu\) is quasi-irreducible if and only if \(\mu\) is
coquasi-irreducible.
\end{enumerate}
\end{prop}

\begin{proof}
Since $\lambda_m(\mu,A)\leq \lambda_-(\mu,A)\leq \lambda_1(\mu,A)$
and
$    -\lambda_1(\mu,A)
        =
        \tilde\lambda_m(\mu,A)
        \leq
        \tilde\lambda_-(\mu,A)
        \leq
        \tilde\lambda_1(\mu,A)
        =
        -\lambda_m(\mu,A),
$
if
$\lambda_1(\mu,A)=\lambda_m(\mu,A)$,
then immediately
$  \lambda_-(\mu,A)=\lambda_1(\mu,A)$
and
$  \tilde\lambda_-(\mu,A)=-\lambda_m(\mu,A)$.
Thus \(\mu\) is quasi-irreducible and coquasi-irreducible. This proves~(i).

We prove (ii). First of all, we assume that
$\lambda_1(\mu,A)>\lambda_m(\mu,A)$,
because otherwise item~(i) already forbids the existence of a coequator. Suppose,
by contradiction, that \(\mu\) is quasi-irreducible and has a coequator
\(\mathcal F\) with $\dim\mathcal F=m-1$.
Let $\mathcal K=\mathcal F^\perp\subset\mathcal E$
be the annihilator of \(\mathcal F\). Since
$    \dim\mathcal K
        =
        \dim\mathcal E-\dim\mathcal F
        =
        1$, 
we have
\[
\lambda_1(\mu,A|_{\mathcal K})
:=
\lim_{n\to\infty}
\frac1n
\int
\log\|A_\omega^n(x)|_{\mathcal K_x}\|
\,d\mathbb P\,d\mu
=
\lim_{n\to\infty}
\frac1n
\int
\log\mathfrak m(A_\omega^n(x)|_{\mathcal K_x})
\,d\mathbb P\,d\mu
=:
\lambda_m(\mu,A|_{\mathcal K}).
\]
Moreover, by Lemma~\ref{lem:annihilator-coequator-abstract},
\(\mathcal K\) is an \(A\)-invariant subbundle and $   \lambda_m(\mu,A|_{\mathcal K})
        =
        \lambda_m(\mu,A)$.
Since $\lambda_1(\mu,A|_{\mathcal K})
        =
        \lambda_m(\mu,A)
        <
        \lambda_1(\mu,A)$, $\mu$ has a non-trivial equator, contradicting the
quasi-irreducibility. This proves~(ii).

We prove (iii) by applying item (ii) to the inverse-dual cocycle $A^{-*}$ covering $f$. Indeed, quasi-irreducibility of $A^{-*}$  is precisely coquasi-irreducibility of \(A\).
Moreover, after the canonical identification \(\mathcal E^{**}\simeq\mathcal E\),
the inverse-dual cocycle of \(A^{-*}\) is again \(A\). Hence, the coequators of the
cotangent cocycle \(A^*\) are exactly the equators of the original cocycle \(A\). 
Therefore, if \(\mu\) is coquasi-irreducible for \(A\), then \(\mu\) is
quasi-irreducible for \(A^{-*}\). Applying item~(ii) to \(A^{-*}\), we conclude that
\(A^{-*}\) has no coequator of dimension \(m-1\). Equivalently, the original cocycle
\(A\) has no equator of dimension \(m-1\). This proves~(iii).

Finally, (iv) follows from (ii) and (iii), since for \(m=2\) non-trivial equators
and coequators need to have dimension one.
\end{proof}

We now prove the abstract linear-bundle version of Corollary~\ref{cor:interplay-expanding-coexpanding}.

\begin{cor}
\label{cor:abstract-interplay-expanding-coexpanding}
Let  $A:\Omega \times \mathcal E \to \mathcal E$ be a random linear bundle isomorphism of rank \(m\) covering a continuous random map $f:\Omega \times X \to X$ of a compact metric space $X$ satisfying that $\log^+\|(A_\omega)^{\pm 1}\|_\infty$ is $\mathbb{P}$-integrable.   Then the following statements hold.
\begin{enumerate}
\item If \(A\) is expanding, has no non-trivial coequators of dimension less than
\(m-1\), and $\lambda_2(\mu,A)<0$
   for every $\mu\in\mathcal{I}(f)$,
then \(A\) is coquasi-irreducible and coexpanding.

\item If \(A\) is coexpanding, has no non-trivial equators of dimension less than
\(m-1\), and  $\lambda_{m-1}(\mu,A)>0$
        for every $\mu\in\mathcal{I}(f)$,
then \(A\) is quasi-irreducible and expanding.
\item If \(m=2\) and
$ \lambda_2(\mu,A)<0<\lambda_1(\mu,A)$ for every $\mu\in\mathcal{I}(f)$,
then
\[
        \text{expanding}
        \quad\Longleftrightarrow\quad
        \text{coexpanding}
        \quad\Longleftrightarrow\quad
        \text{quasi-irreducible}
        \quad\Longleftrightarrow\quad
        \text{coquasi-irreducible}.
\]
\end{enumerate}
\end{cor}

\begin{proof}
We prove (i). Since \(A\) is expanding and
$ \lambda_2(\mu,A)<0$ for every $\mu\in\mathcal{I}(f)$,
by Proposition~\ref{prop:abstract-expanding-quasi}(ii), \(A\) is
quasi-irreducible. Hence, by
Proposition~\ref{prop:abstract-equator-coequator-extremal-dimension}(ii), an
ergodic stationary measure has no coequator of dimension \(m-1\). Since
ergodic \(f\)-stationary measures also have no non-trivial coequators of
dimension less than \(m-1\), we conclude that \(A\) is coquasi-irreducible.
Finally, since $\lambda_m(\mu,A)\leq\lambda_2(\mu,A)<0$
for every $\mu\in\mathcal{I}(f)$,
and \(A\) is coquasi-irreducible, Corollary~\ref{cor:abstract-coexpanding-coquasi}(iii) implies that \(A\) is
coexpanding.

We prove (ii). The proof is symmetric. Since \(A\) is coexpanding and
$\lambda_{m-1}(\mu,A)>0$ for every $\mu\in\mathcal{I}(f)$,
Corollary~\ref{cor:abstract-coexpanding-coquasi}(ii) implies that \(A\) is
coquasi-irreducible. Hence, by
Proposition~\ref{prop:abstract-equator-coequator-extremal-dimension}(iii), an
ergodic stationary measure has no equator of dimension \(m-1\). Since ergodic
\(f\)-stationary measures also have no non-trivial equators of dimension
less than \(m-1\), we conclude that \(A\) is quasi-irreducible. Finally, since $\lambda_1(\mu,A)\geq\lambda_{m-1}(\mu,A)>0$
for every $\mu\in\mathcal{I}(f)$,
and \(A\) is quasi-irreducible,
Proposition~\ref{prop:abstract-expanding-quasi}(iii) implies that \(A\) is
expanding.

Finally, (iii) follows from (i) and (ii), by observing that if \(m=2\)
non-trivial equators and coequators need to have dimension one.
\end{proof}

\subsection{Proof of the intermediate dimensional statements}

We finally explain why statements announced in
\S\ref{subsec:higher-dimensional-expansion} follow from the abstract
linear-bundle results proved above. Let \(A\in \mathrm{L}_{\mathbb{P}}(\pi)\) be a random linear bundle
isomorphism, and fix \(1\le k\le m-1\) where $m=\dim \mathcal E$. 
For the lower variational quantity one must distinguish the
Grassmannian \(G_k(\mathcal E)\) from the whole projective bundle
\(\mathbb P(\wedge^k\mathcal E)\). The required equality follows from
the finite-dimensional identity
\[
   \min_{[v]\in\mathbb P(\wedge^k V)}
   \log\frac{\|\wedge^kLv\|}{\|v\|}
   =
   \log m(\wedge^kL)
   =
   \min_{W\in G_k(V)}
   \log\operatorname{Jac}_k(L|_W),
\]
valid for every invertible linear map \(L:V\to V'\).
Together with the lower Grassmannian variational formula
\cite[Cor.~D.5]{BN26}, this gives
\[
   \lambda_-(\mu,\wedge^k A)
   =
   \lambda_-^{[k]}(\mu,A).
\]
For the upper quantity, the Pl\"ucker embedding and the corresponding
maximal variational formula give
\[
   \lambda_+(\mu,\wedge^k A)
   =
   \lambda_+^{[k]}(\mu,A)
   =
   \Sigma_k(\mu,A).
\]
Moreover, the first two Lyapunov exponents of the exterior-power
cocycle satisfy
\[
   \lambda_1(\mu,\wedge^kA)
   =
   \Sigma_k(\mu,A),
   \qquad
   \lambda_2(\mu,\wedge^kA)
   =
   \Gamma_k(\mu,A).
\]
Consequently, \(A\) is \(k\)-expanding if and only if
\(\wedge^kA\) is expanding, and \(\mu\) is \(k\)-quasi-irreducible
for \(A\) if and only if it is quasi-irreducible for
\(\wedge^kA\).


In view of this, applying Proposition~\ref{prop:mainthmD},~\ref{prop:abstract-expanding-quasi},~\ref{prop:abstract-equator-coequator-extremal-dimension} and Corollary~\ref{cor:abstract-coexpanding-coquasi},~\ref{cor:abstract-interplay-expanding-coexpanding} to \(\wedge^kA\), we get
the corresponding consequences announced in~\S\ref{subsec:higher-dimensional-expansion}. The only point one needs to be careful with is applying Proposition~\ref{prop:mainthmD} to exterior powers while keeping the topology of the original cocycle. To achieve this, one needs the following elementary observation:

\begin{lemma}
\label{lem:exterior-power-continuity}
Let $A,(B_n)_{n\geq 1}$ be in $\mathrm{L}_{\mathbb{P}}(\pi)$ such that ${\rm d}(B_n,A)\to 0$. Assume that there is~\(C>0\)~such~that
\[
        \|A_\omega\|_\infty,\;
        \|B_{n,\omega}\|_\infty,\;
        \|(A_\omega)^{-1}\|_\infty,\;
        \|(B_{n,\omega})^{-1}\|_\infty
        \le C \quad \text{for \(\mathbb P\)-a.e.~\(\omega\) and every \(n\ge1\).}
\]
Then, ${\rm d}(\wedge^k B_n,\wedge^k A)\to 0$  for every $k=1,\dots,m$.
\end{lemma}

\begin{proof}
We use the elementary estimate
\[
        \|\wedge^k L-\wedge^k M\|
        \le
        \sum_{j=0}^{k-1}
        \|L\|^j\|M\|^{k-1-j}\|L-M\|
\]
for linear maps \(L,M\) between Euclidean spaces. In particular, if
\(\|L\|,\|M\|\le C\), then
\[
        \|\wedge^k L-\wedge^k M\|
        \le
        kC^{k-1}\|L-M\|.
\]
Applying this with $L=B_{n,\omega}(x)$,  $M=A_\omega(x)$,
and taking the supremum over \(x\in X\), we get
\[
        \|\wedge^k B_{n,\omega}-\wedge^k A_\omega\|_\infty
        \le
        kC^{k-1}
        \|B_{n,\omega}-A_\omega\|_\infty .
\]
Moreover, since
$ (\wedge^k B_{n,\omega})^{-1}
        =
        \wedge^k((B_{n,\omega})^{-1})$ and
$ (\wedge^k A_\omega)^{-1}
        =
        \wedge^k((A_\omega)^{-1})$, the same estimate gives
\[
        \|(\wedge^k B_{n,\omega})^{-1}
        -
        (\wedge^k A_\omega)^{-1}\|_\infty
        \le
        kC^{k-1}
        \|(B_{n,\omega})^{-1}-(A_\omega)^{-1}\|_\infty .
\]
The exterior-power cocycle has the same base map as the original cocycle. Hence
\begin{align*}
        {\rm d}(\wedge^k B_n,\wedge^k A)
        &\le
        \int \left( d_{C^0}\big((B_n^\pi)_\omega,(f)_\omega\big)
         +
        kC^{k-1}
        \|B_{n,\omega}-A_\omega\|_\infty
        +
        kC^{k-1}
        \|(B_{n,\omega})^{-1}-(A_\omega)^{-1}\|_\infty
        \right) \,d\mathbb P.
\end{align*}
The right-hand side converges to zero because \({\rm d}(B_n,A)\to0\). Thus
${\rm d}(\wedge^k B_n,\wedge^k A)\to0$.
\end{proof}

\bibliographystyle{alpha3}
\bibliography{biblio}

\end{document}